\documentclass{article}[12pt]
\usepackage{geometry}
\usepackage[final]{graphicx}
\usepackage{bm}
\usepackage{defs,fonts}
\usepackage{epstopdf}
\usepackage{mathtools}
\usepackage{amsmath}
\usepackage{amsthm}
\usepackage{tabularx}
\usepackage{subcaption}
\usepackage[width=\textwidth]{caption}
\usepackage{multirow}
\usepackage{multicol}
\usepackage{booktabs}
\usepackage{hhline} 
\usepackage{tikz}
\usetikzlibrary{positioning}
\usepackage{latexsym,amssymb,amsmath,amsfonts,amsthm}
\usepackage{float}   
\usepackage[inline]{enumitem}
\usepackage{wrapfig}
\usepackage{hyperref}
\usepackage{todonotes}
\usepackage[titletoc,toc,title]{appendix}
\usepackage{indentfirst}
\usepackage{verbatim}
\usepackage[backend=biber, style=numeric-comp, url = false, isbn=false, doi=false]{biblatex}{\bibliography{RTE_bib}}
\usepackage[normalem]{ulem}
\usepackage{bbm}
\usepackage{subcaption} 
\usepackage{booktabs} 
\usepackage{algorithm}
\usepackage{algpseudocode}
\usepackage{tabularx}
\usepackage{makecell} 
\usepackage{geometry}
\usepackage{booktabs}
\usepackage{adjustbox}
\usepackage{pgfplots}
\usepackage{array}
\usepackage{chngcntr}
\usepackage{xcolor} 
\usepackage{xspace}

\newcommand{\intD}{\;\mathrm{d}}
\newcommand{\rom}[1]{\uppercase\expandafter{\romannumeral #1\relax}}

\newtheorem{theorem}{Theorem}
\newtheorem{proposition}{Proposition}
\newtheorem{lemma}{Lemma}

\newtheorem{definition}{Definition}
\newtheorem{remark}{Remark}

\newcounter{subeqn}[equation]

\newcommand{\subr}[1]{(\subref*{#1})}

\pgfplotsset{compat=1.18}
\def\thefootnote{\arabic{footnote}}
\usepackage{xcolor}
\definecolor{matplotlibgreen}{RGB}{44, 160, 44}  
\renewcommand{\thefootnote}{\fnsymbol{footnote}}

\newlength{\subfigwidth}
\title{Data-Driven Filtering of the Spherical Harmonics Method}
\author{
   Benjamin Plumridge\footnotemark[2] ,
  Cory Hauck\footnotemark[2] \footnotemark[3]  , \small{AND}
  \large Steffen Schotth\"ofer\footnotemark[3] 
}
\date{}

\begin{document}

\maketitle
\footnotetext[2]{Department of Mathematics, University of Tennessee, Knoxville}
\footnotetext[3]{Computer Science and Mathematics Division, Oak Ridge National Laboratory}

\renewcommand{\thefootnote}{\arabic{footnote}}

\begin{abstract}
We investigate a data-driven approach for tuning the filtered spherical harmonics method (\fpn) when solving the radiation transport equation (RTE). The \fpn method extends the classical spherical harmonics approach (\pn) by introducing regularization through a filter operator, which mitigates spurious oscillations caused by Gibbs' phenomenon. This filter includes a tunable parameter, the {filter strength}, that controls the degree of smoothing applied to the solution. However, selecting an optimal filter strength is nontrivial, often requiring inaccessible information such as the true or a high-order reference solution. To overcome this limitation, we model the filter strength as a neural network whose inputs include local state variables and material cross-sections. The optimal filter strength is formulated as the solution to a PDE-constrained optimization problem, and the neural network is trained using a discretize-then-optimize formulation in PyTorch. We evaluate the learned filter strength across a suite of test problems and compare the results to those from a simple, but tunable constant filter strength. In all cases, the neural-network-driven filter substantially improves the accuracy of the \pn approximation.  For 1-D test cases, the constant filter outperforms the neural-network filter in some cases, but in 2-D problems, the neural-network filter generally performs better.

\end{abstract}

\section*{Funding}
This work is sponsored by the National Science Foundation under DMS-1913277 and by the Applied Mathematics Program at the Office of Advanced Scientific
Computing Research, U.S. Department of Energy.  Work was performed at the Oak Ridge National Laboratory, which is
managed by UT-Battelle, LLC under Contract No. DE-AC05-00OR22725 with the U.S. Department of Energy. The
United States Government retains and the publisher, by accepting the article for publication, acknowledges that the
United States Government retains a non-exclusive, paid-up, irrevocable, world-wide license to publish or reproduce the
published form of this manuscript, or allow others to do so, for United States Government purposes. The Department of
Energy will provide public access to these results of federally sponsored research in accordance with the DOE Public
Access Plan (http://energy.gov/downloads/doe-public-access-plan).

\section{Introduction}
The radiation transport equation (RTE) is an integro-partial differential equation (PDE) used to model the movement of particles that interact with a background material medium via emission, absorption, and scattering processes. 
The RTE has applications in high-energy density physics and nuclear engineering, including core-collapse supernovae \cite{mezzacappa2020physical}, thermal radiative transfer \cite{howell2020thermal}, general radiation hydrodynamics \cite{mihalas1984radiation,pomraning2005equations}, and nuclear reactor theory \cite{lewis1984computational, bell1970nuclear}.
One of the challenges in approximating solutions of the RTE is dealing with its high-dimensional phase space, which can include as many as six dimensions:  three spatial variables, two angular variables, and one energy variable.  If transients are desired, then the RTE will also depend on time. Due to limitations in computational resources, coarse approximations of the RTE are often necessary. In many cases, angular discretization is the primary target for such approximations.

One well known approach for discretizing the RTE in angle is with a  truncated spherical harmonics expansion \cite{brunner2005two,case1967linear}, where the expansion coefficients are functions of space and time that satisfy a system of hyperbolic balance laws.  These equations are referred to as the \pn equations, where the subscript $N$ refers to the polynomial degree at which the expansion is truncated.  
For smooth solutions of the RTE, the \pn equations converge spectrally to the RTE as $N \to \infty$ \cite{Frank2016}; this is one the key advantages of the method. Another advantage is that the operators in the \pn equations are all sparse and therefore cheap to evaluate.  In particular, because the spherical harmonics are eigenfunctions of the scattering operator in the RTE, the corresponding operator in the \pn equations appears as a diagonal matrix \cite{lewis1984computational}.  
Meanwhile, the advection operator of the RTE is converted into a set of banded matrices whose exact form follows from known recursion relations for the spherical harmonics \cite{brunner2005two}.
  
Despite the efficiency in approximating smooth solutions to the RTE, the \pn equations do have drawbacks.
Like all spectral methods \cite{Gottlieb2007}, they exhibit oscillatory
behavior  when approximating non-smooth functions, particularly near discontinuities and large gradients \cite{Brunner2002}.
In some cases, these oscillations can lead to solutions with negative particle densities \cite{Brunner2002} unless some type of numerical limiter is applied \cite{hauck2010positive}.  Because scattering induces smooth solutions, these issues are more prevalent when the amount of scattering is relatively small. 

The filtered spherical harmonics (\fpn) equations are a modification of the standard \pn equations that induce regularization via an artificial scattering operator.   The key advantage of the \fpn approach is that it effectively damps numerical oscillations at a computational cost that remains comparable to that of the original \pn method.  
Filtering the spherical harmonic expansion was first proposed in \cite{Hauck2010}, where the filter was implemented as an post-processing step after every time step of the \pn system.  Later in  
\cite{Radice2013}, the \fpn equations were derived as the continuous limit as $\Delta t \rightarrow 0$ of the filtering scheme presented in \cite{Hauck2010}.
In \cite{Frank2016}, the convergence properties of the \fpn equations were established.  
Roughly speaking, it was shown that the \fpn equations converge to the exact RTE solution at a rate that is equal to the filter order for smooth solutions and (almost) equal to the convergence order of the standard \pn equations for non-smooth solutions. 
Although filtering alone does not ensure positivity, additional limiting strategies can be employed \cite{laiu2016positive, laiu2019positivity}.

For a given filter order, the amount of regularization in the \fpn equations is determined by the filter strength, a single parameter that does not affect the convergence behavior when $N$ is large but can have dramatic effects when $N$ is small \cite{laboure2016implicit}, particularly in regimes with relatively low scattering.  
As with Tikhonov regularization when solving noisy inverse problems, tuning the filter strength is a matter of balancing stability with consistency \cite{Oleary1993}:  while increasing the filter strength improves the stability of the method, it also introduces additional consistency error in the angular discretization.  
It remains an open question how to best tune the filter, but currently the most viable approach seems to be a combination of physics-based scaling arguments and experimental trial and error using relatively cheap, coarse-mesh simulations \cite{Hauck2010,laboure2016implicit}.

The goal of the current work is to design a data-driven model for the filter strength in the \fpn equations. 
Specifically, we design a parametrized neural network that takes local material properties and the solution state as inputs and outputs a suitable filter strength. 
As a consequence, the filter strength will vary both spatially and temporally.
We train this data-driven filter via PDE-constrained optimization.  
We use a discretize-then-optimize approach using a second-order, finite-volume scheme for space-time discretization.
The training loss is the $L^2$ discrepancy of the particle concentration for the data-driven \fpn solution and a high order \pn reference simulation.
To generate training data, we sample various test-problems with different source terms, initial conditions and material properties.

The learned models parameters are used to predict the optimal filter strength in several test problems and compared to the original \pn solution.  
We test the neural network filter over a set of problems that challenge the approximation with discontinuities and steep gradients in the material properties and the solution state.   
Our results show substantial error reductions.
Because the raw \pn solutions are often so poor that almost any filter yields improvements, we introduce a stronger baseline: a constant filter strength, realized as a degenerate neural network with only a bias term. In many instances, even this constant filter markedly improves the accuracy and quality of the \pn approximation. In several one-dimensional cases, it outperforms the neural-network filter, especially in regions of material vacuum.

The remainder of the paper is organized as follows. In Section \ref{sec:background}, we describe the mono-energetic RTE, as well as the \pn and \fpn equations. 
In Section \ref{sec:dd_FPN}, we formulate the details of the data driven filter.
In Section \ref{sec:numerical results}, we show numerical results.   Some of the details of the \pn discretization are provided in the Appendix.

\section{Background} \label{sec:background}

In this section, we briefly recall the radiation transport equation (RTE) for a collection of mono-energetic particles.  We also present the \pn and \fpn equations.

\subsection{Radiation transport equation} 
The radiation transport equation (RTE) governs the evolution of a kinetic distribution $\psi$ of particles moving through a background medium.  Let $X \subset \bbR^3$ have a piecewise smooth boundary $\p \Omega$ and let $\mathbb{S}^2$ be the unit sphere in $\bbR^3$.
In this paper, $\psi$ is defined such that $\psi(\bm{x},\bm{\Omega},t)$ gives the density of particles, with respect to the measure $d \bm{\Omega} dx$, located at position $\bm{x} = (x,y,z)^\top  \in X$, moving in direction $\bm{\Omega} \in \bbS^2$ at time $t \geq 0$.  For particles with unit speed, the RTE takes the form
\begin{subequations}
\begin{align}
\p_t \psi(\bm{x},\bm{\Omega},t) + \bm{\Omega} \cdot \grad_{\bm{x}} \psi(\bm{x},\bm{\Omega},t)   + \sig{a}(\bm{x}) \psi(\bm{x},\bm{\Omega},t) +  \sig{s}(\bm{x}) \cQ \psi(\bm{x},\bm{\Omega},t)    &= S(\bm{x},\bm{\Omega},t),\\
\psi(\bm{x},\bm{\Omega},t = 0) &= \psi_0(\bm{x},\bm{\Omega}),\\
\psi(\bm{x},\bm{\Omega},t) &= \psi_{\text{b}}(\bm{x},\bm{\Omega}), & \forall (\bm{x},\bm{\Omega}) \in \Gamma^-,
\end{align}
\label{eq:RTE}
\end{subequations}
where $S= S(\bm{x},\bm{\Omega},t)$ is a non-negative source and the absorption cross-section $\sig{a}$ and scattering cross-section $\sig{s}$ are non-negative functions of $\bm{x}$.  
The scattering operator $\cQ$, which describes the change in direction of the particles due to collisions with the background medium, is an integral operator in $\bm{\Omega}$:
\begin{align}
 \cQ \psi(\bm{x},\bm{\Omega},t) =  \psi(\bm{x},\bm{\Omega},t) - \int_{\mathbb{S}^2} g(\bm{\Omega} \cdot \bm{\Omega}') \psi(\bm{x},\bm{\Omega}',t) d \bm{\Omega}',
\end{align}
where the scattering kernel $g: [-1,1] \rightarrow \bbR$ is normalized so that $\int_{\mathbb{S}^2} g(\bm{\Omega} \cdot \bm{\Omega}') d \bm{\Omega}' = 1$.  
In the numerical results of this work, we focus on isotropic scattering kernels, in which case $g= \frac{1}{4\pi}$.
However, we maintain $\cQ$ as a general operator in the derivations and proofs below. 
Inflow data is provided on the set
\begin{align}
\Gamma^{-} = \left\{(\bm{x},\bm{\Omega}) \in \p X \times \mathbb{S}^2 : \pm \bfn(\bm{x}) \cdot \bm{\Omega} \ge 0\right\},
\end{align}
where $\bfn(\bm{x})$ is the outward unit normal to $X$ at $\bm{x} \in \p X$, is the inflow boundary in phase space.

The existence of a unique semi-group solution for \eqref{eq:RTE} is given in  \cite[Chapter XXI]{Dautray2000}, under standard assumptions on the data.  If all sources, initial conditions, and boundary conditions are non-negative, then the semi-group solution remains non-negative for all time.

\subsection{\texorpdfstring{\pn}{P\_N} equations}
The \pn equations are based on a spectral Galerkin approximation of the RTE in the angular variable in terms of a truncated spherical harmonic expansion. Let $\theta \in [0,\pi]$ and $\varphi \in [0,2\pi)$ be the polar and azimuthal angles associated to $\bm{\Omega}$.  Then
\begin{align}   
\label{eq:omega}
\bm{\Omega} = \begin{pmatrix} \Omega^{(1)}\\ {\Omega}^{(2)}\\ {\Omega}^{(3)} \end{pmatrix} = \begin{pmatrix} \sin \theta \cos\varphi \\ \sin \theta  \sin \varphi \\ \cos \theta \end{pmatrix} 
= \begin{pmatrix} \sqrt{1-\mu^2} \cos\varphi \\  \sqrt{1-\mu^2} \sin\varphi \\ \mu \end{pmatrix}
\end{align}
where $\mu = \cos \theta$. 
Spherical harmonics \cite{Atkinson2012}, which form an orthonormal basis for $L^2(\bbS^2)$, are defined for for each $\ell \in \{ 0,1,2,...\}$ and  $m  \in \{-\ell,...,\ell\}$ by 
\begin{align}
Y_\ell^m(\bm{\Omega}) = (-1)^m \sqrt{\frac{(2\ell +1)}{4 \pi}\frac{(\ell-m)!}{(\ell+m)!}} P_\ell^m(\mu) e^{im\varphi},
\end{align}
where $P_\ell^m$ represents the Legendre function \cite{Abramowitz1972} of degree $\ell$ and order $m$, and $\mu$ and $\varphi$ are related to $\bm{\Omega}$ via \eqref{eq:omega}.
The motivation for using this basis is that spherical harmonics form a complete set of eigenfunctions for $Q$ \cite{case1967linear}.  More specifically, 
\begin{align}
\label{eq:eigen-Q}
\cQ Y_\ell^m = (1- g_\ell)Y_\ell^m,
\end{align}
where
\begin{equation}
    g_\ell = 2 \pi \int_{-1}^1 P_\ell(\eta) g(\eta)d\eta 
\end{equation}
and $P_\ell$ is the degree $\ell$ Legendre polynomial, normalized such that $ \int_{-1}^1 |P_\ell(\eta) |^2d\eta = \frac{2}{2\ell+1}$.  With this normalization, $P_0(\eta)=1$ and $|P_\ell(\eta)| < 1$ for $\ell >1$ and almost every $\eta \in [-1,1]$ \cite{Atkinson2012}; it follows that $g_0=1$ and $g_\ell < 1$ for $\ell >1$.

For convenience, we use real-valued, orthonormal spherical harmonics
\begin{align}
r_\ell^m =
\begin{cases}
\frac{1}{\sqrt{2}} \left( Y_\ell^{-m} + (-1)^m Y_\ell^m \right) & \text{if } m > 0, \\
Y_\ell^0 & \text{if } m = 0, \\
\frac{i}{\sqrt{2}} \left( Y_\ell^{|m|} - (-1)^{|m|} Y_\ell^{-|m|} \right) & \text{if } m < 0.
\end{cases}
\end{align}
Let $\bfr_\ell: \bm{\Omega} \longrightarrow \bbR^{2\ell+1}$ be a vector-valued function containing all of the spherical harmonic functions of degree $\ell$, i.e.,
\begin{equation}
    \bfr_0 = r_0^0, \quad
    \bfr_1 = [r_1^{-1},  r_1^0,  r_1^1]^\top, \quad  
    \bfr_\ell = [r_{\ell}^{-\ell}, \dots, r_{\ell}^0  \dots  r_{\ell}^{\ell}] ^\top;
\end{equation}
and let $\bfr: \bm{\Omega} \longrightarrow \bbR^{n}$, where $n = (N+1)^2$, be a single vector-valued function containing all the spherical harmonic functions up to degree $N$; that is,  $\bfr = [\bfr_0^\top \dots \bfr_N^\top]^\top$.

The \pn approximation is given by
\begin{equation}
\label{eq:pn-expansion}
    \psipn(\bm{x}, \bm{\Omega}, t) =  \upn^\top(\bm{x},t) \bfr(\bm{\Omega}) = \sum_{\ell =0}^N \sum_{m = -\ell}^\ell [u_{\pn}]_\ell^m(\bm{x},t) r_\ell^m(\bm{\Omega}),
\end{equation}
where
\footnote{Boundary conditions for the \pn equations can be formulated in several ways.  We leave them unspecified until needed in the numerical results.}
 \begin{subequations}
\begin{align}
\vint{\bfr \p_t \psipn} + \vint{ \bfr \bm{\Omega} \cdot \grad_{\bm{x}} \psipn} + \sig{a} \vint{ \bfr \psipn} + \sig{s}  \vint{ \bfr \cQ \psipn }  &= \vint{\bfr S },\\
\vint{ \bfr  \psipn|_{t =0})  } &= \vint{\bfr \psi_0 }.
\end{align}
\label{eq:PN_weak}
\end{subequations}
Substituting the expansion in \eqref{eq:pn-expansion} into \eqref{eq:PN_weak} yields the \pn equations for the coefficient vector $\upn : X \times [0,\infty) \to \bbR^{n}$:
\begin{subequations}
\begin{align}
\p_t \upn(\bm{x},t) + \bf\bfA \cdot \grad_{\bm{x}} \upn(\bm{x},t) + \sig{a}(\bm{x})\upn(\bm{x},t) + \sig{s}(\bm{x}) \bfG \upn(\bm{x},t)  &= \bfs(\bm{x},t),\\
\upn(\bm{x},0) &= \vint{\bfr \psi_0}  .
\end{align}\label{eq:PN}
\end{subequations}
Here $\bfA$ is a three-tensor such that
\begin{align}
      \bfA \cdot \grad_{\bm{x}} \upn 
    = \bfA^{(1)} \p_x  \upn 
    + \bfA^{(2)} \p_y  \upn
    + \bfA^{(3)} \p_z  \upn
    , \quad \text{with} \quad  \bfA^{(m)} = \vint{\bm{\Omega}^{(m)} \bfr \bfr^\top},
\end{align}    
The matrices $\bfA^{(m)} \in \bbR^{n \times n}$ are sparse and symmetric; they can be computed with a recurrence relationship for spherical harmonics, which for real-valued harmonics, can be found in the appendix of \cite{Frank2016}; see also Appendix \ref{sec:flux_matrices}.  
Meanwhile, due to \eqref{eq:eigen-Q}, the matrix $\bfG = \vint{ \bfr Q \bfr^\top } \in \bbR^{n \times n}$ is diagonal with components $G_{(\ell,m),(\ell,m)} = (1 - g_\ell)$.  
The source vector $\bfs = \vint{ \bfr S }$ and initial condition $\vint{ \bfr \psi_0 }$ are known.
We write \eqref{eq:PN} in the compact form
\begin{subequations}
\begin{align}
\p_t \upn(\bm{x},t) + \cT \upn(\bm{x},t)  &= \bfs(\bm{x},t),\\
\upn(\bm{x},0) &= \vint{\bfr \psi_0} ,
\end{align}
\label{eq:PN_abstract}
\end{subequations}
where $\cT \bfv = \bf\bfA \cdot \grad_{\bm{x}} \bfv  + \sig{a}\bfv  + \sig{s}\bfG\bfv $.

\subsection{\texorpdfstring{\fpn}{FP\_N} equations}
The \fpn equations are constructed by adding artificial scattering terms to the \pn equations.  These terms appear in the form of a diagonal operator, whose components are expressed in terms of a filter function \cite{Gottlieb2007}.  The precise definition of this function can vary in the literature; for our purposes, a filter function  of order $\alpha \in \mathbb{N}$ is a real-valued, non-negative, and monotonically decreasing function $f \in C^\alpha([0,1])$ that satisfies the following properties:
\begin{equation}
    (i)~f(0) = 1,
    \qquad
    (ii)~f^{(\ell)}(0) = 0,~\forall \ell \le \alpha - 1, 
    \qquad
    (iii)~f^{(\alpha)}(0) \ne 0.
\end{equation}
In this paper, we will use an exponential filter of order four:
\begin{align}
f(s) = \exp(-s^\alpha), \qquad \alpha=4.
\end{align}

With a filter function in hand, the \fpn equations can be written as
\begin{subequations}
\begin{align}
\p_t \ufpn(\bm{x},t) + \cT \ufpn(\bm{x},t) + \sig{f}(\bm{x},t) \bfF \ufpn(\bm{x},t) &= \bfs(\bm{x},t),\\
\ufpn(\bm{x},0) &= \vint{\bfr \psi_0}   ,
\end{align}\label{eq:FPN}
\end{subequations}
where $\sig{f}(\bm{x},t) \ge 0$ is the filter strength and $\bfF$ is a diagonal matrix with components 
\begin{equation}
    \bfF_{(\ell,m),(\ell,m)} = -\log \left(f \left( \frac{\ell}{N+1}\right) \right).
\end{equation}
In particular, the component $\bfF_{(0,0),(0,0)}=0$, which implies that the filter does not affect the total particle concentration.
 The user must select the filter function and strength appropriately for a given setting.   In previous work, $\sig{f}$ was chosen to be constant.

Given a solution of \eqref{eq:FPN}, the \fpn approximation of $\psi$ is $\psifpn  =\bfr^\top \ufpn$.  The convergence of the filtered approximation to the true solution depends on the order of the filter function $f$, independent of the filter strength.   Rigorous results \cite{Frank2016} suggest that the order of the filter should be chosen based on the expected regularity of the RTE solution.
Roughly speaking, when the filter order $\alpha$ is greater than the angular regularity of $\psi$, the convergence is still spectral but the rate drops by a small constant factor (less than one).  When $\alpha$ is less than the angular regularity of $\psi$, then the convergence rate is given by $\alpha$.  

Numerical experiments in \cite{Hauck2010, laboure2016implicit} show that the potential benefits of filtering are greater for small values of $N$.  Moreover, practical constraints typically limit realistic transport simulations with the \pn equations to $N \leq \cO(10)$ .  Unfortunately, the convergence results in \cite{Frank2016} provide little guidance about how to select the filter strength $\sig{f}$ in order to ensure quality results when $N$ is small.  These facts motivate our data-driven approach.

\section{Data-driven \texorpdfstring{\fpn}{FP\_N} Equations}
\label{sec:dd_FPN}

We construct a data-driven filter for the $\pn$ equations that depends locally in space and time on $\bfs$ and on the solution itself.  Specifically, we let
\begin{equation}
\label{eq:parameterized-filter}
    \sig{f}(\bm{x},t) = \Sig{f}\big(\bm{\xi}(\ufpn(\bm{x},t)),\bfs(\bm{x},t);\bm{w}_{\fpn} \big),
\end{equation}
where the ansatz
$
    \Sig{f} \colon (\bm{\xi},\bfs;\bm{w})\mapsto \bbR^{\geq 0}
$
takes a feature vector $\bm{\xi}$, the source $\bfs$, and a parameter vector $\bm{w}$ and outputs the corresponding filter strength. 
The parameter vector $\bm{w}_{\fpn}$ in  \eqref{eq:parameterized-filter} is determined by a training procedure.  The ansatz $\Sig{f}$ can take on a variety of functional forms.  Here we use neural networks with weights $\bm{w}$ that are trained based on the $L^2$ error in the particle concentration $\phi = \vint{\psi}$ at some final time $\tf$.  For the input features in $\bm{\xi}$, we choose each of the terms in the operator $\cT$ in the \pn equations \eqref{eq:PN}.  Specifically, for any sufficiently smooth function $\bfv : X \to \bbR^n$, $\bm{\xi}(\bfv): X \to \bbR^{n\times 3}$ is given by
\begin{equation}\label{eq_features}
    \bm{\xi}(\bfv) 
    = [\bfA \cdot \grad_{\bm{x}} \bfv
    \quad
    \sig{t} \bfv
    \quad  
    \sig{s} (\bfG-\bfI)\bfv ].
\end{equation}

Let the random variable $\gamma$ parameterize the distribution of sources, initial conditions, and cross-sections that are used for training.  Let $\phi^\gamma = \vint{\psi^\gamma}$, where  $\vint{\psi^\gamma}$ solves \eqref{eq:RTE} with source $S^\gamma$, initial condition $\psi_0^\gamma$, and cross-sections $\sig{a}^\gamma$ and $\sig{s}^\gamma$.  We assume that for all parameterizations $\gamma$, there exists a unique solution of the \fpn equations
\begin{subequations}
\begin{align}
    \p_t \ufpn^\gamma 
    + \cT \ufpn^\gamma 
    + \Sig{f} (\bm{\xi},\vint{ \bfr S^\gamma} ;\bm{w}) \bfF \ufpn^\gamma &= \vint{{\bfr} {S}^\gamma} \\ 
        \ufpn^\gamma(\bm{x},0) &= \vint{\bfr \psi_0^\gamma} 
\end{align}
\end{subequations}
and define the map
\begin{equation}
    \Phi^\gamma:  \bm{w} \mapsto \phi^\gamma_{\fpn}(\cdot,\tf),
\end{equation}
which takes as input the parameter $\bm{w}$ and returns the particle concentration  $\phi^\gamma_{\fpn} = \vint{\psifpn^\gamma} = \sqrt{4\pi} [u_{\fpn}^\gamma]_0^0$ at the time $\tf$.
\footnote{The difference in the concentration $\phi_{\fpn}$ and the moment $[u_{\pn}]_0^0$ is due to the normalization of $r_0^0 = 1/\sqrt{4 \pi}$.}    

We formulate the training as a PDE-constrained minimization problem 
\begin{equation}
\label{eq:cont-objective}
  \min_{\bm{w}}\, J(\bm{w}),
  \quad
  \text{where $J(\bm{w}) =  \mathbb{E}_{\gamma} [j( \Phi^\gamma(\bm{w}); \phi^{\gamma}(\cdot,\tf)  )]$}
\end{equation}
and for any functions $a,b \in L^2(X)$
\begin{align}
j(a;b)
= \frac{1}{4 \pi}
\left\| a - b \right\|^2_{L^2(X)}.
\end{align} 
In practice the reference solution $\phi^{\gamma}$ is itself approximated in angle with a \pn solution with $N$ very large.

\subsection{Discretization of the PDE Constraint}
In practice, the \fpn equations must be discretized in space and time. To this end, we employ a second-order finite-volume scheme in space and Heun's method in time.  This discretization yields an approximate solution $\bar{\bfv}^\gamma$ that is constant in each spatial mesh cell and defined on a finite set of temporal points $ \{0 = t_0, t_1 \ddd t_K = \tf \}$ such that $\bar{\bfv}^\gamma(\bm{x},t_k) \approx {\bfu}^\gamma_{\fpn}(\bm{x},t_k)$ for all $\bm{x} \in X$.  

Let $\vec{\bfv}^{\gamma,k}$ be an array that contains the cell averages of $\bar{\bfv}^\gamma(\cdot,t_k)$ on each spatial mesh cell. 
For each $k$, $\vec{\bfv}^{\gamma,k}$ evolves according to an update rule of the form
\begin{equation}
\label{eq:time_stepper}
    \vec{\bfv}^{\gamma,k+1}(\bm{w}) = \bfH^{\gamma,k}(\vec{\bfv}^{\gamma,k} (\bm{w}); \bm{w}),
    \quad k \in \{0, \dots, K -1\},
\end{equation}
where $\vec{\bfv}^{\gamma,0}$ is computed from the cell-wise averages $\ufpn^\gamma(\cdot,0)$. 
The details of the operator $H^{\gamma,k}$ come from the the finite-volume discretization that is described in Appendix \ref{sec:space-time-disc}.
Define the discrete solution operators $\bar{\Phi}^\gamma$ and $\vec{\Phi}^\gamma$ by 
\begin{align}\label{eq:disc_solution_operator}
   \bar{\Phi}^\gamma: \bm{w} \mapsto \bar{\phi}_{\fpn}^\gamma(\cdot,\tf) 
   \qquand
      \vec{\Phi}^\gamma: \bm{w} \mapsto \vec{\phi}_{\fpn}^{\gamma,K},
\end{align}
where, for each $k \in \{0\ddd K\}$, $\bar{\phi}_{\fpn}^\gamma(\cdot,t_k)$ is the piecewise-constant (in space) approximation of $\phi_{\fpn}^\gamma(\cdot,t_k)$ and $\vec{\phi}_{\fpn}^{\gamma,K}$ is the corresponding vector of cell-averages computed using the same finite volume method in \eqref{eq:time_stepper}.  The operators in \eqref{eq:disc_solution_operator} are related by a linear, invertible reconstruction operator $\cE$; that is, $\bar{\Phi}^\gamma(\bm{w}) = \cE \vec{\Phi}^\gamma (\bm{w})$.

From \eqref{eq:cont-objective}, we obtain the discrete objective function
\begin{align}\label{eq:closed_objective_disc}
    \bar{J}(\bm{w}) 
    =  \mathbb{E}_{\gamma} [j( \bar{\Phi}^\gamma(\bm{w});\bar{\phi}^\gamma(\cdot,\tf))]
    = \mathbb{E}_{\gamma} [j( \cE \vec{\Phi}^\gamma(\bm{w}); \cE \vec{\phi}^{\gamma,K})].
\end{align}
where the piece-wise constant function $\bar{\phi}^\gamma(\cdot,t_k)$ with cell averages in $\vec{\phi}^{\gamma,t_k}$ approximates $\phi^\gamma(\cdot,k)$ with the same finite volume method.
We expect that $\bar{J}^\gamma$ is differentiable with respect to $\bm{w}$ and thus we choose a gradient descent-based optimizer to find optimal parameters $\bm{w}$. 

\begin{remark}
    The approach here is called \textit{discretize, then optimize} in PDE-constrainted optimization literature, see, e.g. \cite{Hinze2008}. Alternatively, one can formally derive the Fr\'echet derivative of $J$ with respect to $\bm{w}$ which yields an adjoint set of equations. The adjoint \fpn equations then need to be discretized to facilitate the optimization. This approach is called \textit{optimize, then discretize}. In general the approaches are not equivalent.
\end{remark}

\subsection{Gradient-based Optimization}

At each iteration, we approximate the gradient of $\bar{J}$ using a batch $B$ of training data with realizations $\gamma_b \in B$:
\footnote{The details for this sampling will be discussed in the numerical results.}  
\begin{align}
\label{eq:closed_objective_disc_batch}
    \bar{J}(\bm{w}) 
    \approx \bar{J}^B(\bm{w})
    = \frac{1}{|B|}\sum_{\gamma_b \in B} j( \bar{\Phi}^\gamma(\bm{w});\bar{\phi}^\gamma(\cdot,\tf)).
\end{align}
The gradient descent procedure is then 
\begin{align}\label{eq_sgd}
    \bm{w}^{\ell+1}=\bm{w}^{\ell} -\lambda \nabla_{\bm{w}} \bar{J}^B(\bm{w}^\ell),
\end{align}
where $\lambda>0$ is the learning rate and the initial weights $ \bm{w}^{\ell=0}$ are normally distributed. 

In practice, the gradient in \eqref{eq_sgd} is computed efficiently via automatic differentiation.   For the sake of completeness, we describe the underlying components of the computation.  First, $\nabla_{\bm{w}} \bar{J}^B =  [\frac{ \p\bar{J}^B}{\p  \bm{w}}]^\top$ is computed via the chain rule
\begin{align}
\label{dJbar-dw}
\frac{ \p\bar{J}^B}{\p  \bm{w}} (\bm{w}) 
= \frac{1}{|B|}\sum_{\gamma_b \in B} 
\cD j
(\bar{\Phi}^{\gamma_b}(\bm{w});\bar{\phi}^\gamma(\cdot,\tf))
\left[
 \frac{\p \bar{\Phi}^{\gamma_b}}{\p \bm{w}} (\bm{w})\right]
 = \frac{1}{|B|}\sum_{\gamma_b \in B} 
\cD j
(\bar{\Phi}^{\gamma_b}(\bm{w});\bar{\phi}^\gamma(\cdot,\tf))
\left[\cE
 \frac{\p \vec{\Phi}^{\gamma_b}}{\p \bm{w}} (\bm{w})\right],
\end{align}
where $\cD j$ is the Fr\'echet derivative of $j$, which in \eqref{dJbar-dw} is evaluated at $\bar{\Phi}^{\gamma_b} (\bm{w})$ and acts on $\cE
 \frac{\p \vec{\Phi}^{\gamma_b}}{\p \bm{w}} (\bm{w})$. For any functions $a,b,c \in L^2(X)$ 
\begin{equation}
    Dj(a;b)[c]
    = \lim_{\veps \to 0^+}
    \frac{d}{d\veps}
    j(a+\veps c;b)
    = \frac{1}{2\pi} \int_X  \left(a(\bm{x}) - b(\bm{x}) \right) c(\bm{x}) d \bm{x}.
\end{equation}
The remaining derivative needed in \eqref{dJbar-dw} is
\begin{equation}
    \frac{\p \vec{\Phi}^{\gamma_b}}{\p \bm{w}} 
= \sqrt{4 \pi} \left[\frac{\p \vec{\bfv}^{\gamma_b,K}}{\p \bm{w}}\right]_0^0,
\end{equation} 
which is computed recursively in $k$, starting with $k=K$ and using \eqref{eq:time_stepper}, as described at the end of Appendix \ref{sec:space-time-disc}

\subsection{Feature Selection and Rotational Invariance}

In this subsection, we introduce the notion of rotational invariance for the RTE and $\pn$ equations and establish sufficient conditions on $\Sig{f}$ for the \fpn equations to preserve this property.

\begin{definition}\label{def_rotation}
Given any orthogonal matrix $O\in\bbR^{3\times 3}$, the operator $\cR_O$ induced by $O$ is given by
\begin{equation}
    (\cR_O f)(\bm{x},\bm{\Omega}) = f(O x,O \bm{\Omega}), 
    \quad \text{for any $f: X \times \bbS^2 \to \bbR$}. 
\end{equation}
The operator $\cW_O$ acting on the corresponding moment $\bfv =\vint{\bfr f}$ is given by
\begin{equation}
\label{eq:W_O}
    (\cW_O \bfv)(\bm{x}) = \vint{{\bfr} \,(\cR_O f)(\bm{x},\cdot)}\:.
\end{equation}
\end{definition}
\begin{remark}
    In what follows, we may apply $\cR_O$ to a function of $\bm{x}$ only, in which case $\cR_O f(\bm{x}) = f(O\bm{x})$.  For vector-valued functions, $\cR_O$ is applied component-wise.
\end{remark}

\begin{lemma}
\label{lem_rot_u}
Let ${O}\in\bbR^{3\times 3}$ be an orthogonal matrix.  Then for each $\ell \in \{0 \ddd N\}$, there exists an orthogonal matrix $\bfZ^\ell_O\in\bbR^{(2\ell+1)\times (2\ell+1)}$ such that 
\begin{equation}
\label{eq:Z_O}
    \bfr_\ell (\bm{\Omega}) 
    =  \bfZ^\ell_O\bfr_\ell(O {\bm{\Omega}}) .
\end{equation}
Thus $\bfr (\bm{\Omega}) 
    =  \bfZ_O\bfr(O {\bm{\Omega}})$,
where $\bfZ = \operatorname{diag}(\bfZ_O^0 \ddd \bfZ_O^N) \in \bbR^{n\times n}$.
Moreover,
\begin{align}
	(\cW_O \bfv)(\bm{x})
    = (\bfZ_O \bfv)(O \bm{x}).
    \label{eq:rot_operator}
\end{align}

\end{lemma}
\begin{proof}
We follow the proof from \cite[Lemma 1]{garrett2016eigenstructure}. 
For each polynomial order $\ell$, $\operatorname{span}\{r_{\ell}^{-\ell} \ddd r_{\ell}^0  \ddd  r_{\ell}^{\ell}\}$ is  invariant under orthogonal transformations  in ${\bm{\Omega}}$ \cite[Theorem 1.1.7]{dai2013approximation}. Thus there exists a matrix $\bfZ_O^\ell \in\bbR^{(2\ell+1)\times(2\ell+1)}$ such that \eqref{eq:Z_O} holds.  Let $\hat{\bm{\Omega}} = O \bm{\Omega}$ so that $d{\hat{\bm{\Omega}}} = d\bm{\Omega}$.
Since the components of $\bfr_\ell$ are orthonormal,
\begin{align}
I = \int_{\bbS^2} \bfr_\ell(\bm{\Omega}) \bfr_\ell (\bm{\Omega}) ^\top d \bm{\Omega}
= \int_{\bbS^2} \bfZ_O^\ell \bfr_\ell(\hat{\bm{\Omega}}) \bfr_\ell(\hat{\bm{\Omega}}) ^\top(\bfZ_O^\ell)^\top d \hat{\bm{\Omega}} 
=\bfZ_O^\ell (\bfZ_O^\ell)^\top .
\end{align}

It follows that $\bfZ_O^\ell$ is orthogonal and hence so is $\bfZ_O$.  To show \eqref{eq:rot_operator}, direct application of \eqref{eq:W_O} and \eqref{eq:Z_O} gives
\begin{equation}\label{eq:TQ_1}
(\cW_O \bfv)(\bm{x}) 
= \int_{\bbS^2} \bfr(\bm{\Omega}) f(O \bm{x}, O \bm{\Omega}) \intD \bm{\Omega} \\
= \bfZ_O \int_{\bbS^2} \bfr(\hat{\bm{\Omega}}) f(O \bm{x}, \hat{\bm{\Omega}}) \intD \hat{\bm{\Omega}}\\
= \bfZ_O \bfv (O \bm{x}).
\end{equation}
\end{proof}

\begin{remark}
Recursive algorithms to compute $\bfZ_O$ are given in~\cite{ivanic1996rotation, choi1999rapid, blanco1997evaluation}, and an explicit example for $N=1$ is provided in \cite[Eq.~(11)]{green2003}.
\end{remark}

\begin{proposition}
The RTE \eqref{eq:RTE} is rotationally invariant in the sense that $\psi$ satisfies \eqref{eq:RTE} with source $S$ if and only if, for any orthogonal matrix $O \in \bbR^{3 \times 3}$, $\cR_O \psi$ satisfies \eqref{eq:RTE} with a source $\cR_O S$ and cross-sections $\cR_O \sig{a}$ and $\cR_O \sig{s}$.  Similarly,
the \pn equations \eqref{eq:PN} are rotationally invariant in the sense that $\upn$ satisfies \eqref{eq:PN} with source $\bfs$, if and only if for any orthogonal matrix $O \in \bbR^{3 \times 3}$, $\cW_O \upn$ satisfies \eqref{eq:PN} with a source $\cW_O \bfs$ and cross-sections $\cR_O \sig{a}$ and $\cR_O \sig{s}$. 
\end{proposition}

\begin{proof}
    These result are well known; see e.g., \cite{alldredge2019regularized}.  We provide a proof here for completeness.  For the RTE, the result amounts to proving that $\cR_O$ commutes with the operators $\bm{\Omega} \cdot \grad_{\bm{x}}$ and  $\cQ$.  Let $\hat{\bm{x}} = O\bm{x}$, $\hat{\bm{\Omega}} = O \bm{\Omega}$, and  $\hat{\bm{\Omega}}' = O \bm{\Omega}'$. Direct calculation gives
    \begin{equation}
    \begin{split}
    \cR_O( \bm{\Omega} \cdot \grad_{x} \psi)(\bm{x},\bm{\Omega},t)
    &= \hat{\bm{\Omega}} \cdot \grad_{\hat{\bm{x}}} \psi(\hat{\bm{x}},\hat{\bm{\Omega}},t)
        = O \bm{\Omega} \cdot O \grad_{\bm{x}} \psi(\hat{\bm{x}},\hat{\bm{\Omega}},t)
        = \bm{\Omega} \cdot \grad_{\bm{x}} (\cR_O \psi)(\bm{x},\bm{\Omega},t)  
    \end{split}
    \end{equation}
and, since $d \hat{\bm{\Omega}}' = d\bm{\Omega}'$,
    \begin{equation}
    \begin{split}
        \cR_O \left(\int_{\bbS^2} g(\bm{\Omega} \cdot \bm{\Omega}') \psi(\bm{x},\bm{\Omega}',t) d \bm{\Omega}'\right)
        &= \int_{\bbS^2} g(\hat{\bm{\Omega}} \cdot \bm{\Omega}') \psi(\hat{\bm{x}},\bm{\Omega}',t) d \bm{\Omega}'
        = \int_{\bbS^2} g(\hat{\bm{\Omega}} \cdot \hat{\bm{\Omega}}') \psi(\hat{\bm{x}},\hat{\bm{\Omega}}',t) d \hat{\bm{\Omega}}' \\
        &= \int_{\bbS^2} g(\bm{\Omega} \cdot \bm{\Omega}') \psi(\hat{\bm{x}},\hat{\bm{\Omega}}',t) d \bm{\Omega}'
        = \int_{\bbS^2} g(\bm{\Omega} \cdot \bm{\Omega}') (\cR_O\psi)(\bm{x},\bm{\Omega}',t) d \bm{\Omega}'.
    \end{split}
    \end{equation}
Similarly, for the \pn equations, the result amounts to proving that the operator $\cW_O$ commutes with the operators $\bfA \cdot \grad_{\bm{x}}$ and $\bfG$. Direction calculation, The using the relation $\bfZ_O\bfr(O {\bm{\Omega}})$, gives
    \begin{equation}
    \label{eq:W-Adotgrad-commute}
    \begin{split}
        \cW_O (\bfA \cdot \grad_{\bm{x}} \upn)(\bm{x},t)
        &= \cW_O \left(\int_{\bbS^2} \bfr(\hat{\bm{\Omega}}) \bfr(\hat{\bm{\Omega}})^\top \hat{\bm{\Omega}} \cdot \grad_{\bm{x}} \upn (\bm{x},t ) d \hat{\bm{\Omega}} \right) 
        = \bfZ_O \int_{\bbS^2} \bfr(\hat{\bm{\Omega}}) \bfr(\hat{\bm{\Omega}})^\top \hat{\bm{\Omega}} \cdot \grad_{\hat{\bm{x}}} \upn (\hat{\bm{x}},t ) d \hat{\bm{\Omega}} \\
        &= \int_{\bbS^2} \bfr(\bm{\Omega}) \bfr^\top(\bm{\Omega})  \bm{\Omega} \cdot \grad_{x} \bfZ_O \upn (\hat{\bm{x}},t ) d \bm{\Omega}
        = \bfA \cdot \grad_{\bm{x}} (\cW_O \upn )(\bm{x},t ),
    \end{split}
    \end{equation}
and, because $\bfZ_O$ and $\bfG$ have the same block structure and each block of $\bfG$ is a multiple of the identity, $\bfZ_O$ and $\bfG$ commute.  Hence
    \begin{equation}
    \label{eq:W-G-commute}
    \begin{split}
    \cW_O (\bfG \upn)(\bm{x},t) 
    = \bfZ_O \bfG \upn(\hat{\bm{x}},t) 
    = \bfG \bfZ_O \upn(\hat{\bm{x}},t) 
    = \bfG (\cW_O \upn)(\bm{x},t).
    \end{split}
    \end{equation}
\end{proof}

When the filter strength is constant, the \fpn solution inherits the rotational invariance of the underlying \pn solution. 
However, when the filter strength depends on the \fpn solution, rotational invariance must be explicitly enforced.

\begin{theorem}
\label{theo_rot_invariant_part_reg}
Suppose that for any $\bfv:X \to \bbR^n$ sufficiently smooth, $\Sig{f}$ satisfies
\begin{equation}
\label{eq:filter-invariance-assumption}
    \Sig{f}(\bm {\xi}(\bfv),\bfs;\bm{w}) 
    = \Sig{f}(\bfZ_O \bm {\xi}(\bfv), \bfZ_O \bfs; \bm{w}).
\end{equation}
Then the \fpn equations \eqref{eq:FPN} are rotationally invariant in the sense that $\ufpn$ satisfies \eqref{eq:FPN} with source $\bfs$ if and only if for any orthogonal matrix $O \in \bbR^{3 \times 3}$, $\cW_O \ufpn$ satisfies \eqref{eq:FPN} with a source $\cW_O \bfs$ and cross-sections $\cR_O \sig{a}$ and $\cR_O \sig{s}$. 
\end{theorem}

\begin{proof}
    The proof is similar to the proof for the \pn equations, but with the additional filter term.  
    As before, let $\hat{\bm{x}} = O\bm{x}$.  Then as with $\bfG$, $\bfZ_O$ and $\bfF$ commute.  
    Hence
    \begin{equation}
        \cW_O \big(\sig{f}(\bm{x},t) \bfF \ufpn(\bm{x},t)\big)
        = \sig{f}(\hat{\bm{x}},t)  \bfF \bfZ_O\ufpn(\hat{\bm{x}},t)
        =  \sig{f}(\hat{\bm{x}},t)  \bfF (\cW_O\ufpn)(\bm{x},t),
    \end{equation}
    where $\sig{f}$ is given by \eqref{eq:parameterized-filter}.  
    By the assumption on $\Sig{f}$ in \eqref{eq:filter-invariance-assumption}, for any parameter vector $\bm{w}$,
    \begin{equation}
    \label{sigf-commute}
        \sig{f}(\hat{\bm{x}},t) 
        = \Sig{f}\big(\bm{\xi}(\ufpn)(\hat{\bm{x}},t), \bfs(\hat{\bm{x}}); \bm{w} \big)
        = \Sig{f}(\bfZ_O\bm{\xi}(\ufpn)(\hat{\bm{x}},t), \bfZ_O\bfs(\hat{\bm{x}});\bm{w}).
    \end{equation} 
     However, for any sufficiently smooth function $\bfv \colon X \to \bbR^n$
    \begin{equation}
        \label{eq:Sigf-Z}
    \begin{split}
    Z_O\bm{\xi}(\bfv)(\hat{\bm{x}}) 
    &= \left[Z_O\bfA \cdot \grad_{\bm{x}} \bfv (\hat{\bm{x}}) 
    \quad  
    \sig{t}(\hat{\bm{x}})  Z_O\bfv (\hat{\bm{x}})  
    \quad  
    \sig{s}(\hat{\bm{x}})  Z_O(\bfG-\bfI)\bfv (\hat{\bm{x}})  \right]
    \\
    &=\left[\cW_O (\bfA \cdot \grad_{\bm{x}} \bfv) (\bm{x})
    \quad  
    \cW_O (\sig{t} \bfv) (\bm{x})  
    \quad  
    \cW_O(\sig{s}(\bfG-\bfI)\bfv)) (\bm{x})  \right]
    \end{split}
\end{equation}
where the last line in \eqref{eq:Sigf-Z} follows from the calculations in \eqref{eq:W-Adotgrad-commute} and \eqref{eq:W-G-commute}.  In addition, $\bfZ_O\bfs(\hat{\bm{x}}) = (\cW_O \bfs)(\bm{x})$ by definition.
  Hence \eqref{sigf-commute} becomes
    \begin{equation}
        \sig{f}(\hat{\bm{x}},t)
        = \Sig{f}(\bm{\xi}(\bfZ_O \ufpn)(\hat{\bm{x}},t), \bfZ_O \bfs(\hat{\bm{x}});\bm{w})
        = \Sig{f}(\bm{\xi}( \cW_O \ufpn)(\bm{x},t), (\cW_O\bfs)(\bm{x});\bm{w})
    \end{equation}  
    Thus the proof is complete.
\end{proof}

Given $\bfu = [\bfu_0 \ddd \bfu_\ell \ddd \bfu_N]^\top \in \bbR^1 \times \dots \times \bbR^{2\ell+1} \times \dots \times \bbR^{2N+1} = \bbR^{n}$, let $|\bfu_\ell| = \sqrt{\bfu_\ell^\top \bfu_\ell}$ be the Euclidean norm of $\bfu_\ell$.
To satisfy the condition in \eqref{eq:filter-invariance-assumption}, we pre-process the features $\bm{\xi}(\ufpn)$ using the operator $\bfP\colon\bbR^n\rightarrow\bbR^{N+1}$, given by 
\begin{align}\label{eq:rot_preprocessor} 
    \bfP \bfu = \left[|\bfu_0|,\dots,|\bfu_N|\right]^\top.
\end{align}

\begin{proposition}
  For any neural network $\cN$  with parameters $\bm{w}$, the ansatz 
   \begin{equation}
    \Sig{f}(\bm {\xi}(\cdot),\bfs;\bm{w}) 
    =  \cN (\bfP \bm {\xi}(\cdot), \bfP \bfs ;\bm{w})
   \end{equation}
   fulfills \eqref{eq:filter-invariance-assumption}.
\end{proposition}
\begin{proof}    
    Due to the block structure of $\bfZ_O$ and orthogonality of each block (see Lemma \ref{lem_rot_u}) it follows that $\bfP \bfu =\bfP \bfZ_O\bfu$ for any $\bfu \in \bbR^n$.
    Hence $\bfP \bm {\xi}(\cdot) = \bfP \bfZ_O \bm {\xi}(\cdot)$ and 
    \begin{align}
   \Sig{f}(\bm {\xi}(\cdot),\bfs;\bm{w}) 
    =  \cN (\bfP \bm {\xi}(\cdot), \bfP \bfs ;\bm{w})
    = \cN (\bfP \bfZ_O \bm {\xi}(\cdot), \bfP \bfZ_O \bfs ;\bm{w})
    =  \Sig{f}(\bfZ_O \bm {\xi}(\cdot),\bfZ_O \bfs;\bm{w}),
    \end{align}
which is exactly  \eqref{eq:filter-invariance-assumption}.
\end{proof}

\section{Numerical Results}
\label{sec:numerical results}
In the following, we perform numerical tests with \fpn approximations in 1-D and 2-D geometries, which are both described in Appendix \ref{app:reduction}.
 We explore several test cases to evaluate the performance of the data-driven \fpn equations.  All numerical results can be reproduced using the GitHub repository \url{https://github.com/benplumridge/NN_FPN}.  

We model the data-driven filter $\Sig{f}$ by applying the map $\bfP$ from \eqref{eq:rot_preprocessor} to each component of the feature vector in \eqref{eq_features}, obtaining the preprocessed features 
\begin{equation}
    \bfP \bm{\xi}(\bfv) 
    = [\bfP \bfA \cdot \grad_{\bm{x}} \bfv
    \quad
    \sig{t} \bfP \bfv
    \quad  
    \sig{s} \bfP (\bfG-\bfI)\bfv ].
\end{equation} 
In general there are $3N+3$ features ($N+1$ for each column of $\bm{\xi}(\bfv)$).  However, for isotropic scattering, which is assumed in all of the problems below, only the $\ell=0$ component of the vector $\sig{s} \bfP (\bfG-\bfI)\bfv$ is nonzero.  

Any neural network architecture with non-negative scalar output and the preprocessor $\bfP$ is a feasible choice for the proposed framework. 
Thus as a baseline comparison, we train a scalar filter, i.e. $\Sig{f} = w\in\mathbb{R}_+$, which is independent of the state of the moment system. 
The corresponding filter values are displayed in Table \ref{tab:sigf_const}.
For the neural network based and constant filter, each ansatz is trained for a fixed moment order $N$; once trained it can be used in any simulation with the same moment order $N$ and same space-time discretization.

\begin{table}[H]
\centering
\begin{tabular}{@{}c ccc c cccc@{}}
\toprule
& \multicolumn{3}{c}{1-D} & \phantom{abc} & \multicolumn{4}{c}{2-D} \\
\cmidrule{2-4} \cmidrule{6-9}
$N$ & 3 & 7 & 9 & & 3 & 5 & 7 & 9 \\
\midrule
$\Sigma_f $ & 7.74 & 6.47 & 4.76 & & 31.75 & 27.27 & 16.52 & 13.62 \\
\bottomrule
\end{tabular}
\caption{Optimal scalar constant $\Sigma_f = w \in \mathbb{R}^+$ for 1-D and 2-D. Values are obtained using a degenerate network consisting of a single bias, trained on the same data as the neural network ansatz.  For testing in 1-D the network is trained multiple times, and the values above are one representative instance.}
 \label{tab:sigf_const}
\end{table}

As validation metrics we measure relative error of the particle concentration to the reference solution for the P$_N$ solution ($e_N$), the \fpn solution with neural network filter strength $(e_N^\text{nn})$, and the \fpn solution with constant, but trainable, filter strength ($e_N^\text{const}$):
    \begin{equation}
        e_N = \frac{\left\|\phi_{\pn} - \phi \right\|_{L^2(X)}}{\left\|\phi \right\|_{L^2(X)}} 
        ,\qquad
        e_N^\text{nn} := \frac{\left\|\phi_{\fpn}^{\text{nn}} - \phi \right\|_{L^2(X)}}{\left\|\phi \right\|_{L^2(X)}}
        ,\quand
        e_N^\text{const} := \frac{\left\|\phi_{\fpn}^{\text{const}} - \phi \right\|_{L^2(X)}}{\left\|\phi \right\|_{L^2(X)}},
    \end{equation}
where the reference $\phi$ is computed using a high-order $\pn$ solution and the same space-time discretization and the other solutions.  We also report the ratios
\begin{equation}
    r_N^\text{nn} = \frac{e_N^\text{nn}}{e_N}  
    \qquand
    r_N^\text{const}    = \frac{e_N^\text{const}}{e_N}
\end{equation}
to quantity the improvement provided by the filtering over the standard $\pn$ method.

\subsection{1-D Tests} \label{sec:numerical_1D}

In slab geometries, the \pn and \fpn equations can be reduced to equations in one spatial variable $z \in [\zL,\zR]$. In this setting, the kinetic distribution $\psi = \psi(z,\mu)$, where $\mu = \Omega^{(3)}$ (see \eqref{eq:omega}), is approximated by the expansion $\psifpn(z,\mu) = \tilde{\bfp} ^\top(\mu) \bfv(z,t)$, where $\tilde{\bfp}: [-1,1] \to \bbR$ is a vector-valued functions whose components are the Legendre polynomials normalized with respect to  $L^2[-1,1]$ and $\bfv = [v_0 \ddd v_N]^\top$ satisfies the 1-D \fpn equations
\begin{equation}
    \p_t \bfv + \tilde{\bfA} \p_z \bfv + \sig{a} \bfv + \sig{s} \tilde{\bfG} \bfv + \sig{f} \tilde{\bfF} \bfv = \tilde{\bfs},
\end{equation}
with matrices $\tilde{\bfA}$, $\tilde{\bfG}$, and $\tilde{\bfF}$ that are defined in Appendix \ref{subsec:reduction-1D}.  The physical cross-sections and isotropic source $\tilde{\bfs} = [\tilde{s}_0 \ddd \tilde{s}_N]^\top$ are all functions of $z$, while the filter strength $\sig{f}$ depends on $z$ and $t$ through the state $\bfv$ and source $\tilde{\bfs}$.

We test the data-driven \fpn solver for $N \in \{3,7,9\}$.  In both training and test cases, the initial condition and source are isotropic; that is $v_\ell|_{t=0} =\tilde{s}_\ell=0$ for all $\ell >0$.  Thus we only need to specify $v_0|_{t=0}$ and $\tilde{s}_0$. Moreover, except for Reed's Problem, all components of $v$ are zero on the boundary (vacuum boundary condition).

\subsubsection{Data Sampling Strategy}
All simulation setups to generate training data use the spatial domain $[\zL,\zR]=[-1,1]$ with zero boundary conditions for all components of $\bfv$ on both sides of the domain.  
We specify the random variable $\gamma$ that constitutes the data sampling strategy to train the data-driven filter for the 1-D test cases. 
Specifically, $\gamma$ describes the choice of initial condition for $v_0$ and source $s_0$, selecting from
\begin{enumerate}[label=(\alph*)]
    \item a centered Gaussian with standard deviation $0.1$ as initial condition:
    \begin{align}\label{eq_gauss}
v_0(z,t=0) &=  \frac{1}{\sqrt{2 \pi \varsigma^2}} \exp\left(-\frac{z^2}{2\varsigma^2}\right), \quad \varsigma=0.1
 \end{align}
and zero source;
    \item the indicator function on the interval $[-0.25,0.25]$ as initial condition:
    \begin{equation}
       v_0(z,t=0) = \begin{cases}
           0, & |z| > 0.25,\\
           1, & |z| \leq 0.25,
       \end{cases} 
    \end{equation}
    and zero source;
    \item a bump function centered at $0$ as initial condition: 
    \begin{equation}
       v_0(z,t=0) = \begin{cases}
           0, & |z| > 0.5,\\
           \cos (\pi z), & |z| \leq 0.5,
       \end{cases} 
    \end{equation} and zero source;
    \item the indicator function on the interval $[-0.25,0.25]$ and a source term given by
   \begin{equation}
    \tilde{s}_0(z)
    =2\mathbbm{1}_{z<-0.75} + 2\mathbbm{1}_{z>0.75} + \mathbbm{1}_{-0.25<z<0.25}.
   \end{equation}
\end{enumerate}
These simulation setups are illustrated in Figure \ref{fig:IC1D}.
All setups are simulated until final time $\tf=0.5$.
The value of $\gamma$ also indicates the value for the scattering and absorption cross section.  For each of the four initial condition and source setups above, constant values of $\sig{s}$ and $\sig{a}$ are chosen randomly:
\begin{equation}
\sig{s} \sim \cU(0, 1), \quad \sig{a} \sim \cU(0, 1 - \sig{s}),
\end{equation}
where $\cU(a,b)$ denotes the uniform distribution on the interval $(a,b)$.

\begin{figure}[t!]
    \centering
    \begin{subfigure}{0.23\textwidth}
        \centering
        \includegraphics[width=\linewidth]{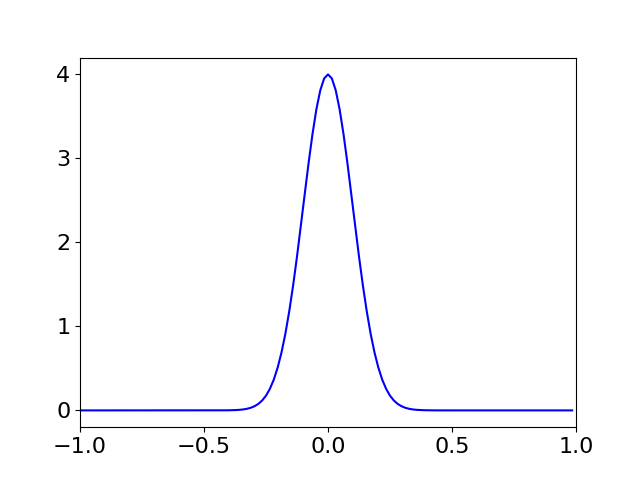}
        \caption{Gaussian}
        \label{fig:gauss}
    \end{subfigure}
    \hfill
    \begin{subfigure}{0.23\textwidth}
        \centering
        \includegraphics[width=\linewidth]{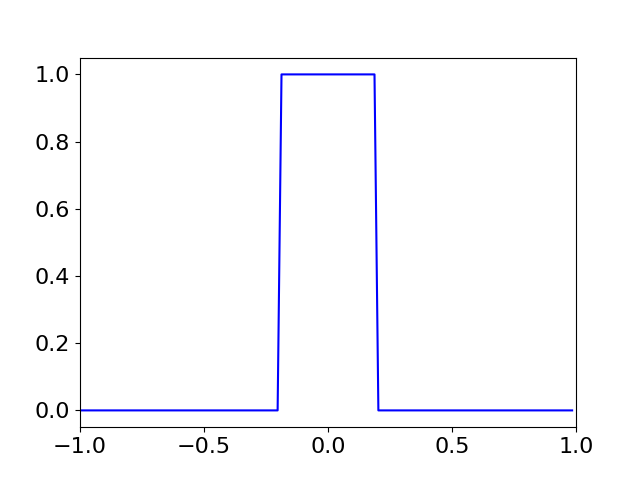}
        \caption{Square-Wave}
        \label{fig:square}
    \end{subfigure}
    \begin{subfigure}{0.23\textwidth}
        \centering
        \includegraphics[width=\linewidth]{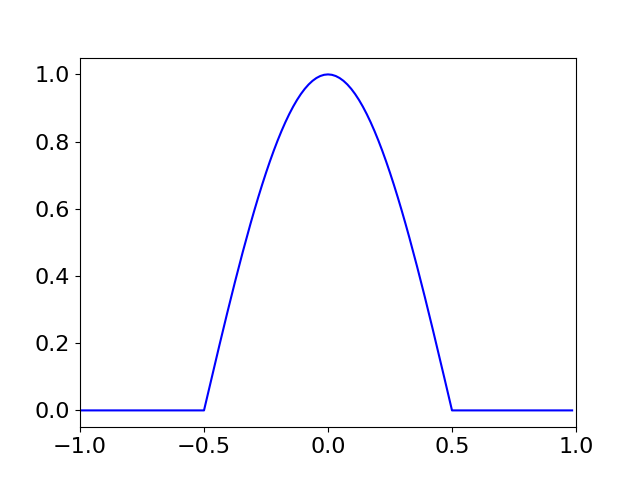}
        \caption{Bump function}
        \label{fig:bump}
    \end{subfigure}
    \hfill
    \begin{subfigure}{0.23\textwidth}
        \centering
        \includegraphics[width=\linewidth]{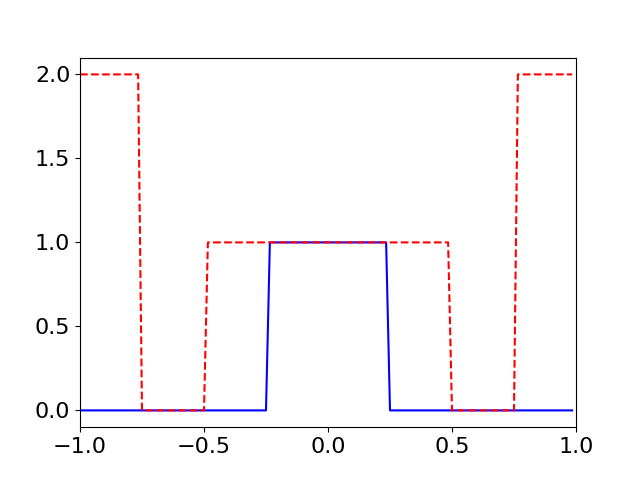}
        \caption{Discontinuous Source}
        \label{fig:disc_source}
    \end{subfigure}
   \caption{Simulation setups for the 1-D data generation. 
   The blue solid line denotes the initial data and the red dotted line denotes the source. 
}
    \label{fig:IC1D}
\end{figure}

\subsubsection{Neural Network Architecture and Training}
We train the filter strength $\Sig{f}$ for $N \in \{3,7,9\}$, modeling the filter strength with a fully connected feed-forward neural network with a single hidden layer and scalar output. 
ReLu activation functions are applied to both the hidden and output layers.  
The features are preprocessed and normalized before being fed into the network. 
The weights and biases of the hidden and output layers, $\bm{W}^{(1)}, \bm{b}^{(1)}$ and $\bm{W}^{(2)}, \bm{b}^{(2)}$ are initialized from uniform distributions: 

\begin{equation}
    \bm{W}^{(1)} \sim \mathcal{U}(-n_{\text{f}}^{-1/2}, \, n_{\text{f}}^{-1/2}), \quad
    \bm{b}^{(1)} \sim \mathcal{U}(-n_{\text{f}}^{-1/2}, \, n_{\text{f}}^{-1/2}),
\end{equation}
\begin{equation}
    \bm{W}^{(2)} \sim \mathcal{U}(-n_{\text{h}}^{-1/2}, \, n_{\text{h}}^{-1/2}), \quad
    \bm{b}^{(2)} \sim \mathcal{U}(-n_{\text{h}}^{-1/2}, \, n_{\text{h}}^{-1/2}),
\end{equation}
where $n_{\text{f}}$ is the number of input features (which includes $2N+3$ solution features and one source feature) and $n_{\text{h}}$ is the number of hidden neurons.
Table \ref{tab:1D_params} shows the network and training parameters. 
The networks are trained using the Adam optimizer \cite{Kingma2017} in  with weight decay between the layers.  

\begin{table}[H]
\centering
\resizebox{\textwidth}{!}{
\begin{tabular}{cccccccc}
\toprule
input features $(n_{\text{f}})$ &hidden layers & hidden width $(n_{\text{h}})$ & learning rate & batch size & number of epochs  & weight decay\\
\midrule
2N + 4 &1 & 50 & 0.1 & 64 & 200 & 1e-5 \\
\bottomrule
\end{tabular}
}
\caption{The table shows parameters used in 1-D training. The learning rates are chosen by an initial hyperparameter search. }
\label{tab:1D_params}
\end{table}

In each iteration, each of the simulation setups shown in Figure \ref{fig:IC1D} is assigned $\frac{\text{batch size}}{4}$ independently sampled realizations of $\sig{a}$ and $\sig{s}$.
For each of these configurations, we compute a high-order P$_{127}$ reference solution and an \fpn solution using the data-driven filter strength.
Both \fpn and the reference P$_{127}$ system use grid spacings of $\dz = \frac{1}{64}$ and $\Delta t = \frac{1}{128}$ for the space-time discretization. 
We repeat this process 10 times, producing 10 models for each $N \in \{3,7,9\}$. 
In the following tests, we plot particle concentrations from one representative iteration and report the sample mean of the relative errors $r_N^{\text{nn}}$ over all 10 runs.

\subsubsection{Test Case: Gaussian}
The Gaussian test case features a localized initial condition in a purely scattering medium.  The spatial domain is [-1,1] and solutions are run to $\tf = 0.5$ and $\tf=1.0$.  The cross-sections are $\sig{s} = 0.5$ and $\sig{a} = 0$.  The initial condition is a Gaussian of the form in \eqref{eq_gauss}, but with $\varsigma=0.05$. 
We compare the data-driven \fpn simulation for $N \in \{3,7,9\}$ to a P$_{127}$ reference solution.
We observe in Figure \ref{tab:gauss} that the data-driven \fpn equation with the neural network-based filter $\Sig{f}$ yields a final time solution for $\phi_{\fpn}^{\text{nn}}(\cdot,t_{\tf})$ that matches the high-order reference simulation well for $N \in \{3,7,9\}$, whereas the low-order \pn simulations exhibit oscillatory behavior.  
While the constant filter reduces the \pn error significantly in all cases, the neural network-based filter solution achieves greater improvements.

\begin{figure}[H]
    \centering

    \begin{subfigure}[b]{0.9\textwidth}
        \centering
        \begin{tabular}{ccc}
            $N=3$ & $N=7$ & $N=9$ \\
            \includegraphics[width=0.3\textwidth]{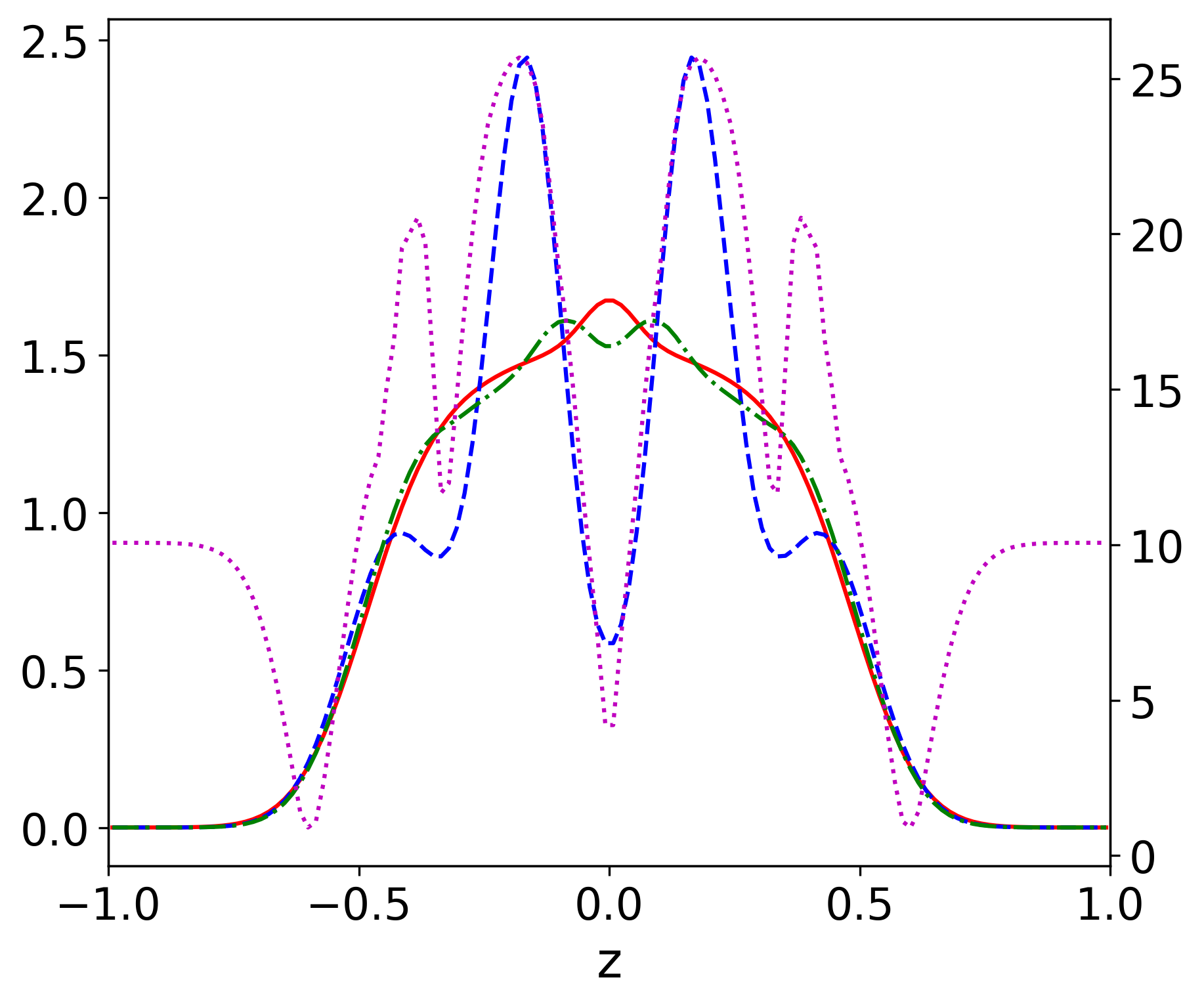} &
            \includegraphics[width=0.3\textwidth]{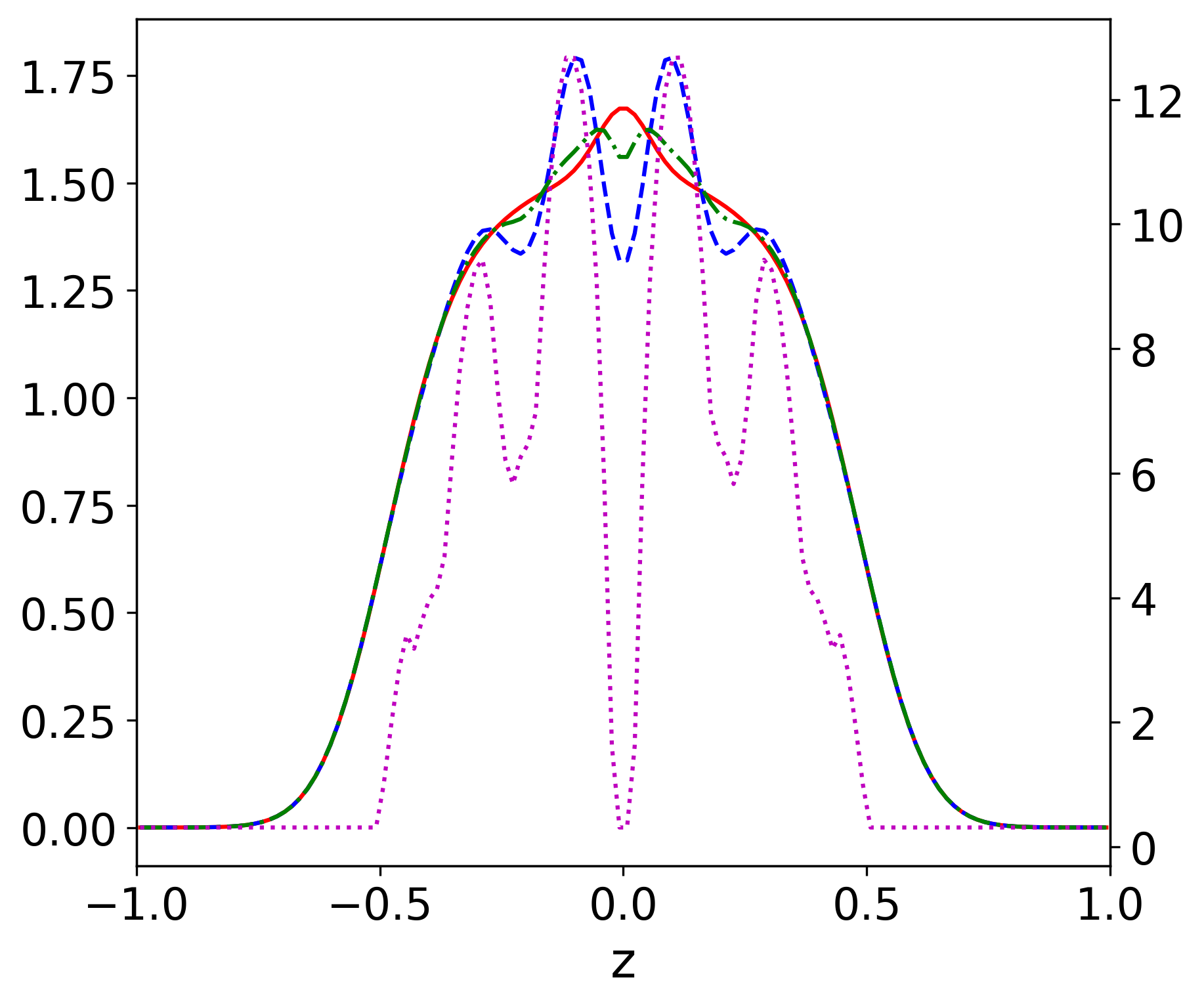} &
            \includegraphics[width=0.3\textwidth]{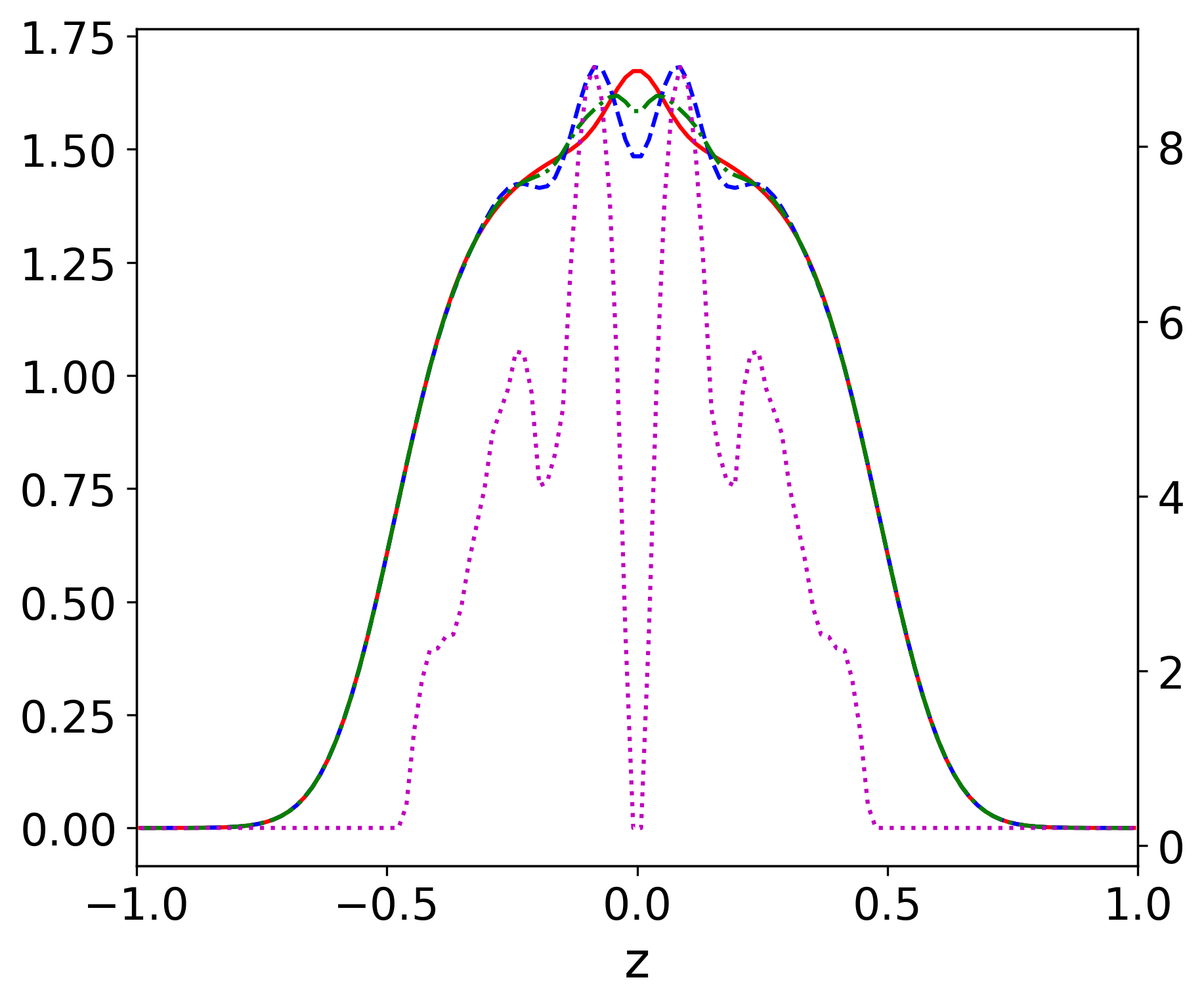} \\
            \includegraphics[width=0.3\textwidth]{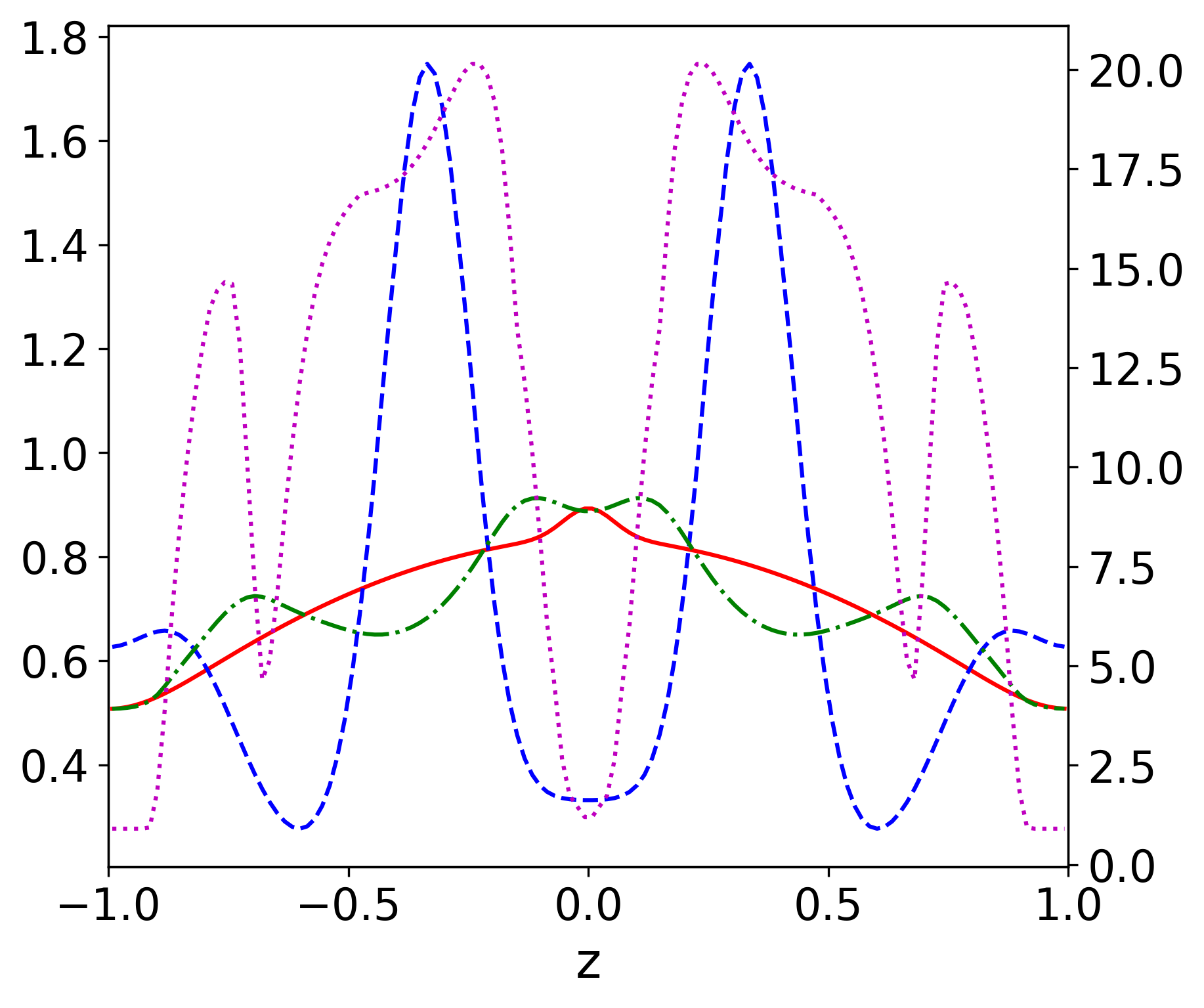} &
            \includegraphics[width=0.3\textwidth]{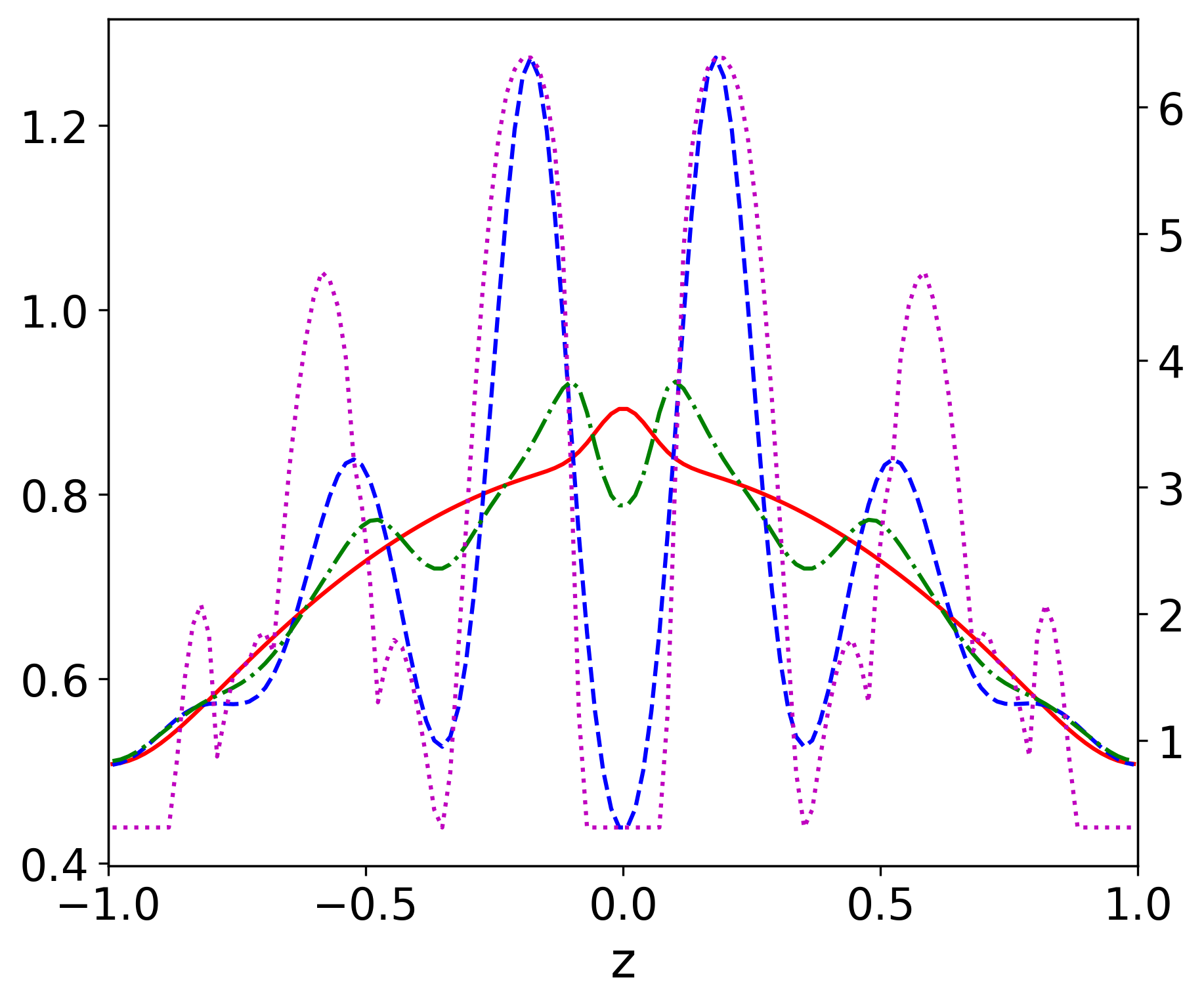} &
            \includegraphics[width=0.3\textwidth]{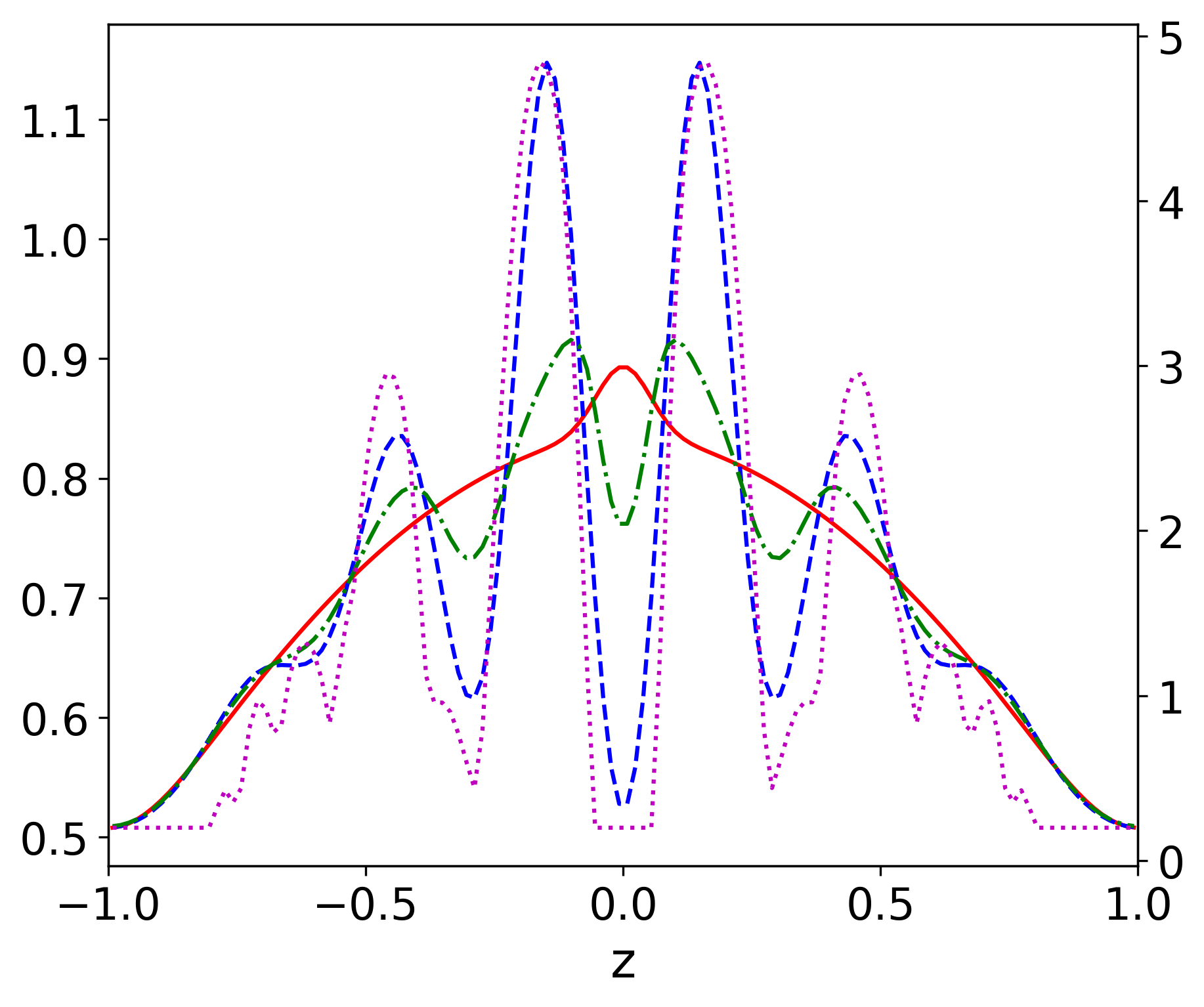} \\
         \multicolumn{3}{c}{\includegraphics[width=0.5\textwidth]{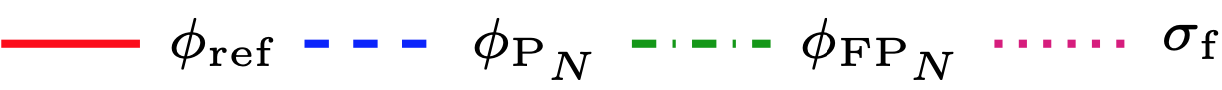}} \\
        \end{tabular}
        \caption{Particle concentrations at $\tf=0.5$ (top row) and $\tf=1.0$ (bottom row) for $N=3,7,9$.
        In each plot, the scale for the particle concentrations is shown on the left and the scale for $\sig{f}$ is shown on the right.}
    \end{subfigure}

    \vspace{1em}

    \begin{subfigure}[b]{\textwidth}
        \centering
        
\resizebox{\textwidth}{!}{
        \begin{tabular}{@{}c ccc c ccc@{}}
        \toprule
        & \multicolumn{3}{c}{Neural Network ($r_N^{\text{nn}}$)} & \phantom{abc} & \multicolumn{3}{c}{Constant ($r_N^{\text{const}}$)} \\
        \cmidrule{2-4} \cmidrule{6-8}
        $\tf$  & $N=3$ & $N=7$ & $N=9$ & & $N=3$ & $N=7$ & $N=9$ \\
        \midrule
        0.5 & $\bm{0.100 \pm 0.006}$ & $\bm{0.240 \pm 0.001}$ & $\bm{0.388 \pm 0.002}$ & & $0.592 \pm 0.002$ & $0.643 \pm 0.002$ & $0.748 \pm 0.000$ \\
        1.0 & $\bm{0.151 \pm 0.004}$ & $\bm{0.198 \pm 0.000}$ & $\bm{0.295 \pm 0.002}$ & & $0.338 \pm 0.003$ & $0.357 \pm 0.002$ & $0.464 \pm 0.000$ \\
        \bottomrule
        \end{tabular}}
         \caption{Ratio of the \fpn relative to the P$_N$ error for the neural network filter strength ($r_N^{\text{nn}}$) and constant filter strength ($r_N^{\text{const}}$). Smaller numbers are better, and  bold denotes which trainable filter that does better for each choice of $N$ and $\tf$.}
        \end{subfigure}

    \caption{(Gaussian) 
    The neural network ansatz consistently outperforms the constant ansatz in all cases.
    }
    \label{tab:gauss}
\end{figure}

\subsubsection{Test Case: Vanishing Cross-Section}
 The vanishing cross-section problem features non-smooth initial data in a purely scattering medium with smooth, but large variations in $\sig{s}$.
In particular, $\sig{s}$ is close to zero in the center of the spatial domain $[-1,1]$ and large at the boundaries.  
The initial data and scattering cross-sections are \cite{Buet2004} 
\begin{subequations}
\begin{align}
v_0(z) &=  
\begin{cases}
1, & z \in (-0.2,0.2),\\
0, & z \in [-1,0,-0.2] \cup [0.2,1.0]
\end{cases}\\
\sig{s}(z) &= 100z^4.
\end{align}
\end{subequations}

We compare the data-driven \fpn simulation for $N \in \{3,7,9\}$ with the P$_{127}$ reference solution.  The results, which are given in 
 Figure \ref{tab:van_cs}, show that the constant filter produces the smallest error in all but one case. While the neural network filter does reduce the error with respect to the \pn solution, the error in the center of the domain, where the $\sig{s}$ is small is still significant.

\begin{figure}[H]
    \begin{subfigure}[b]{0.9\textwidth}
        \centering
        \begin{tabular}{ccc}
            $N=3$ & $N=7$ & $N=9$ \\
            
          \includegraphics[width=0.3\textwidth]{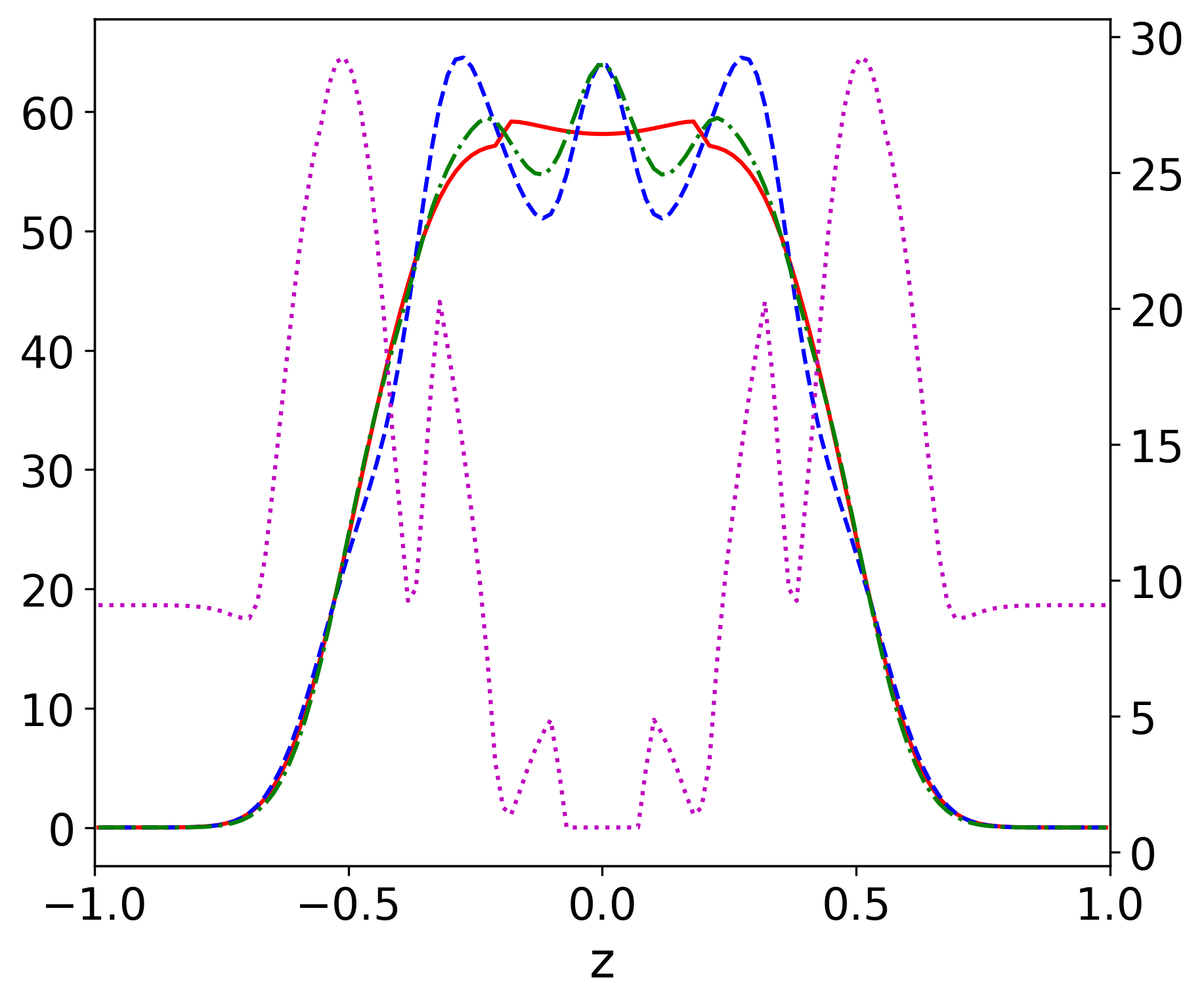}
             & \includegraphics[width=0.3\textwidth]{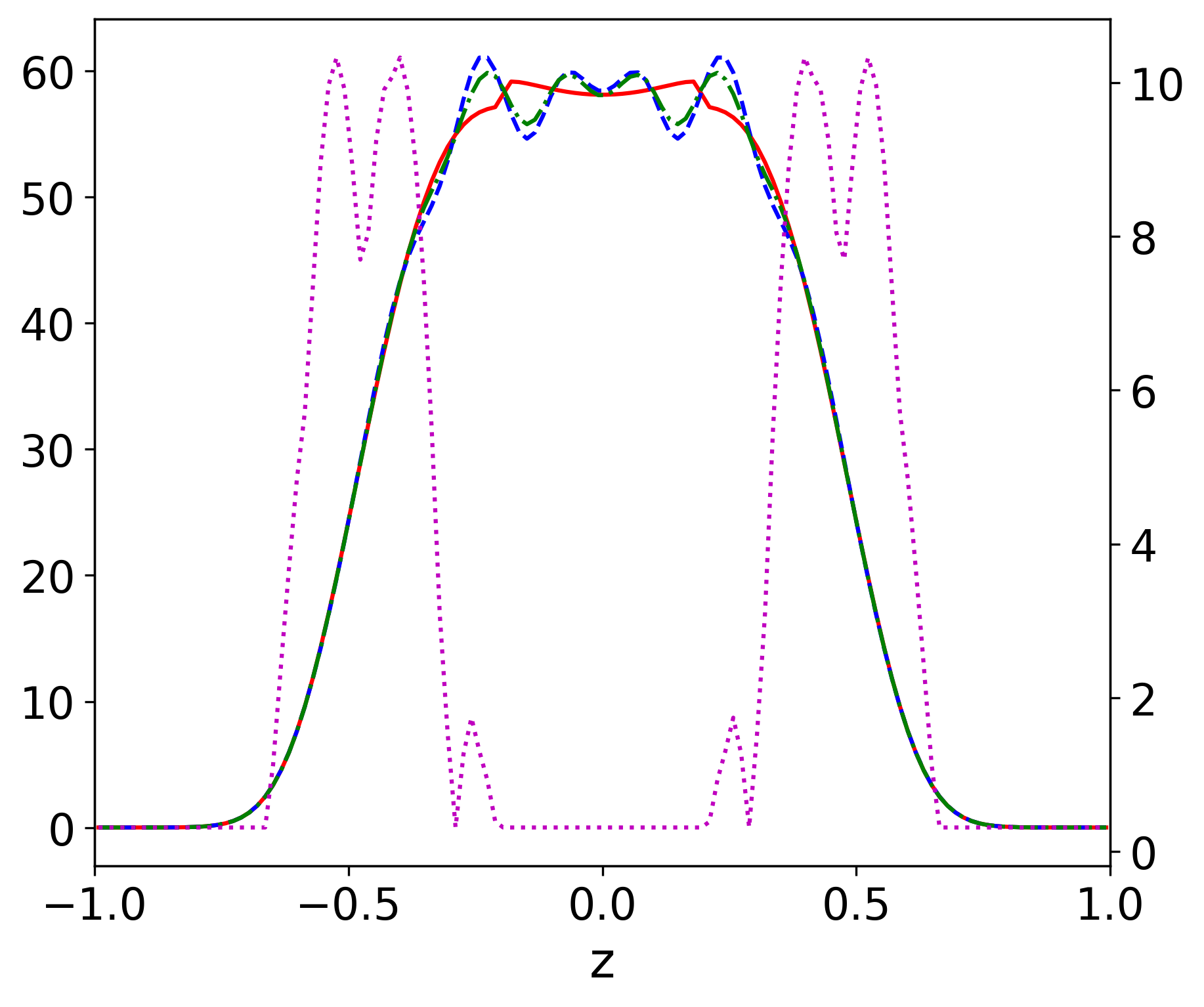}  
             & \includegraphics[width=0.3\textwidth]{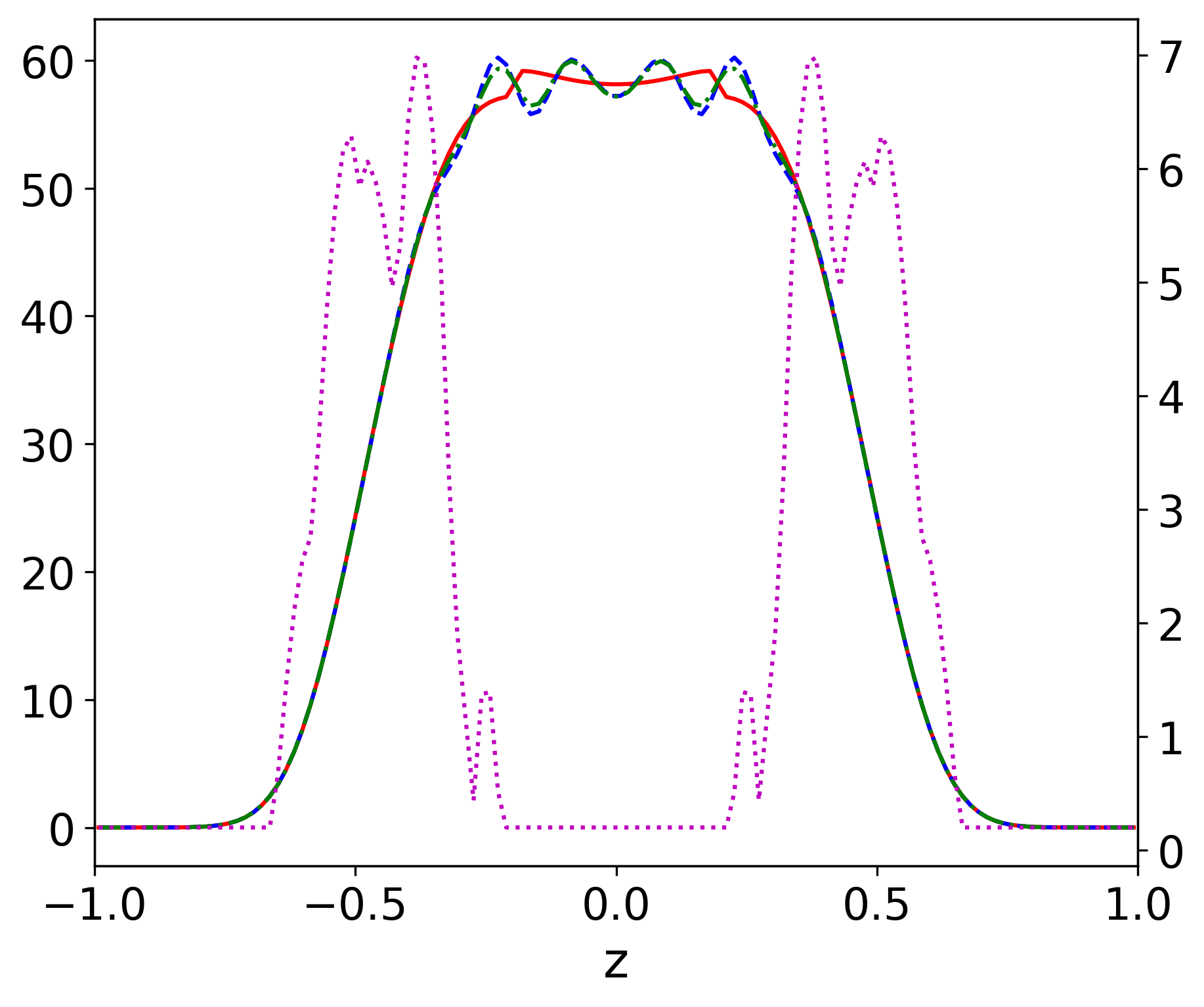} \\
        \includegraphics[width=0.3\textwidth]{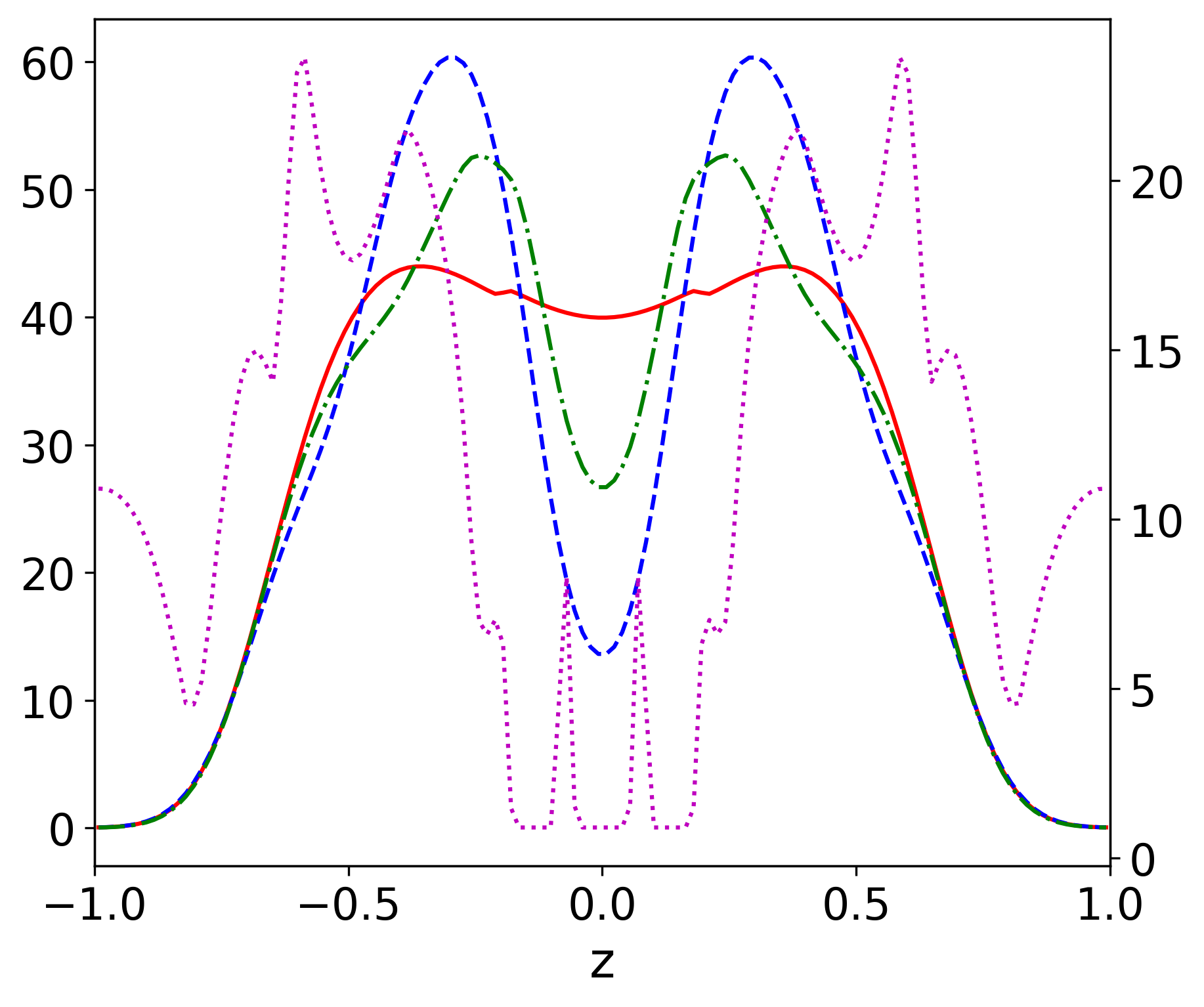} 
             & \includegraphics[width=0.3\textwidth]{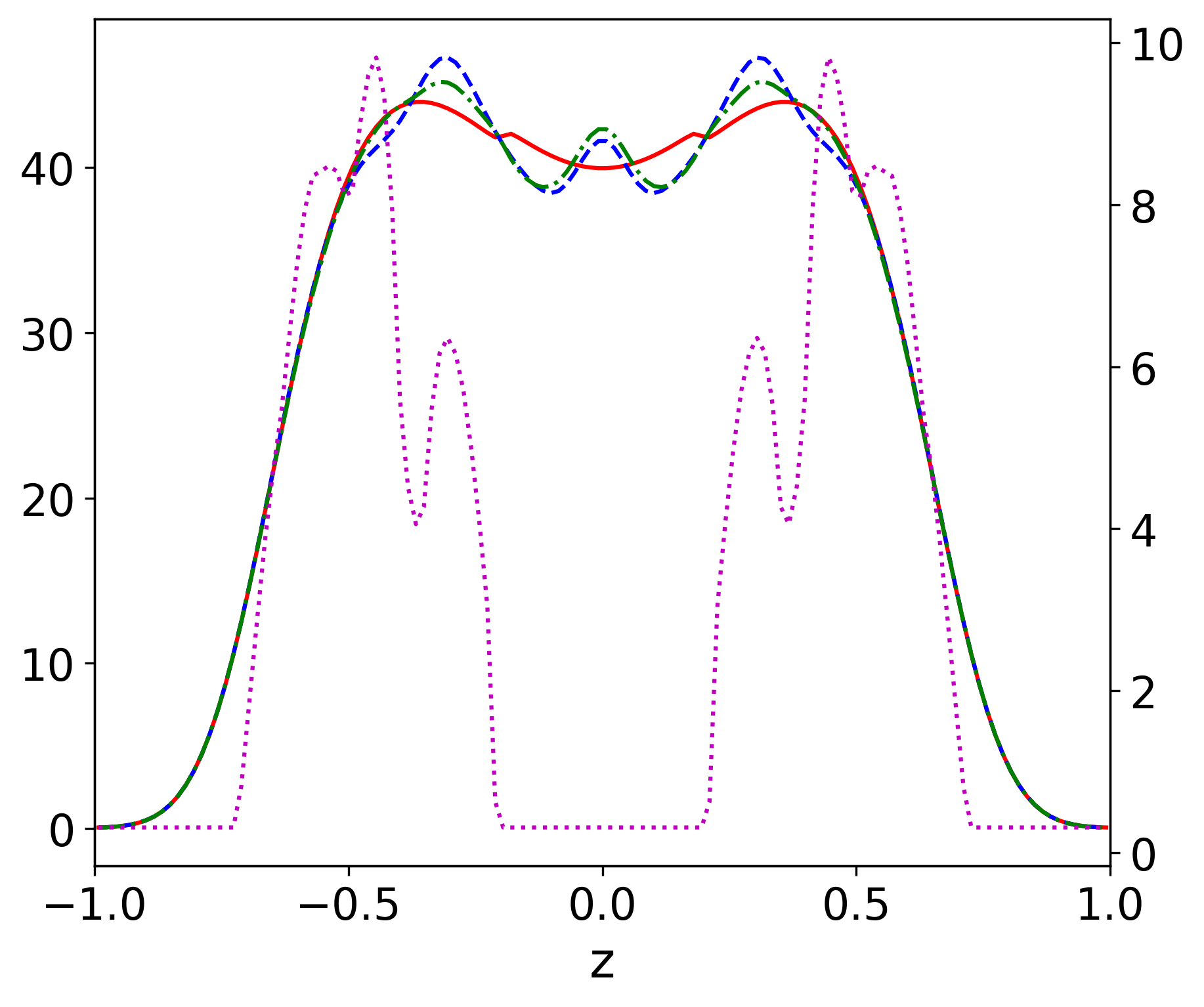}
             & \includegraphics[width=0.3\textwidth]{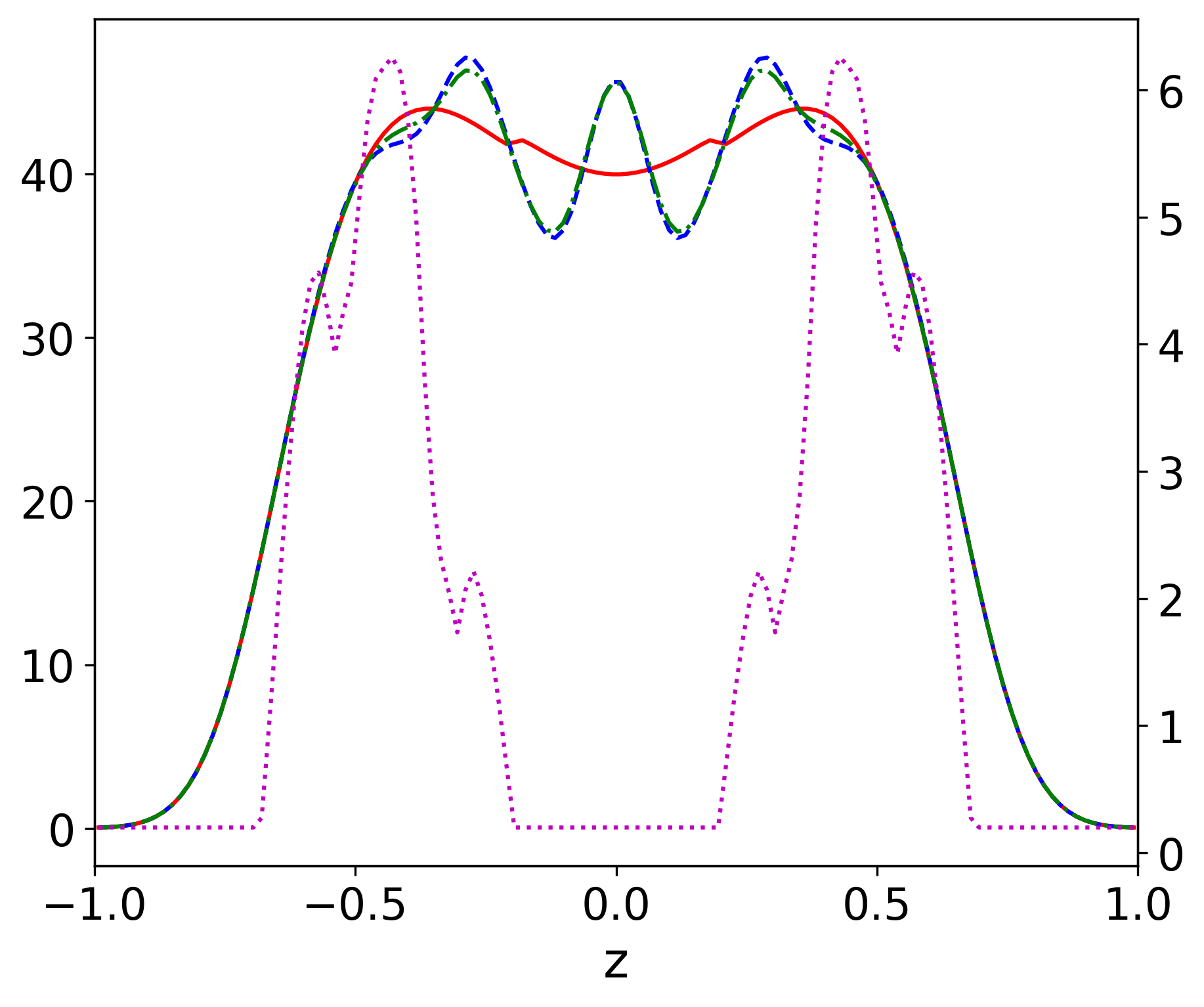}\\ 
            \multicolumn{3}{c}{\includegraphics[width=0.5\textwidth]{figures/legend.png}} \\
            \end{tabular}
                    \caption{Particle concentrations at $\tf=0.5$ (top row) and $\tf=1.0$ (bottom row) for $N=3,7,9$.        
                    The scale for the particle concentration is shown on the left and the scale for $\sig{f}$ is shown on the right in each plot.}
           \end{subfigure}
    \vspace{1em}
    \centering
    \begin{subfigure}[b]{\textwidth}
        \centering
        
\resizebox{\textwidth}{!}{
        \begin{tabular}{@{}c ccc c ccc@{}}
        \toprule
        & \multicolumn{3}{c}{Neural Network ($r_N^{\text{nn}}$)}   & \phantom{abc} & \multicolumn{3}{c} {Constant ($r_N^{\text{const}}$)}\\
        \cmidrule{2-4} \cmidrule{6-8}
        $\tf$ & $N=3$ & $N=7$ & $N=9$ & & $N=3$ & $N=7$ & $N=9$ \\
        \midrule
        0.5 & $\bm{0.437 \pm 0.001}$ & $0.678 \pm 0.000$ & $0.772 \pm 0.001$ & & $0.646 \pm 0.002$ & $\bm{0.649 \pm 0.002}$ & $\bm{0.728 \pm 0.000}$ \\
        1.0 & $0.497 \pm 0.046$ & $0.748 \pm 0.000$ & $0.908 \pm 0.000$ & & $\bm{0.323 \pm 0.003}$ & $\bm{0.446 \pm 0.002}$ & $\bm{0.465 \pm 0.001}$ \\
        \bottomrule
        \end{tabular}
        }
          \caption{Ratio of the \fpn relative to the P$_N$ error for the neural network filter strength ($r_N^{\text{nn}}$) and constant filter strength ($r_N^{\text{const}}$). Smaller numbers are better, and  bold denotes which trainable filter that does better for each choice of $N$ and $\tf$.}
    \end{subfigure}
        \caption{(Vanishing Cross-Section) 
        The constant ansatz yields the smaller error in all but one case.
        }
    \label{tab:van_cs}
\end{figure}

\subsubsection{Test Case: Discontinuous Cross-Section}
The discontinuous cross-section problem contains sharp discontinuities in $\sig{s}$ and the initial data. 
This scenario is common in physical models with multiple materials. 
The problem domain is [-1,1].  The initial data and scattering cross-sections are \cite{Hauck2008}

\begin{subequations}
\begin{align}
v_0(z) &=  
\begin{cases}
1, & z \in (-0.2,0.2),\\
0, & z \in [-1,-0.2] \cup [0.2,1.0]
\end{cases}\\
\sig{s}(z) &=  
\begin{cases}
0.2, & z \in [-0.65,-0.35] \cup [0.35,0.65],\\
1, & z \in [-1,-0.65) \cup (-0.35,0.35) \cup (0.65,1].
\end{cases}
\end{align}
\end{subequations}

We compare the data-driven \fpn simulation for $N \in \{3,7,9\}$ with the P$_{127}$ reference solution.
We observe in Figure \ref{tab:disc} that the low-order \pn simulation yields a final time solution with oscillatory behavior.
The neural network-based filter $\Sig{f}$ suppresses these oscillations, improving accuracy.
While the constant filter reduces the \pn error significantly over all cases, the neural network-based filter solution achieves greater improvements in all but two cases.

\begin{figure}[H]
    \centering

    \begin{subfigure}[b]{0.9\textwidth}
        \centering
        \begin{tabular}{ccc}
            $N=3$ & $N=7$ & $N=9$ \\
              \includegraphics[width=0.3\textwidth]{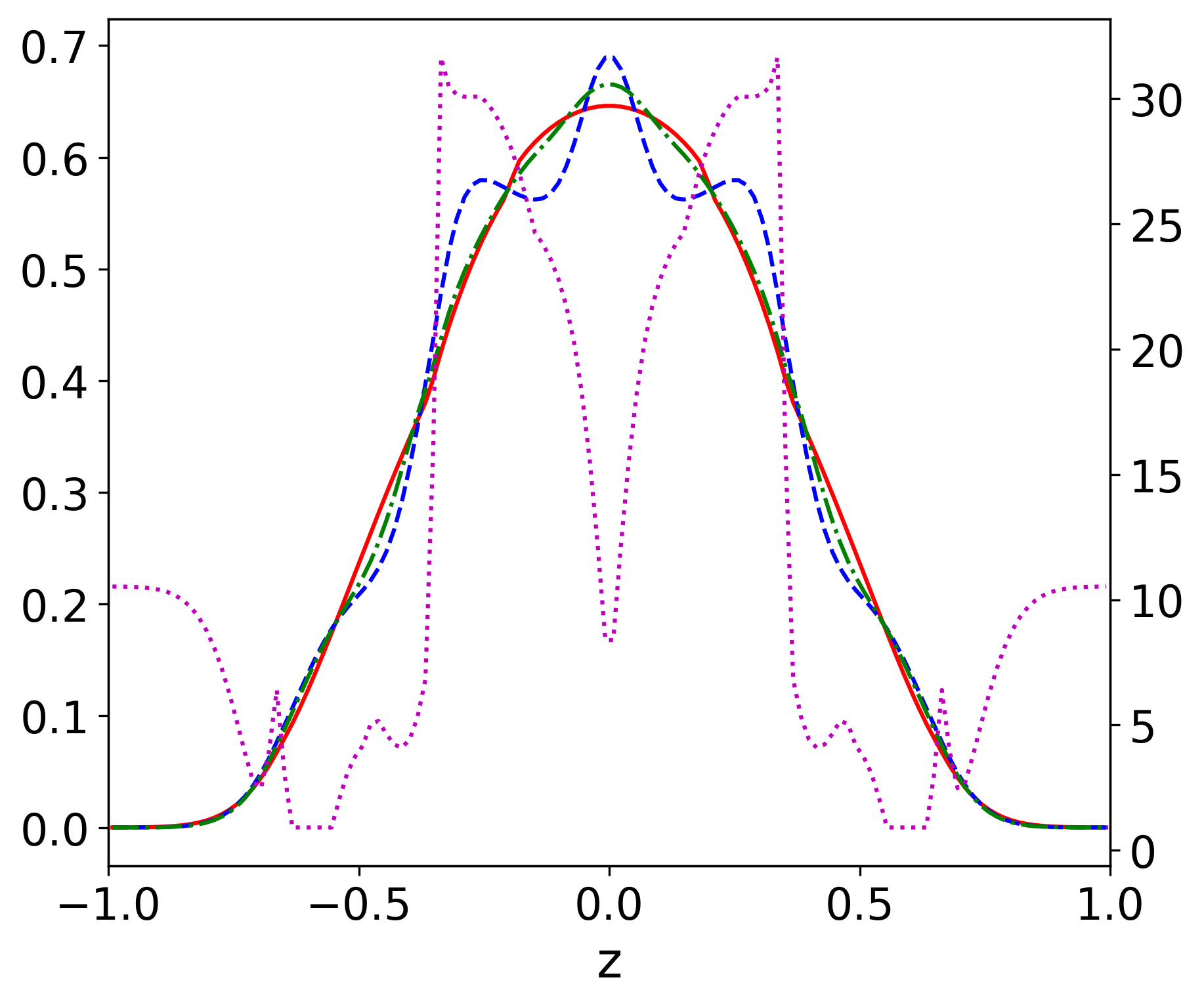}
            &   \includegraphics[width=0.3\textwidth]{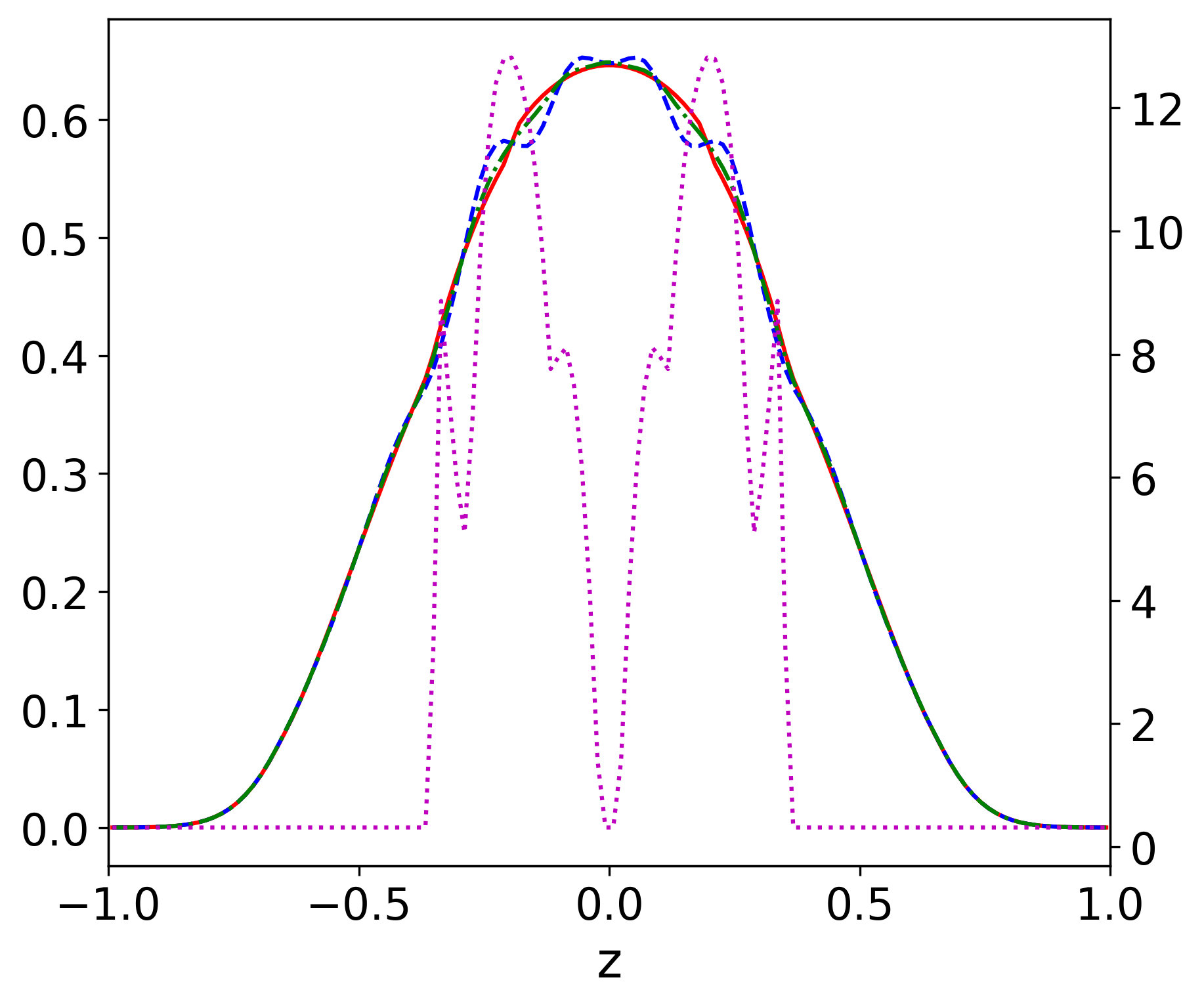}
            & \includegraphics[width=0.3\textwidth]{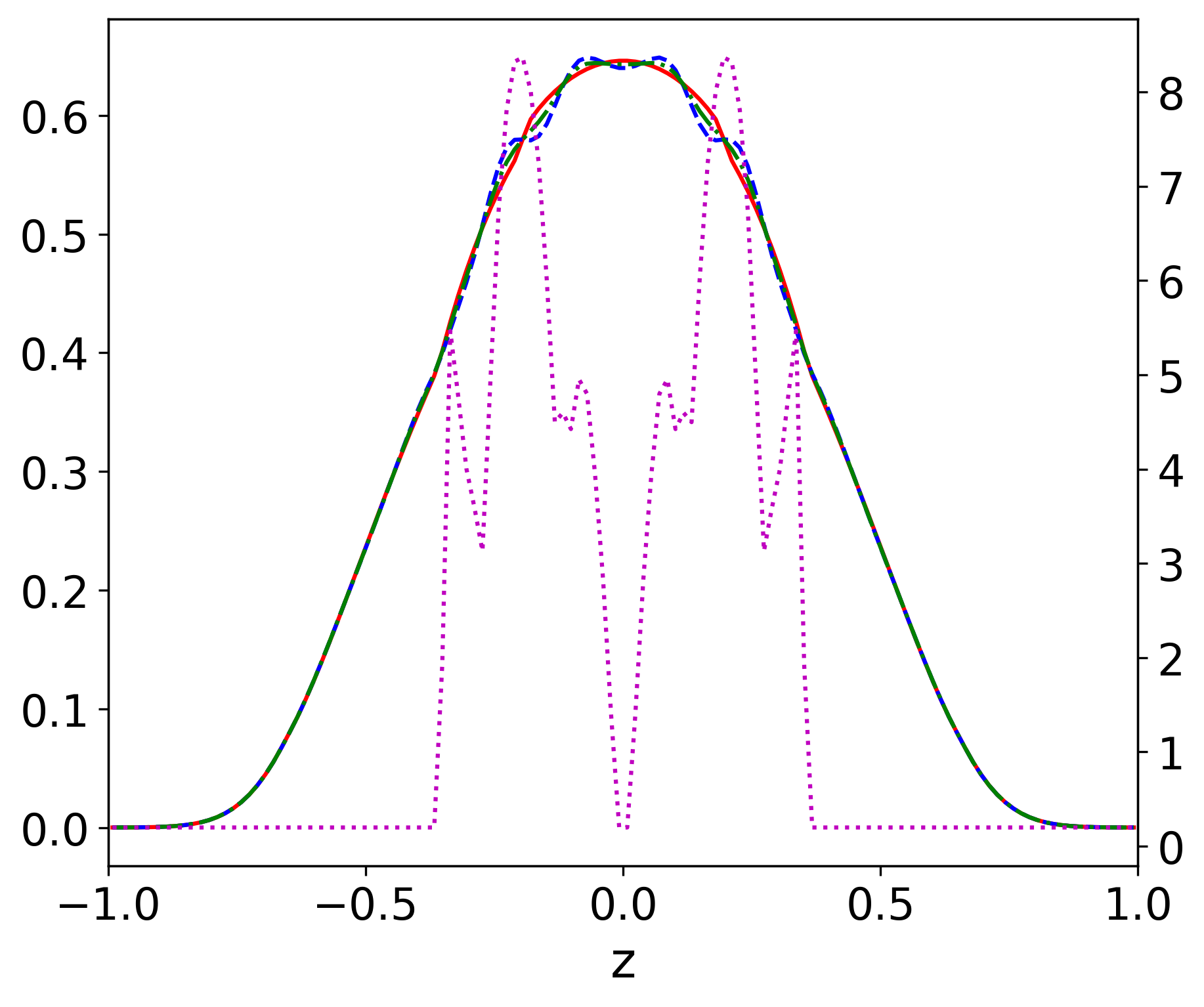} \\
        \includegraphics[width=0.3\textwidth]{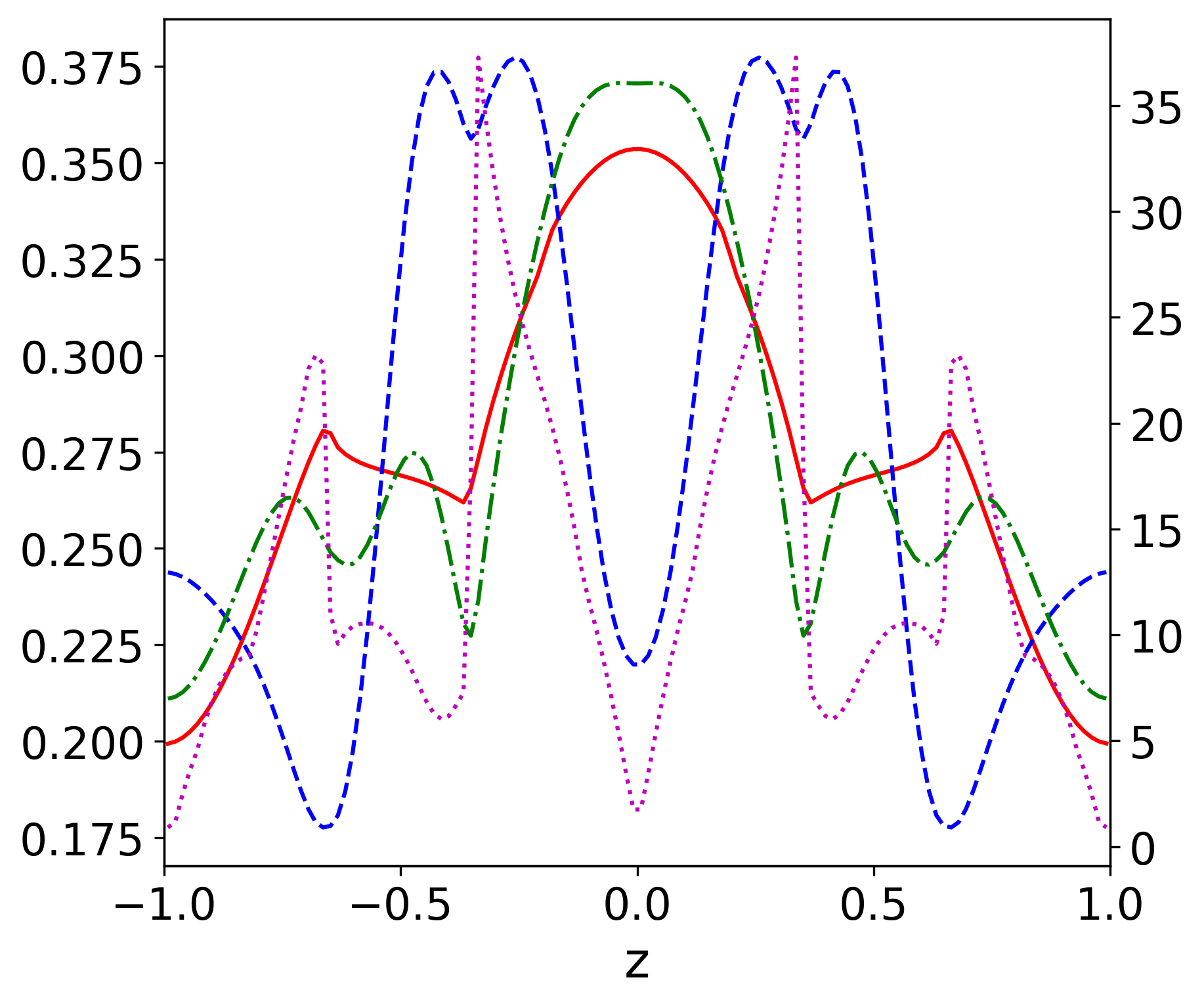}
            & \includegraphics[width=0.3\textwidth] {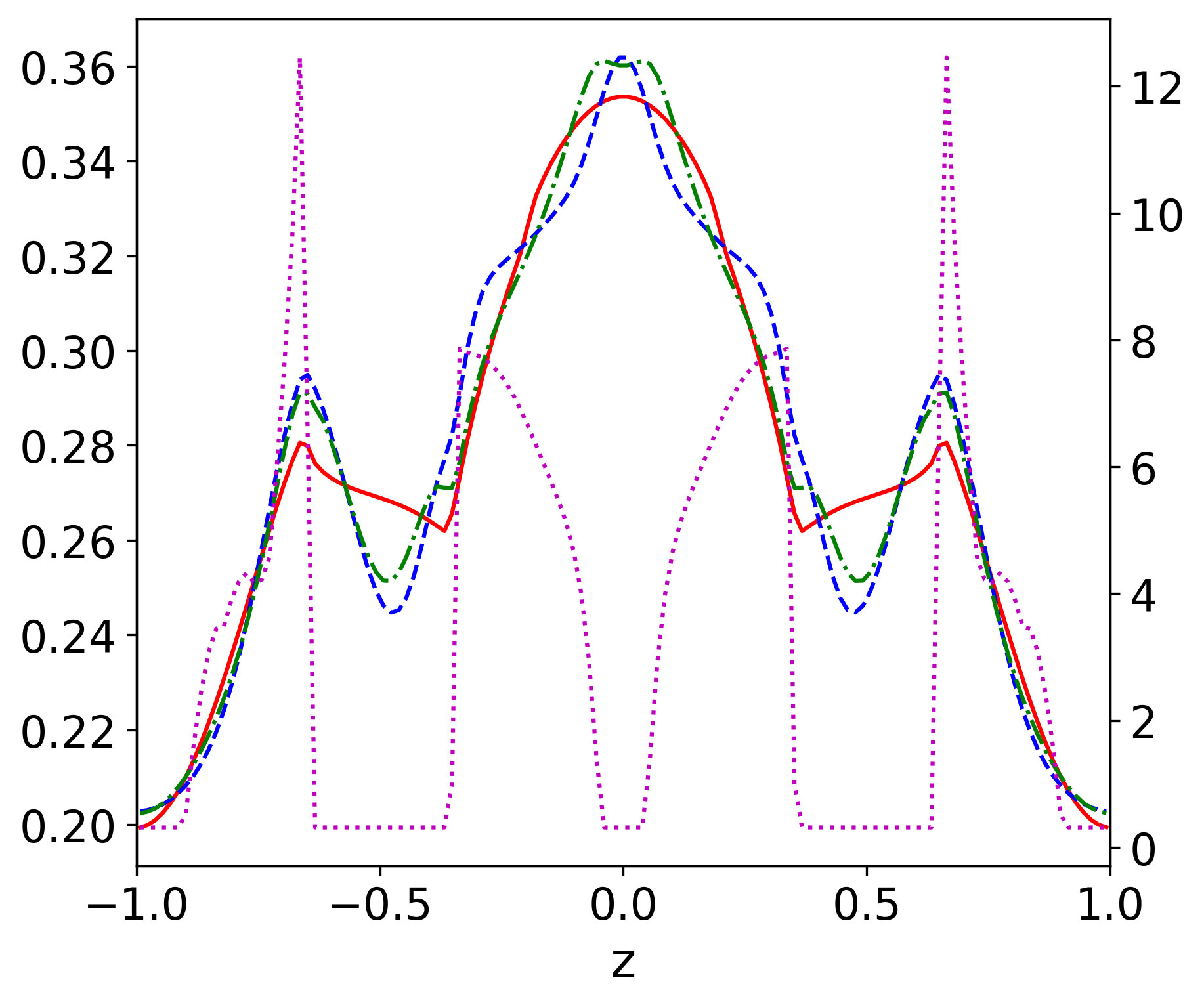}
            & \includegraphics[width=0.3\textwidth] {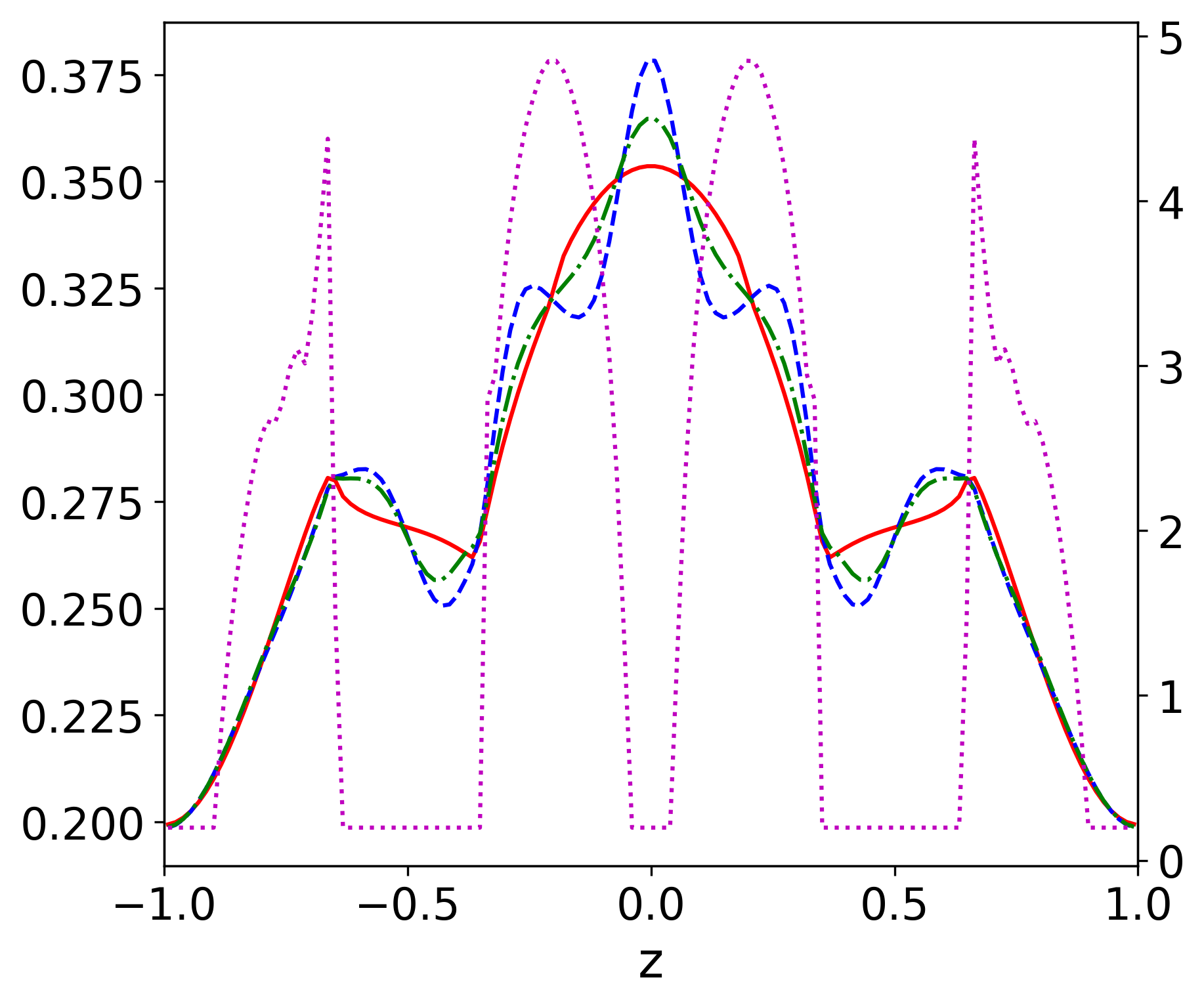} \\
            \multicolumn{3}{c}{\includegraphics[width=0.5\textwidth]{figures/legend.png}} \\
        \end{tabular}
        \caption{Particle concentrations at $\tf=0.5$ (top row) and $\tf=1.0$ (bottom row) for $N=3,7,9$. 
        The scale for the particle concentration $\phi$ is shown on the left and the scale for $\sig{f}$ is shown on the right in each plot. }

    \end{subfigure}

    \vspace{1em}

    \begin{subfigure}[b]{\textwidth}
        \centering
        
\resizebox{\textwidth}{!}{
       \begin{tabular}{@{}c ccc c ccc@{}}
        \toprule
        & \multicolumn{3}{c}{Neural Network ($r_N^{\text{nn}}$)} & \phantom{abc} & \multicolumn{3}{c}{Constant ($r_N^{\text{const}}$)} \\
        \cmidrule{2-4} \cmidrule{6-8}
        $\tf$ & $N=3$ & $N=7$ & $N=9$ & & $N=3$ & $N=7$ & $N=9$ \\
        \midrule
        0.5 & $\bm{0.309 \pm 0.012}$ & $\bm{0.325 \pm 0.001}$ & $\bm{0.492 \pm 0.002}$ & & $0.624 \pm 0.001$ & $0.649 \pm 0.002$ & $0.729 \pm 0.000$ \\
        1.0 & $\bm{0.246 \pm 0.021}$ & $0.633 \pm 0.000$ & $0.487 \pm 0.001$ & & $0.312 \pm 0.003$ & $\bm{0.346 \pm 0.002}$ & $\bm{0.447 \pm 0.001}$ \\
        \bottomrule
        \end{tabular}}
       \caption{Ratio of the \fpn relative to the P$_N$ error for the neural network filter strength ($r_N^{\text{nn}}$) and constant filter strength ($r_N^{\text{const}}$). Smaller numbers are better, and  bold denotes which trainable filter that does better for each choice of $N$ and $\tf$.}
          \end{subfigure}
    \caption{(Discontinuous Cross-Section) 
    The neural network ansatz yields the smallest error in all but one case. }
    \label{tab:disc}
\end{figure}

\subsubsection{Test Case: Reed's Problem}

Reed's problem involves transport through a domain with multiple materials characterized by different cross-sections and sources.
The problem layout is described in Table \ref{tar_reed}.
The left boundary condition is reflecting, while the right boundary is vacuum.

\begin{table}[H]
    \centering
\begin{tabular}{@{}lccccc@{}}
\toprule
& Region 1 & Region 2 & Region 3 & Region 4 & Region 5 \\
\midrule
$x$           & $[0,2)$   & $[2,3)$   & $[3,5)$   & $[5,6)$   & $[6,8]$ \\
$S$           & 50        & 0         & 0         & 1         & 0       \\
$\sig{s}$    & 0         & 0         & 0         & 0.9       & 0.9     \\
$\sig{t}$    & 50        & 5         & 0         & 1         & 1       \\
\bottomrule
\end{tabular}
    \caption{Spatial configuration for Reed's problem. }
    \label{tar_reed}
\end{table}

Solutions at $\tf= 5$ and $\tf=10$ are presented in Figure \ref{tab:Reed's}.
For $N =3$, the neural network filter yields smaller errors than the constant filter, whereas for $N=7$ and $N=9$, the constant filter solutions achieve the smallest errors. As in the vanishing cross-section test, the largest errors occur in the streaming region (Region 3) where the $\sig{f}$ is small.

\begin{figure}[t]
    \centering
    \begin{subfigure}[b]{0.9\textwidth}
        \centering
        \begin{tabular}{ccc}
            $N=3$ & $N=7$ & $N=9$ \\
             \includegraphics[width=0.3\textwidth]{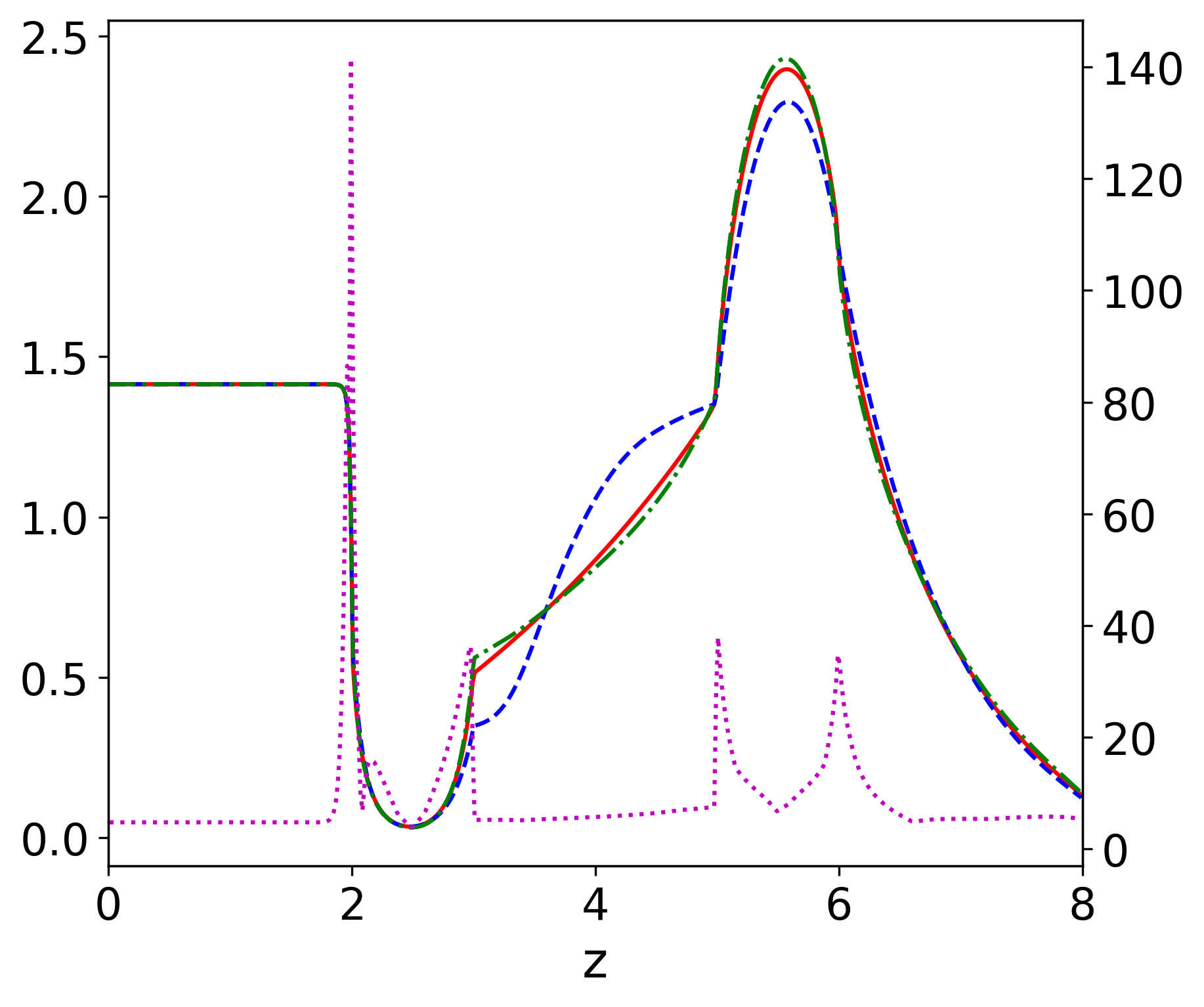}
            & \includegraphics[width=0.3\textwidth]{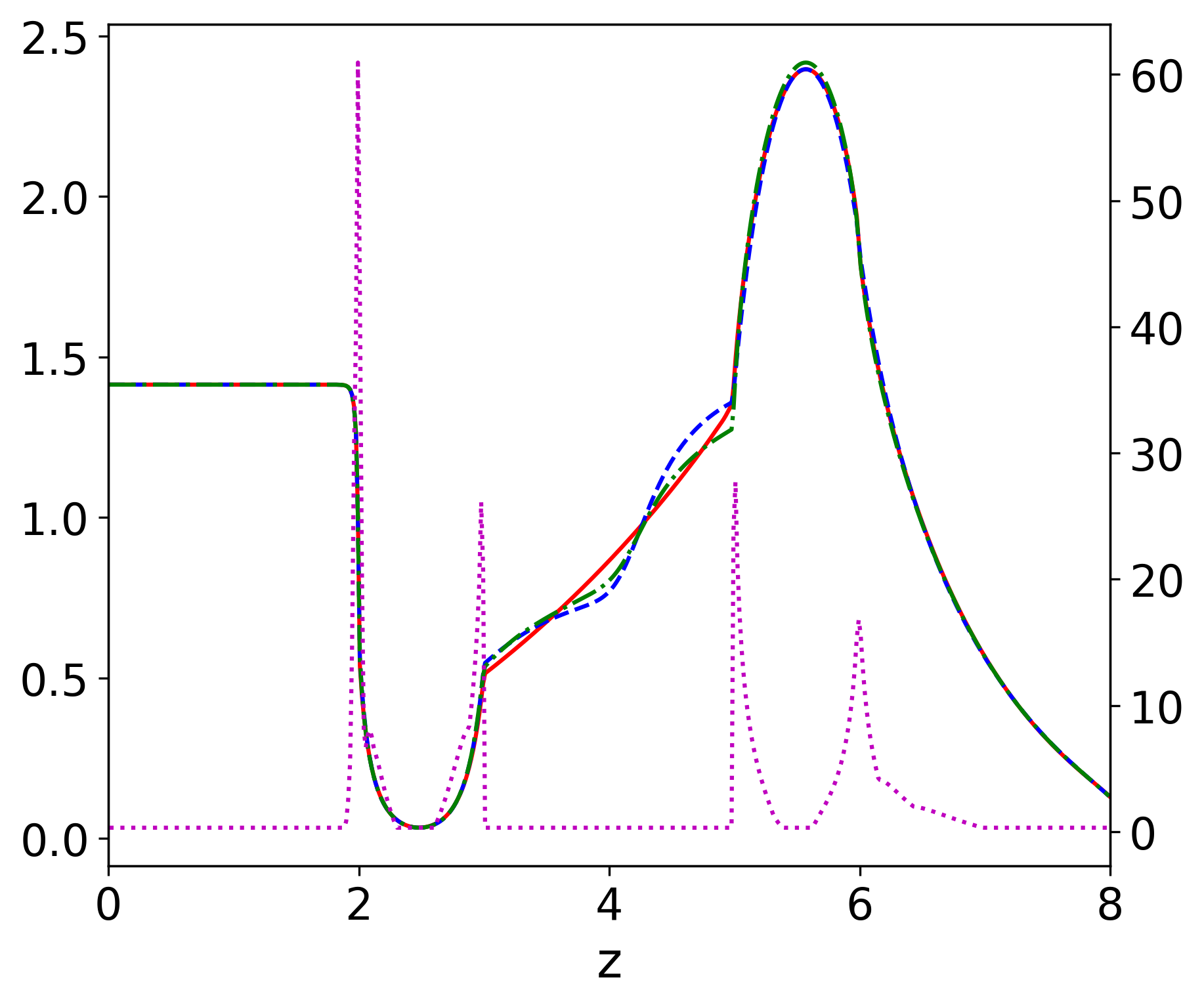}
            & \includegraphics[width=0.3\textwidth]{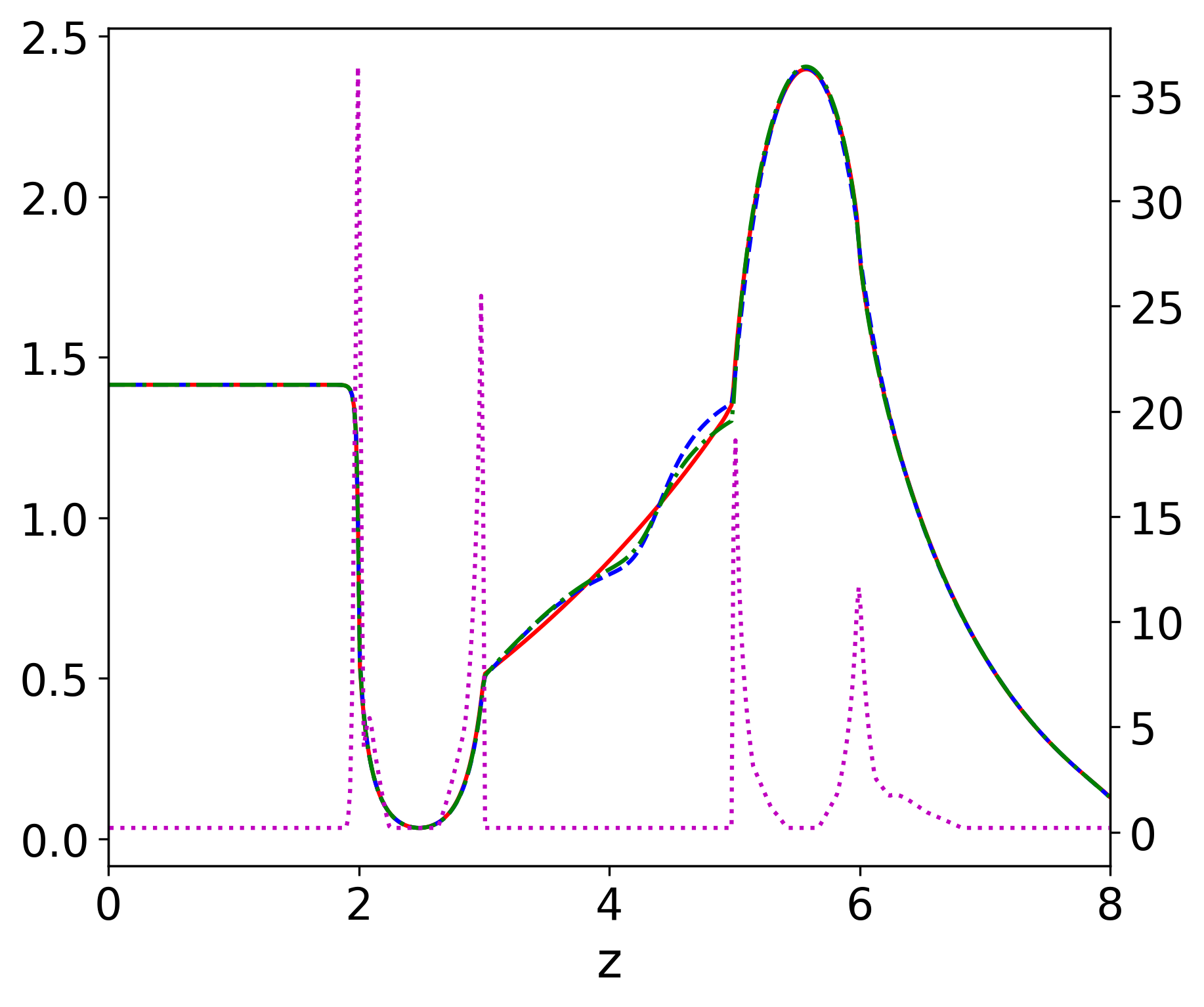}  \\
             \includegraphics[width=0.3\textwidth]{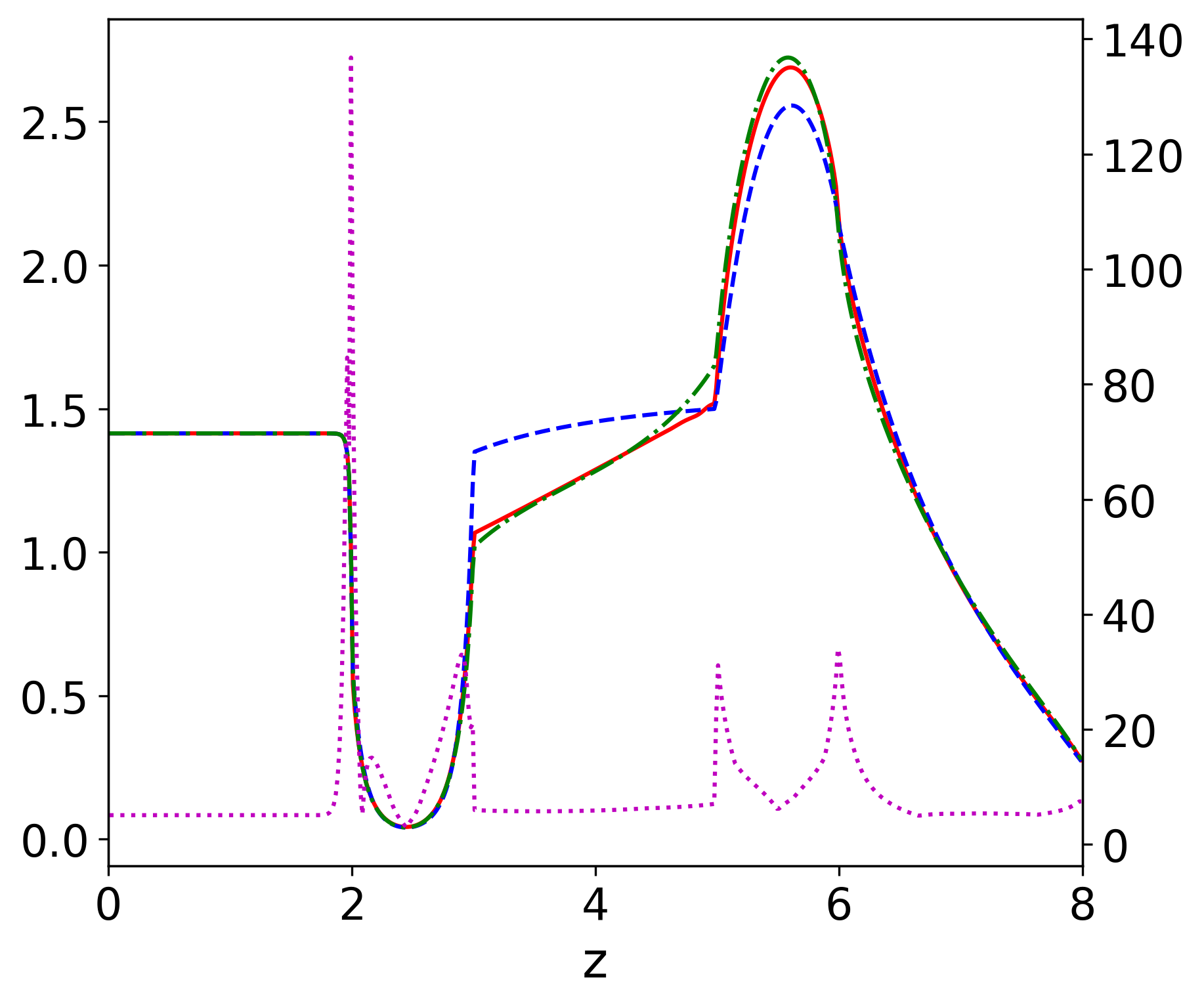}
            & \includegraphics[width=0.3\textwidth] {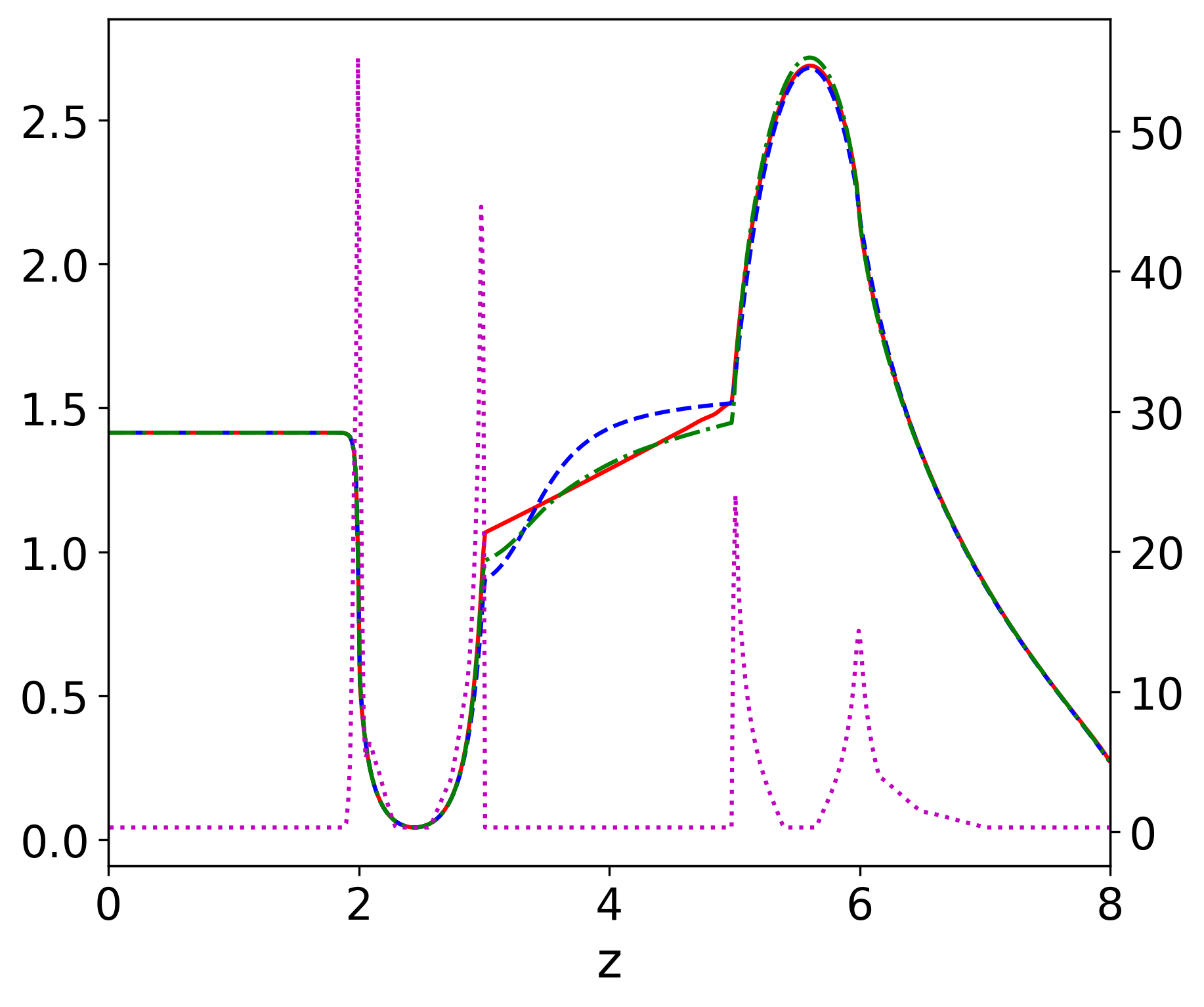}
            & \includegraphics[width=0.3\textwidth] {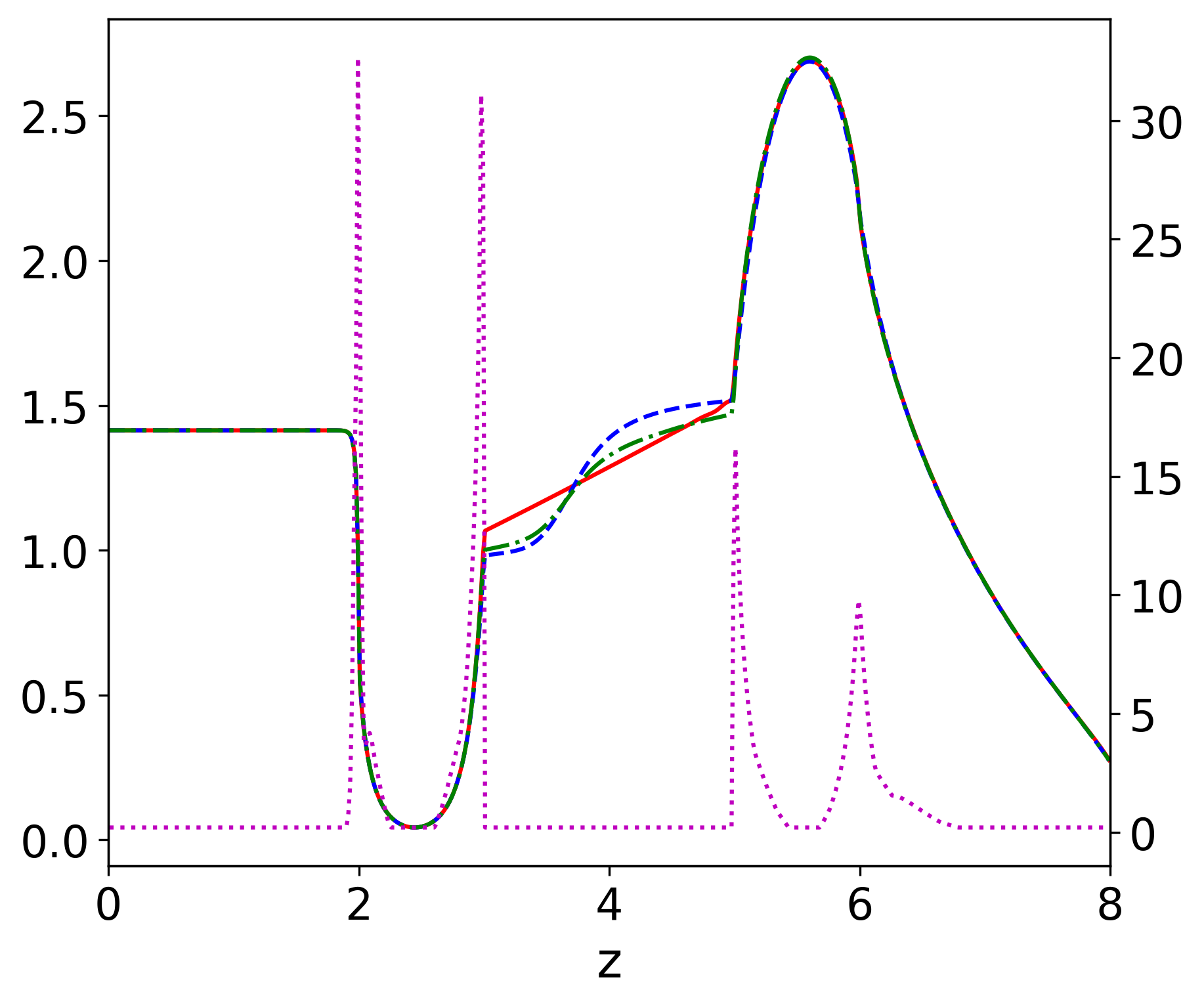} \\
            \multicolumn{3}{c}{\includegraphics[width=0.5\textwidth]{figures/legend.png}} \\
        \end{tabular}
        \caption{Particle concentrations at $\tf= 5$ (top row) and $\tf =10$ (bottom row) for $N=3,7,9$.
        The scale for the particle concentration is shown on the left and the scale for $\sig{f}$ is shown on the right in each plot.}
    \end{subfigure}

    \vspace{1em}

    \begin{subfigure}[b]{\textwidth}
        \centering
        
\resizebox{\textwidth}{!}{
        \begin{tabular}{@{}c ccc c ccc@{}}
        \toprule
        & \multicolumn{3}{c}{Neural Network ($r_N^{\text{nn}}$)} & \phantom{abc} & \multicolumn{3}{c}{Constant ($r_N^{\text{const}}$)} \\
        \cmidrule{2-4} \cmidrule{6-8}
        $\tf$ & $N=3$ & $N=7$ & $N=9$ & & $N=3$ & $N=7$ & $N=9$ \\
        \midrule
        5 & $\bm{0.224 \pm 0.001}$ & $0.593 \pm 0.000$ & $0.625 \pm 0.000$ & & $0.273 \pm 0.000$ & $\bm{0.282 \pm 0.000}$ & $\bm{0.281 \pm 0.000}$ \\
        10 & $\bm{0.289 \pm 0.030}$ & $0.438 \pm 0.000$ & $0.602 \pm 0.000$ & & $0.555 \pm 0.000$ & $\bm{0.423 \pm 0.001}$ & $\bm{0.316 \pm 0.000}$ \\
        \bottomrule
        \end{tabular}
        }
        \caption{Ratio of the \fpn relative to the P$_N$ error for the neural network filter strength ($r_N^{\text{nn}}$) and constant filter strength ($r_N^{\text{const}}$). Smaller numbers are better, and  bold denotes which trainable filter that does better for each choice of $N$ and $\tf$.}
          \end{subfigure}
     \caption{(Reed's Problem)   
   The neural network ansatz yields the smallest errors for $N=3$, while the constant filter strength yields the smallest errors for $N = 7$ and $N = 9$.}
        \label{tab:Reed's}
\end{figure}

\subsubsection{Ablation Study}

We assess the sensitivity of the neural network ansatz through {single-feature} and {leave-one-out} ablation tests on Reed’s problem with $N=3$ and $N =7$ at $\tf=5$.  In the single-feature tests, each feature is tested in isolation by setting all others to zero. 
In the leave-one-out test, one feature is removed at a time while keeping the rest. The error ratios $r_N^{\text{nn}}$ for the particle concentration are reported in Table~\ref{tab:ablation}.  The data is presented as a mean and standard deviation over ten runs.

For $N = 3$, feature $\sig{t} \bfv$ gives the smallest error ratio in the single-feature tests and when it is omitted in the leave-one-out tests, the error increases significantly, indicating the importance of this feature in the model. 
For $N = 7$, feature $\tilde{\bfA} \partial_x \bfv$ gives the smallest error ratio in the single-feature tests, while omitting this feature increases the error ratio significantly in the leave-one-out tests.

\begin{table}[H]
\centering
\resizebox{0.9\textwidth}{!}{
\begin{tabular}{l l c c c c c c c}
\toprule
$N$ & Feature & $\tilde{\bfA} \partial_x \bfv$ & $\sig{t} \bfv$ & $\sig{s} (\tilde{\bfG} - I) \bfv$ & $\tilde{\bfs}$ & None & All & Interpretation \\
\midrule
\multirow{2}{*}{3} 
 & Single-feature   & $0.374 \pm 0.016$ & $0.193 \pm 0.005$ & $0.267 \pm 0.007$ & $0.270 \pm 0.001$ & $0.272 \pm 0.001$ & $0.224 \pm 0.001$ & Only this feature included \\
 & Leave-one-out    & $0.162 \pm 0.013$ & $0.290 \pm 0.001$ & $0.302 \pm 0.014$ & $0.225 \pm 0.011$ & $0.224 \pm 0.001$ & -- & This feature omitted \\
\midrule
\multirow{2}{*}{7} 
 & Single-feature   & $0.184 \pm 0.000$ & $0.419 \pm 0.000$ & $0.323 \pm 0.000$ & $0.339 \pm 0.000$ & $0.334 \pm 0.000$ & $0.438 \pm 0.000$ & Only this feature included \\
 & Leave-one-out    & $0.417 \pm 0.000$ & $0.150 \pm 0.000$ & $0.438 \pm 0.000$ & $0.438 \pm 0.000$ & $0.438 \pm 0.000$ & -- & This feature omitted \\
\bottomrule
\end{tabular}
}
\caption{
Errors from single-feature and leave-one-out tests for Reed's problem with $N = 3$ and $N = 7$ at $\tf = 5$ are presented.
For the single feature tests, the  ``None'' column shows the error with no features (bias only), while ``All'' includes all features. 
Feature $\sig{t} \bfv$ yields the smallest relative error for $N = 3$ and feature $\tilde{\bfA} \partial_x \bfv$ yields the smallest relative error for $N = 7$. 
In the leave-one-out tests, each column indicates the omitted feature, with ``None'' referring to the full model.
Omitting feature $\tilde{\bfA} \partial_x \bfv$ results in the smallest error for $N = 3$ and omitting feature $\sig{t} \bfv$ results in the smallest error for $N = 7$.
}
\label{tab:ablation}
\end{table}

\subsection{2-D Tests} \label{sec:numerical_2D}

In geometries with no variation in the $z$ direction, the \pn and \fpn equations can be reduced to equations in two spatial variables $(x,y) \in D \subset \bbR^2$.  In this setting, the kinetic distribution $\psi = \psi(x,y,\bm{\Omega})$ is approximated by the expansion $\psifpn(x,y,\bm{\Omega}) = \tilde{\bfr}^\top (\bm{\Omega}) \bfv(x,y,t)$, where $\tilde{\bfr}: \bbS^2 \to \bbR^{(N+1)(N+2)/2}$ is a vector-valued function containing components of $\bfr$ with respect to spherical harmonic components that are even with respect to $\mu$ and $\bfv$ satisfies the 2D \fpn equations
\begin{equation}
    \p_t \bfv 
    + \tilde{\bfA}^{(1)} \p_x \bfv 
    +  \tilde{\bfA}^{(2)} \p_y \bfv  
    + \sig{a}\bfv 
    + \sig{s} \tilde{\bfG} \bfv 
    + \sig{f} \tilde{\bfF} \bfv 
    = \tilde{\bfs},
\end{equation}
with matrices $\tilde{\bfA}^{(1)}$, $\tilde{\bfA}^{(2)}$, $\tilde{\bfG}$, and $\tilde{\bfF}$ that are described in Appendix \ref{subsec:reduction-2D}.  Because the kinetic distribution $\psi$ in this case is an even function of $\mu$, moments with respect to spherical harmonics that are odd with respect to $\mu$ are identically zero and not considered. The physical cross-sections and source $\tilde{\bfs}$ are all function of $x$ and $y$, while the filter strength $\sig{f}$ depends on $x$, $y$, and $t$ through the state $\bfv$ and source $\tilde{\bfs}$.

Below we describe the data-sampling strategy, architectural details  and numerical evaluation of the data driven filter in several 2-D test cases. We test the data-driven \fpn solver for $N \in \{3,5,7,9\}$.  In both training and testing scenarios, the initial condition and source are isotropic; that is $v_\ell^m = \tilde{s}_\ell^m = 0$ for all $\ell >0$; thus we need only specify $v_0^0|_{t=0}$ and $\tilde{s}^0_0$. Moreover all components of $v$ are zero on the boundary (vacuum boundary condition).

\subsubsection{Data Sampling Strategy}\label{sec:data2D}
All simulation setups to generate training data take the domain $D=[-1,1]\times [-1,1]$.
We specify the random variable $\gamma$ that constitutes the data sampling strategy to train the data-driven filter for the 2-D test cases. 
Specifically, $\gamma$ describes the choice of initial condition for $v_0$ and sources $\tilde{s}_0$.  The initial conditions are selected from
\begin{enumerate}[label=(\alph*)]
    \item a zero centered Gaussian with standard deviation $0.1$:
    \begin{equation}\label{eq:gauss}
        v_0^0(x,y,t=0) 
        = \frac{1}{\sqrt{2 \pi \varsigma^2 }}\exp\left( -\frac{x^2 + y^2}{2 \varsigma^2} \right), \quad  \varsigma = 0.1;
    \end{equation}
    
    \item a piecewise linear function with maximum at the origin:
    \begin{equation}
        v_0^0(x,y) = \max \Big( 0, \; 1 - |x| - |y| \Big)
    \end{equation}

    \item the indicator function on the interval $[-0.2,0.2]\times [-0.2,0.2]$,
    \item the indicator function on the interval $[-0.4,0.4]\times [-0.4,0.4]$,
    \item a bump function centered at the origin:
    \begin{equation}
        v_0^0(x,y) =
        \begin{cases}
            10 \cdot \cos(2\pi r), & \text{if } \cos(2\pi r) > 0 \;\text{and}\; r < \frac{3}{8}, \\[1mm]
            0, & \text{otherwise}
        \end{cases}
        \quad
        r =  \sqrt{x^2 + y^2},
    \end{equation}

\end{enumerate}
The sources are selected from 
\begin{enumerate}[label=(\alph*)]
    \item zero
    \item the indicator function on a circle with center $z=0$ and radius $0.05$,
    \item two plates
    \begin{equation}  
    v_0^0(x,y) =
        \begin{cases}
            20, & \text{if } x \in [-0.25, 0.25] \;\text{and}\; y \in [0.7, 0.8], \\[1mm]
            40, & \text{if } x \in [-0.25, 0.25] \;\text{and}\; y \in [-0.8, -0.7], \\[1mm]
            0, & \text{otherwise,}
        \end{cases}
    \end{equation}

    \item the indicator function of a frame with width $0.1$:
    \begin{equation} 
        v_0^0(x,y) =
        \begin{cases}
            5, & \text{if } \big(|x| \in [0.6,0.8] \;\text{and}\; |y| \le 0.8\big)
            \;\text{or}\; \big(|y| \in [0.6,0.8] \;\text{and}\; |x| \le 0.8\big), \\[1mm]
            0, & \text{otherwise,}
        \end{cases}
    \end{equation}

    \item a zero centered Gaussian \eqref{eq:gauss} with standard deviation $0.03$.
\end{enumerate}
Several of the initial conditions and sources are shown in Figure \ref{fig:IC2D}.
All setups are simulated until final time $\tf=0.5$.

\begin{figure}[H]
    \centering
    \begin{subfigure}{\subfigwidth}
        \centering
        \includegraphics[width=\linewidth]{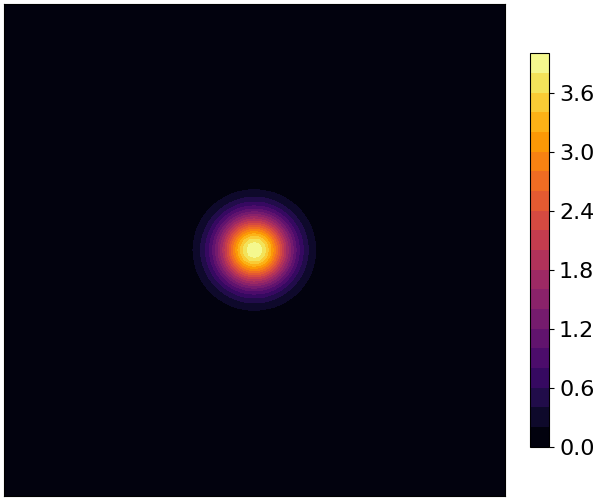}
        \caption{Gaussian}
        \label{fig:2D_gauss}
    \end{subfigure}
    \begin{subfigure}{\subfigwidth}
        \centering
        \includegraphics[width=\linewidth]{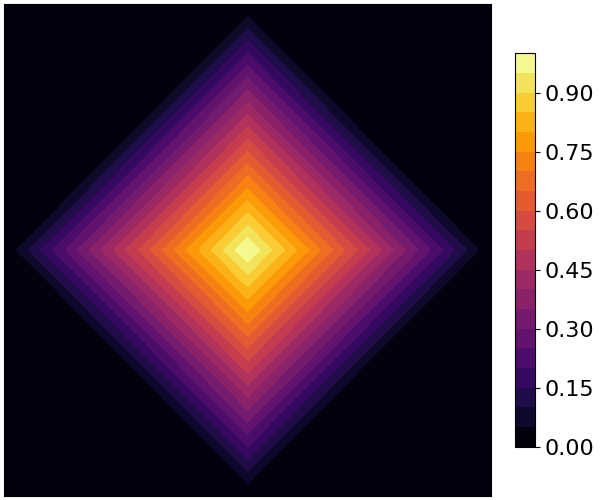}
        \caption{Piece-wise linear}
        \label{fig:2D_hat}
    \end{subfigure}
    \begin{subfigure}{\subfigwidth}
        \centering
        \includegraphics[width=\linewidth]{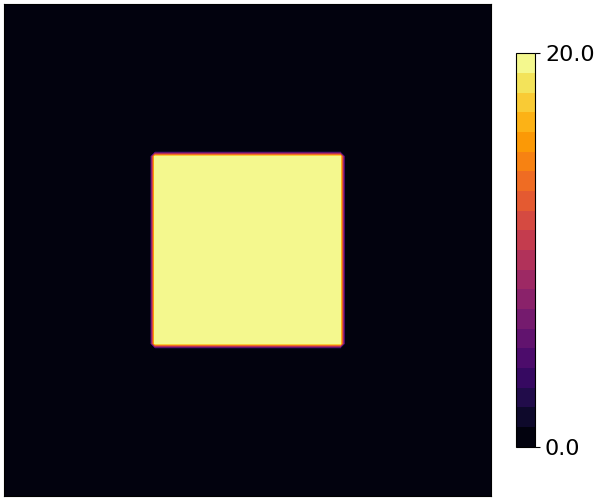}
        \caption{Piece-wise constant}
        \label{fig:2D_step}
    \end{subfigure}
    
    \begin{subfigure}{\subfigwidth}
            \centering
            \includegraphics[width=\linewidth]{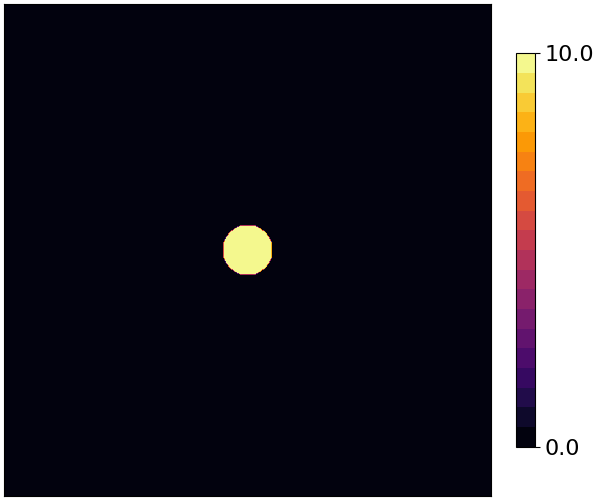}
            \caption{Pulse}
            \label{fig:2D_pulse}
    \end{subfigure}
    \begin{subfigure}{\subfigwidth}
        \centering
        \includegraphics[width=\linewidth]{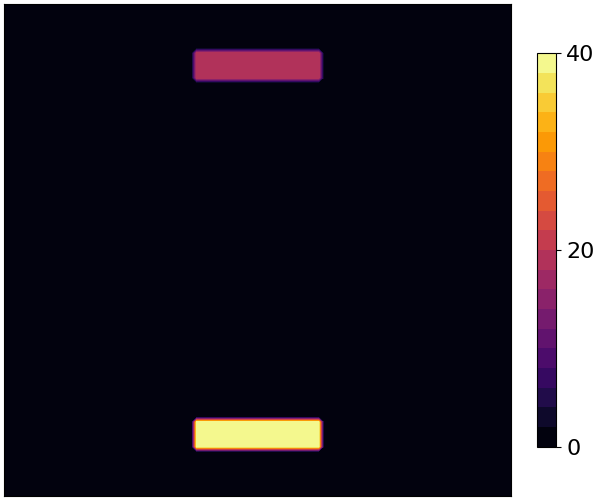}
        \caption{Two Plates}
        \label{fig:2D_plates}
    \end{subfigure}
    \begin{subfigure}{\subfigwidth}
        \centering
        \includegraphics[width=\linewidth]{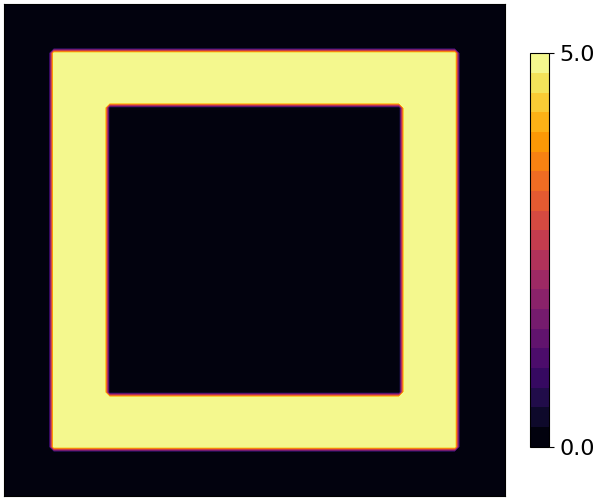}
        \caption{Frame}
        \label{fig:2D_frame}
    \end{subfigure}
   \caption{Figures \subr{fig:2D_gauss} - \subr{fig:2D_step} show initial conditions for training data.  While the Gaussian is smooth, steep gradients are challenging to resolve with the \pn method.  The irregularities in piece-wise linear and piece-wise constant functions bring their own challenges.  The other initial conditions that are not shown are a bump function and a narrow piece-wise constant function.
   Figures \subr{fig:2D_pulse} - \subr{fig:2D_frame} show isotropic sources used to train the model.  The other training sources that are not shown are a Gaussian and a zero function.}
    \label{fig:IC2D}
\end{figure} 

Furthermore, $\gamma$ describes the values for the scattering and absorption cross section, i.e.  $\sig{s}$ and $\sig{a}$ which are distributed as follows
\[
\sig{s} \sim \cU(0, 1), \quad \sig{a} \sim \cU(0, 1 - \sig{s}),
\]
where $\cU$ denotes the uniform distribution in an interval.

\subsubsection{Neural Network Architecture and Training}

We train filters $\Sig{f}$ for $N \in \{3,5,7,9\}$, where the preprocessed features are put into a feedforward neural network. The network is fully connected with $\operatorname{tanh}$ activations. Residual connections and LayerNorm (see the PyTorch documentation\cite{paszke2017automatic}) are applied after each hidden layer to stabilize training. A final ReLU-activated output layer maps to a single scalar per input sample. 
The first layer is equipped with a batch-normalization layer \cite{ioffe2015batchnormalizationacceleratingdeep} that estimates mean and variance of the features across the grid and batch. 
The weights and biases of the hidden and output layers, $\bm{W}^{(i)}$ and $\bm{b}^{(i)}$, are initialized from uniform distributions. For the first layer, the initialization depends on the number of input features $n_{\text{f}}$:  
\begin{equation}
    \bm{W}^{(1)}, \, \bm{b}^{(1)} \sim \mathcal{U}\!\left(-n_{\text{f}}^{-1/2}, \, n_{\text{f}}^{-1/2}\right).
\end{equation}
For subsequent layers $i \in \{2, \dots, 5\}$, the initialization is based on the number of hidden neurons $n_{\text{h}_i}$ in layer $i$:  
\begin{equation}
    \bm{W}^{(i)}, \, \bm{b}^{(i)} \sim \mathcal{U}\!\left(-n_{\text{h}_i}^{-1/2}, \, n_{\text{h}_i}^{-1/2}\right).
\end{equation}
Table \ref{tab:2D_params} shows the network and training parameters. 
The networks are trained using the Adam optimizer \cite{Kingma2017} in PyTorch.

\begin{table}[H]
\centering
\resizebox{\textwidth}{!}{
\begin{tabular}{ccccccc}
\toprule
input features $(n_{\text{f}})$ & hidden layers & hidden width $(n_{\text{h}_i})$ & learning rate & batch size & number of epochs  & weight decay\\
\midrule
 2N + 4 & 4 & N+2 &  1e-2 & 5 & 200 & 0\\
\bottomrule
\end{tabular}
}
\caption{The table shows parameters used in 2-D training. The learning rate is chosen by an initial hyperparameter search.}
 \label{tab:2D_params}
\end{table}

In each iteration of the training process, we draw five realizations of $\gamma$, uniformly sampling and pairing initial conditions and sources described in Section \ref{sec:data2D}. 
Each pairing is assigned independently sampled realizations of $\sig{a}$ and $\sig{s}$. 
Both \fpn and the reference $P_{37}$ system use  grid spacings of $\dx = \frac{1}{50}$ and $\Delta t = \frac{1}{100}$ for the space-time discretization.

\subsubsection{Test Case: Line-source}

The line-source problem \cite{Brunner2002} presents a challenging benchmark due to the presence of steep gradients in the scalar flux and the emergence of large high-order moments in the angular flux. Physically, the initial condition corresponds to a highly localized source at the center of the domain, which is typically modeled as a Dirac delta function. In our setup, we approximate this source using a smooth Gaussian \eqref{eq:gauss} centered at $(x, y) = (0,0)$ with standard deviation $\varsigma = 0.03$.

We consider a sequence of polynomial orders $N \in \{3, 5, 7, 9\}$ and compare the classical \pn approximation with the \fpn approximation. Representative solution profiles of the particle density for different values of $N$ are shown in Figure~\ref{fig:ls}.  The \fpn method exhibits significantly reduced oscillations in the scalar flux, particularly near the source, where Gibbs-like artifacts dominate in the \pn approximation. For smaller $N$, the neural network filter performs marginally better than the constant filter.  However, for moderate $N$, the benefit of the neural network filter over the constant is significant.

\begin{figure}[H]
\centering
    \begin{subfigure}[c]{\subfigwidth}  
        \includegraphics[width=\textwidth] {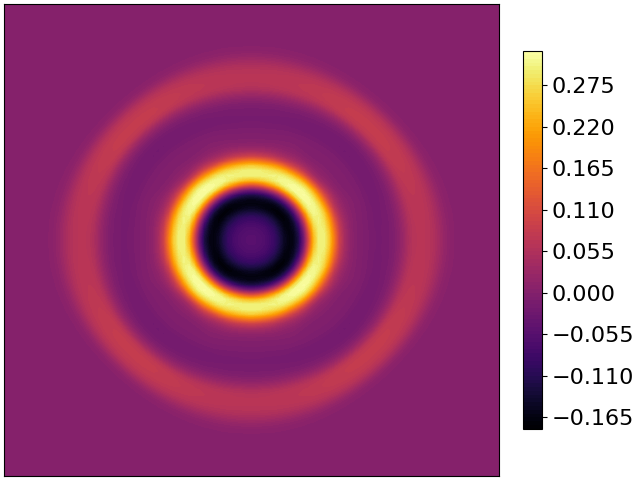}
        \caption{P$_3$}
        \label{ls:P3}
    \end{subfigure} 
    \begin{subfigure}[c]{\subfigwidth}  
        \includegraphics[width=\textwidth] {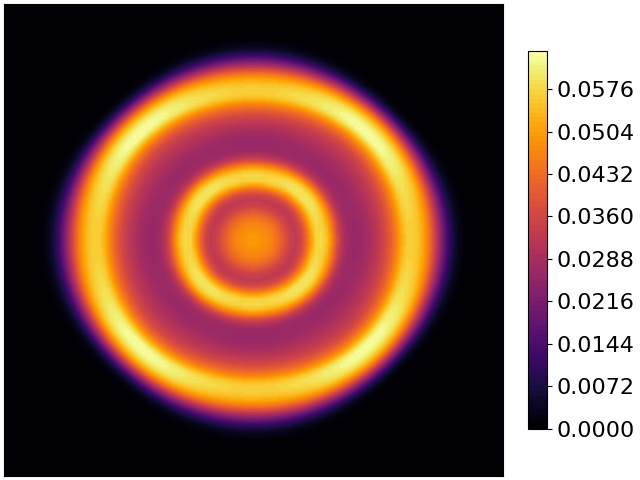}
        \caption{FP$_3$}
         \label{ls:FP3}
    \end{subfigure}
    \begin{subfigure}[c]{\subfigwidth}  
        \includegraphics[width=\textwidth] {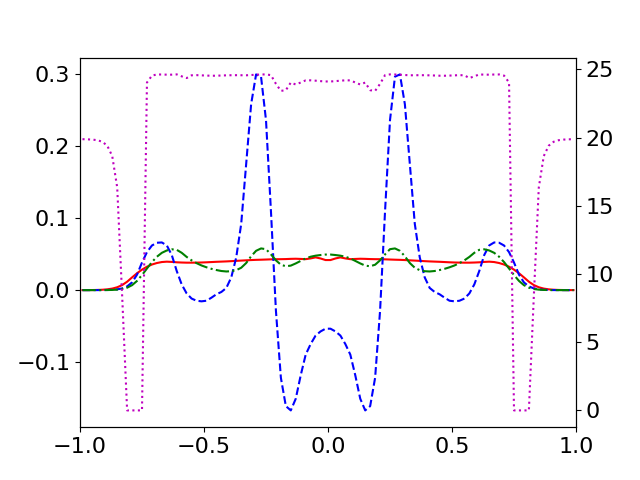}
        \caption{slice 0}
         \label{ls:slice0_P3}
    \end{subfigure}
    \begin{subfigure}[c]{\subfigwidth}  
        \includegraphics[width=\textwidth] {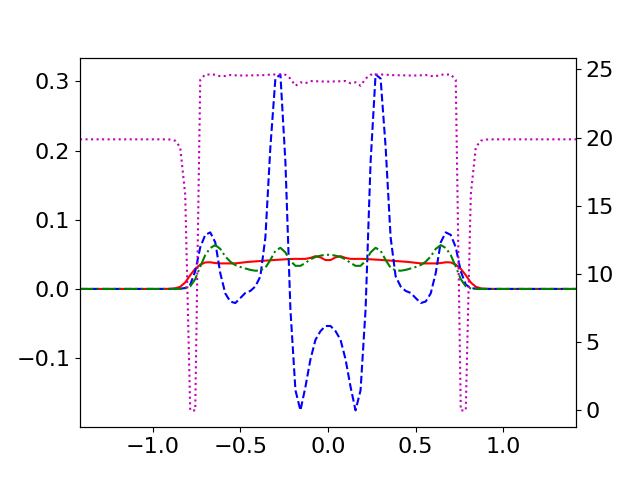}
        \caption{slice 45}
        \label{ls:slice45_P3}
    \end{subfigure}
    \centering
    \begin{subfigure}[c]{\subfigwidth}  
        \includegraphics[width=\textwidth] {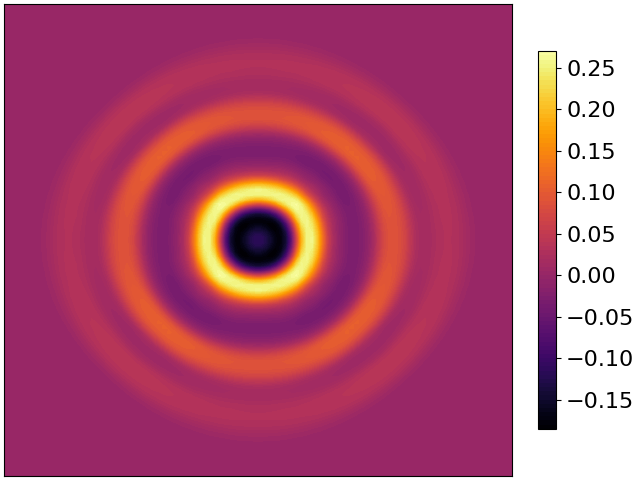}
        \caption{P$_5$}
         \label{ls:P5}
    \end{subfigure} 
    \begin{subfigure}[c]{\subfigwidth}  
        \includegraphics[width=\textwidth] {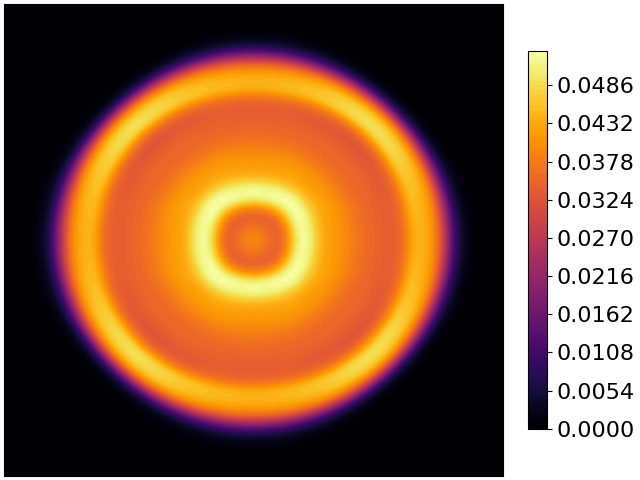}
        \caption{FP$_5$}
         \label{ls:FP5}
    \end{subfigure}
    \begin{subfigure}[c]{\subfigwidth}  
        \includegraphics[width=\textwidth] {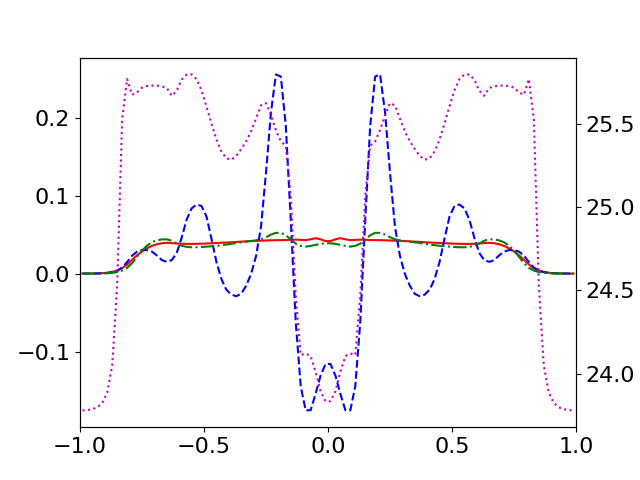}
        \caption{slice 0}
         \label{ls:slice0_P5}
    \end{subfigure}
    \begin{subfigure}[c]{\subfigwidth}  
        \includegraphics[width=\textwidth] {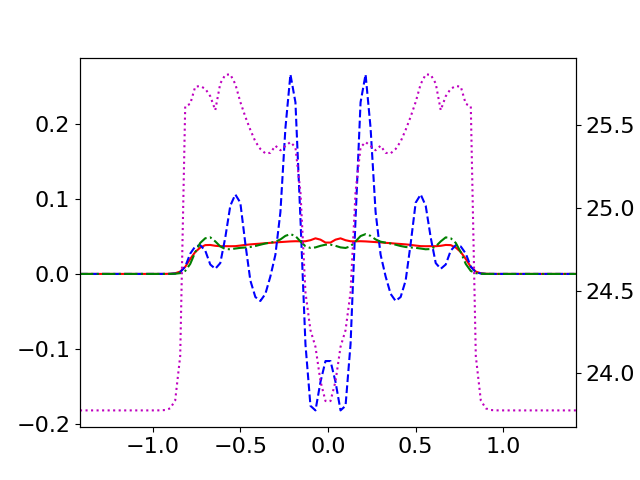}
        \caption{slice 45}
        \label{ls:slice45_P5}
    \end{subfigure}
    \centering
    \begin{subfigure}[c]{\subfigwidth}  
        \includegraphics[width=\textwidth] {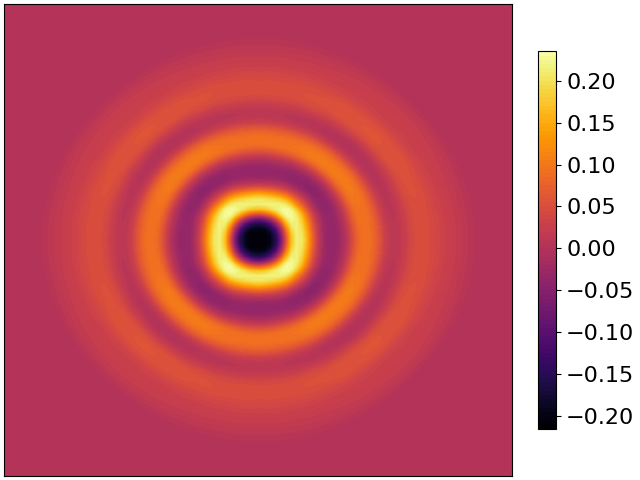}
        \caption{P$_7$}
         \label{ls:P7}
    \end{subfigure} 
    \begin{subfigure}[c]{\subfigwidth}  
        \includegraphics[width=\textwidth] {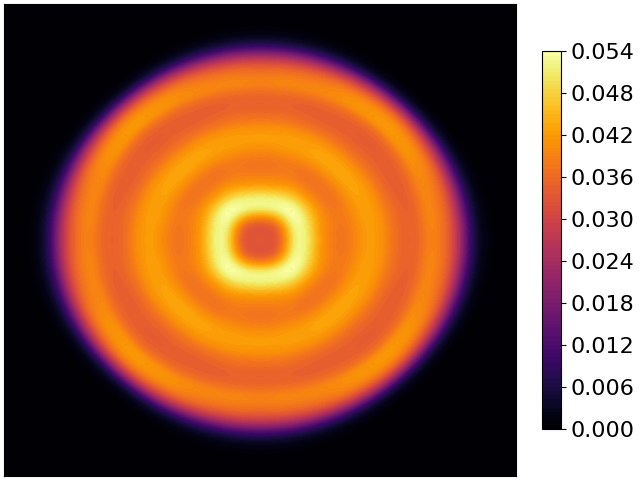}
        \caption{FP$_7$}
        \label{ls:FP7}
    \end{subfigure}
    \begin{subfigure}[c]{\subfigwidth}  
        \includegraphics[width=\textwidth] {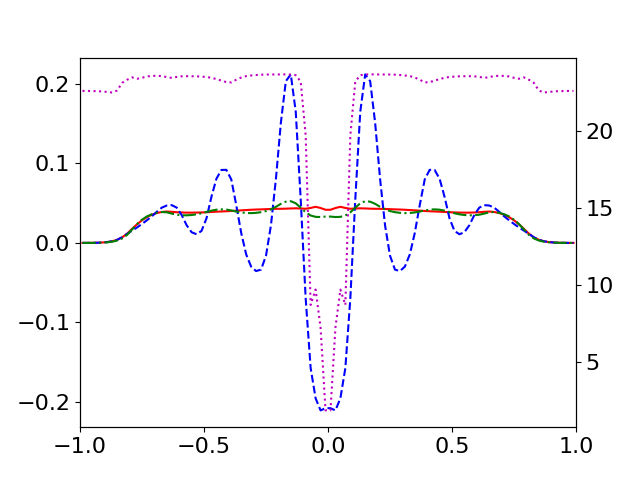}
        \caption{slice 0}
         \label{ls:slice0_P7}
    \end{subfigure}
    \begin{subfigure}[c]{\subfigwidth}  
        \includegraphics[width=\textwidth] {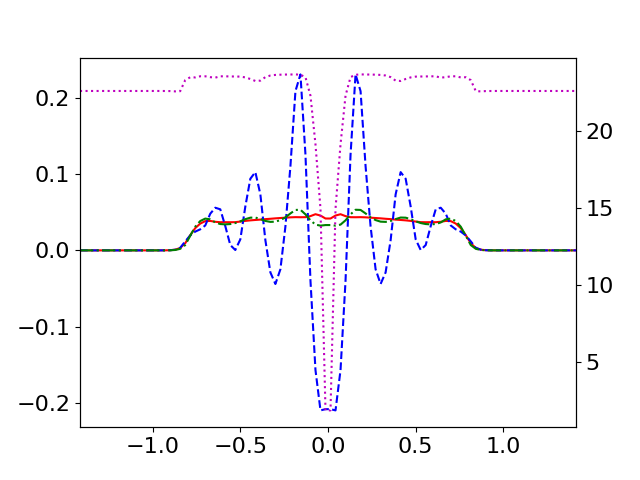}
        \caption{slice 45}
        \label{ls:slice45_P7}
    \end{subfigure}
    \centering
    \begin{subfigure}[c]{\subfigwidth}  
        \includegraphics[width=\textwidth] {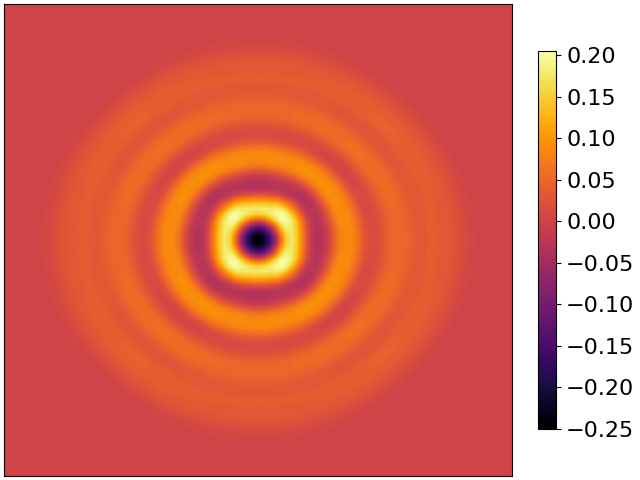}
        \caption{P$_9$}
         \label{ls:P9}
    \end{subfigure} 
    \begin{subfigure}[c]{\subfigwidth}  
        \includegraphics[width=\textwidth] {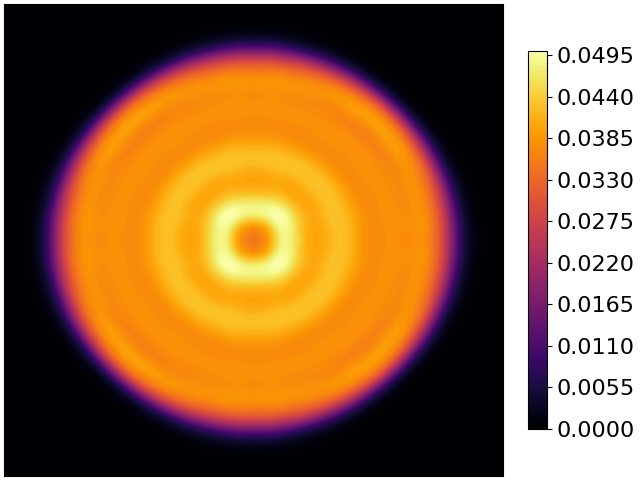}
        \caption{FP$_9$}
        \label{ls:FP9}
    \end{subfigure}
    \begin{subfigure}[c]{\subfigwidth}  
        \includegraphics[width=\textwidth] {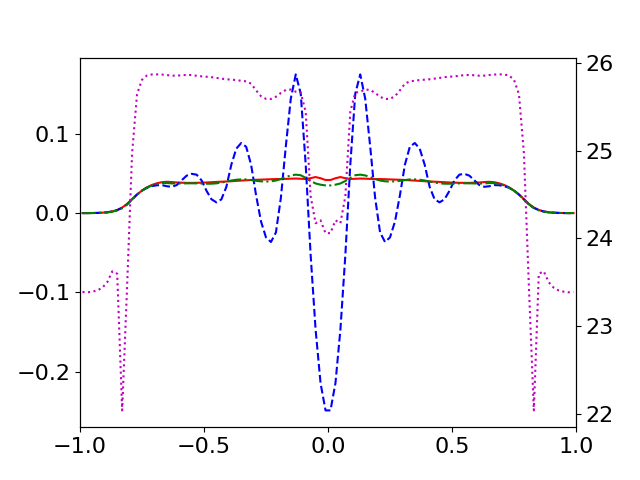}
        \caption{slice 0}
         \label{ls:slice0_P9}
    \end{subfigure}
    \begin{subfigure}[c]{\subfigwidth}  
        \includegraphics[width=\textwidth] {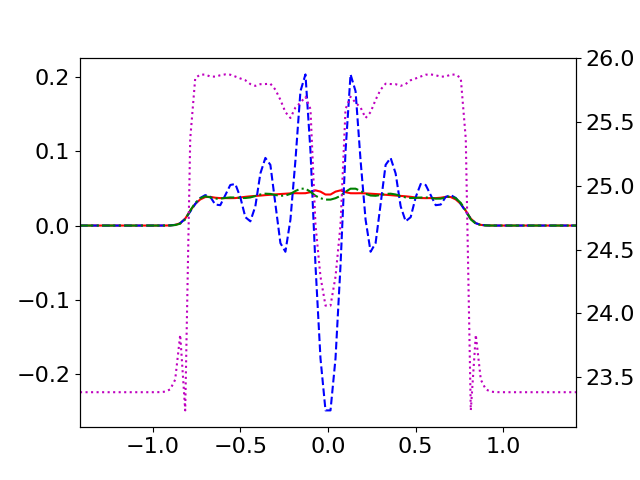}
        \caption{slice 45}
        \label{ls:slice45_P9}
    \end{subfigure} 

    \begin{subfigure}[c]{\linewidth} 
    \centering
    \vspace{1.0em}
        {\includegraphics[width=0.5\textwidth]{figures/legend.png}}
    \end{subfigure}
    \vspace{0.5em}
    \centering
  
    \centering   
    \begin{tabular}{@{}c cccc ccc@{}}
        \toprule
        $N$ & Plot Ref & $e_N$ & $e_N^{\text{nn}}$ & $e_N^{\text{const}}$ & $r_N^{\text{nn}}$ & $r_N^{\text{const}}$ \\
        \midrule
        3 & \subr{ls:P3} - \subr{ls:slice45_P3} 
        & 2.3141 & $\bm{0.3276}$ & 0.3859 & $\bm{0.1416}$ & 0.1668 \\ 
        5 & \subr{ls:P5} - \subr{ls:slice45_P5} 
        & 1.7282 & $\bm{0.1333}$ & 0.1389 & $\bm{0.0771}$ & 0.0804 \\ 
        7 & \subr{ls:P7} - \subr{ls:slice45_P7} 
        & 1.3168 & $\bm{0.0822}$ & 0.1599 & $\bm{0.0624}$ & 0.1214 \\
        9 & \subr{ls:P9} - \subr{ls:slice45_P9} 
        & 0.9664 & $\bm{0.0408}$ & 0.1632 & $\bm{0.0422}$ & 0.1688 \\ 
        \bottomrule 
    \end{tabular}
     \caption{(Line-Source) The figure shows particle concentrations for \pn and \fpn at $\tf = 0.75$. 
     The first row shows the $N=3$ profiles and increases sequentially to $N=9$ in the bottom row.  
     From left to right, the plots show \pn, \fpn, a slice at $y = 0$ and a slice at 45 degrees. 
     In the plots of the 1-D slices, the scale for particle concentration is shown on the left axis and the scale for $\sig{f}$ is shown on the right axis.
     The horizontal axis measures the signed distance from the origin.
     As $N$ increases, the error of the \fpn solution decreases and the oscillations associated with the \pn approximation are significantly dampened. 
     The values shown in bold in the table represent the smallest errors between the neural network and constant filter solutions.
     }
      \label{fig:ls}
\end{figure}

\subsubsection{Test Case: Lattice}

The lattice problem \cite{Brunner2002} poses a challenging benchmark due to the presence of sharp discontinuities in the material cross-sections. The layout of the material and a $P_{37}$ reference solution are given in Figure~\ref{fig:latt_layout}. Numerical simulations are performed using a uniform spatial resolution of $\Delta x = \Delta y = \frac{1}{40}$ and a temporal step size of $\Delta t = \frac{1}{160}$.

\begin{figure}[H]
    \centering
    \begin{minipage}{0.3\textwidth}
        \centering
        \begin{tikzpicture}[scale=0.55] 
      \def\numx{7}
      \def\numy{7}
      
      \foreach \i in {0,...,6} {
        \foreach \j in {0,...,6} {
          \pgfmathsetmacro{\xleft}{\j}
          \pgfmathsetmacro{\ybottom}{\i}

          \def\fillcolor{white}

          \pgfmathsetmacro{\xedge}{\j}
          \pgfmathsetmacro{\yedge}{\i}

          \ifnum \i=3
            \ifnum \j=3
              \def\fillcolor{orange}
            \fi
          \fi
    
          \ifnum \i=1
            \ifnum \j=1 \def\fillcolor{black} \fi
            \ifnum \j=3 \def\fillcolor{black} \fi
            \ifnum \j=5 \def\fillcolor{black} \fi
          \fi
          \ifnum \i=2
            \ifnum \j=2 \def\fillcolor{black} \fi
            \ifnum \j=4 \def\fillcolor{black} \fi
          \fi
          \ifnum \i=3
            \ifnum \j=1 \def\fillcolor{black} \fi
            \ifnum \j=5 \def\fillcolor{black} \fi
          \fi
          \ifnum \i=4
            \ifnum \j=2 \def\fillcolor{black} \fi
            \ifnum \j=4 \def\fillcolor{black} \fi
          \fi
          \ifnum \i=5
            \ifnum \j=1 \def\fillcolor{black} \fi
            \ifnum \j=3 \def\fillcolor{black} \fi
            \ifnum \j=5 \def\fillcolor{black} \fi
          \fi

        \ifnum \i=5
          \ifnum \j=3
            \def\fillcolor{white}
          \fi
        \fi
          \filldraw[fill=\fillcolor, draw=black] (\xleft,\ybottom) rectangle ++(1,1);
        }
      }
    \end{tikzpicture}
    \end{minipage}  
    \begin{minipage}{0.3\textwidth}
        \centering
        \begin{tabular}{c c c c}
            \toprule
            Region & $\sig{t}$ & $\sig{s}$ & $S$\\
            \midrule
            white  & 1  & 1 & 0 \\
            orange & 10 & 0 & 1 \\
            black  & 10 & 0 & 0 \\
            \bottomrule
        \end{tabular}
    \end{minipage}
    \begin{minipage}{0.3\textwidth}
        \centering
        \includegraphics[width=1\textwidth] {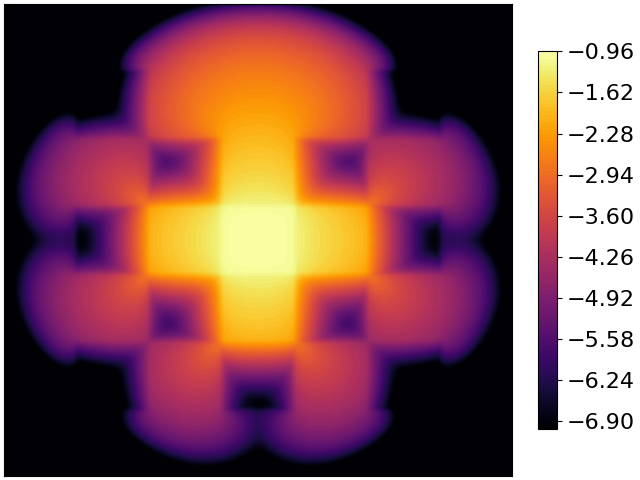}
    \end{minipage}
    
    \caption{(Lattice) The left plot shows the geometric layout of the lattice problem.  
    The spatial domain is $[0,7] \times [0,7]$ and each square has an area of 1. 
    The intensity is initially zero everywhere and  radiation is introduced by an isotropic source in the orange region.  
    The orange and black regions are purely absorbing, while the white regions are purely scattering. 
    The boundary conditions are vacuum. 
    The values of the material cross-sections and isotropic source are shown in the center table.  
    The right figure shows the $P_{37}$ solution at $\tf = 3.2$ with $\log_{10}$ scaling.}   
        \label{fig:latt_layout}
\end{figure}
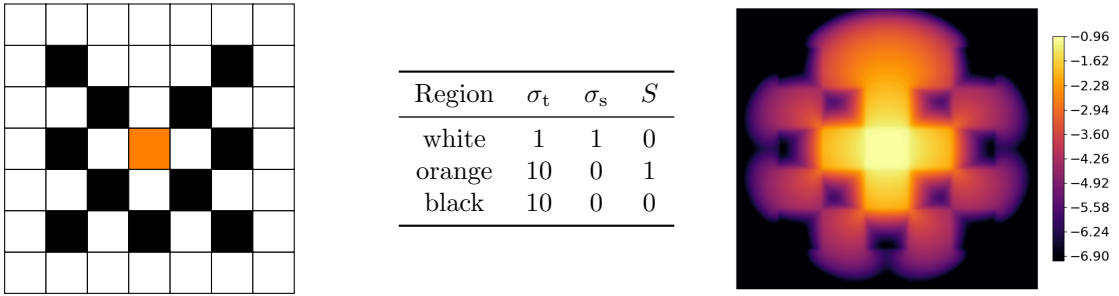

Representative solution profiles for both the \pn and \fpn methods at $\tf = 3.2$ are shown in Figure~\ref{fig:latt}.  The profiles show the improved resolution of the filtered solutions in regions with steep gradients and material heterogeneities.  The behavior of the learned filter strength is further illustrated in Figure~\ref{fig:latt_sigf}, which shows the filter strength for $N=5$ at both $\tf = 1.6$ and $\tf = 3.2$. The filter accurately identifies regions of high anisotropy and sharp features, thereby enabling more accurate and efficient propagation of radiative transport features throughout the domain.

Table~\ref{tab:latt} reports the $L^2$ errors and relative errors at times $\tf = 1.6$ and $\tf = 3.2$ for varying angular resolutions. The boldface entries highlight the more accurate result between the neural-network-based filter and the constant filter baseline. Both filter provide substantial accuracy gains, with the neural network performing slightly better in each case.

\begin{figure}[H]
\centering
    \begin{subfigure}[c]{0.32\linewidth}
        \includegraphics[width=\textwidth] {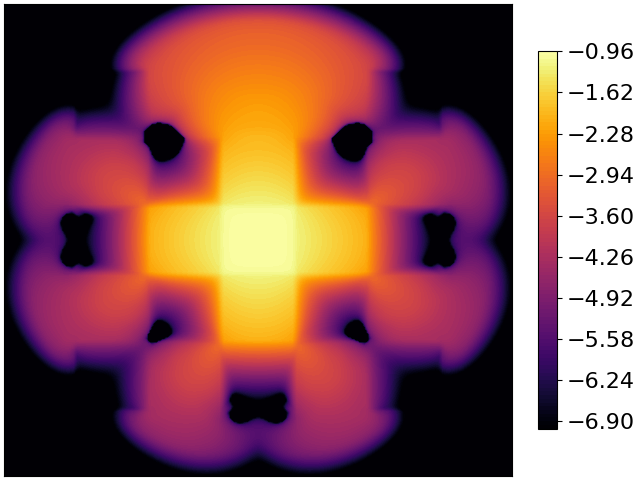}
        \caption{\pn[3]}
        \label{latt:P3}
    \end{subfigure}
    \begin{subfigure}[c]{0.32\linewidth}
        \includegraphics[width=\textwidth] {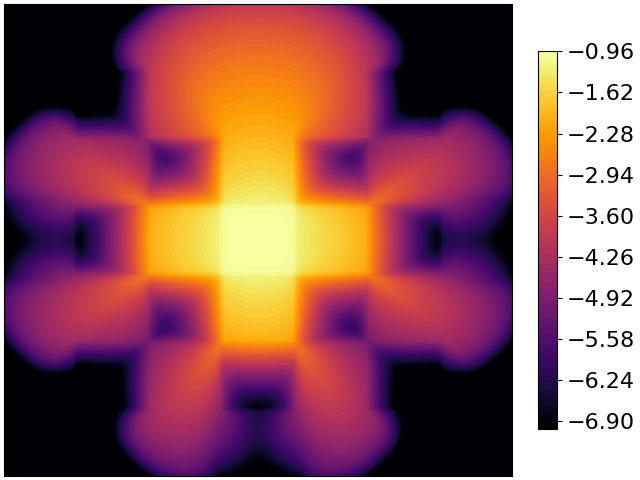}
                \caption{FP$_3$}
        \label{latt:FP3}
    \end{subfigure}
    
        \begin{subfigure}[c]{0.32\linewidth}
        \includegraphics[width=\textwidth] {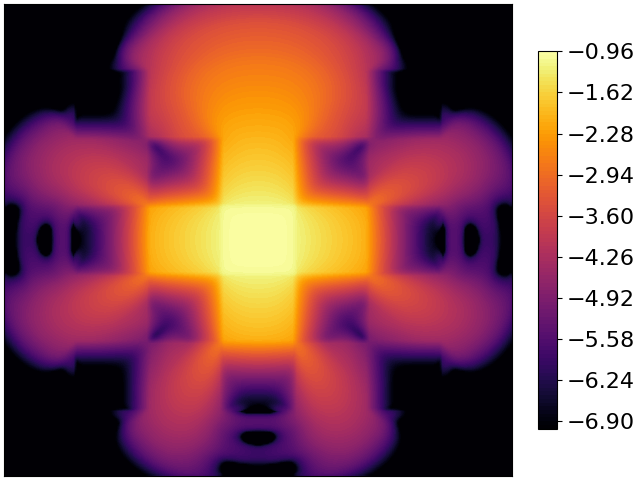}
        \caption{P$_9$}
        \label{latt:P9}
    \end{subfigure}
    \begin{subfigure}[c]{0.32\linewidth}
        \includegraphics[width=\textwidth] {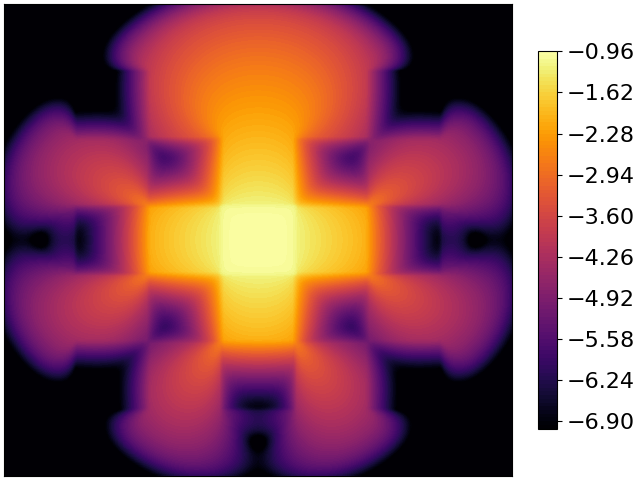}
        \caption{FP$_9$}
        \label{latt:FP9}
    \end{subfigure}
    \caption{(Lattice) Particle concentrations at $\tf = 3.2$ for the \pn and \fpn methods are shown. 
The corresponding errors with respect to the \pn[37] solution are reported in Table~\ref{tab:latt}}. 
    \label{fig:latt}
\end{figure}

\begin{figure}[H]
\centering
        \begin{subfigure}[c]{0.36\linewidth}  
        \includegraphics[width=\textwidth] {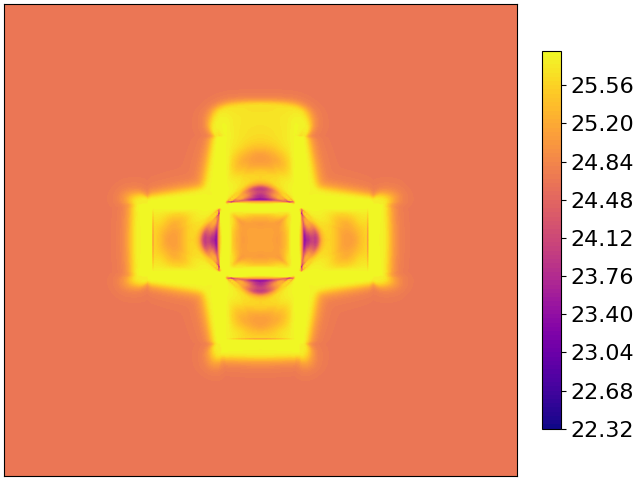}
    \end{subfigure} 
    \begin{subfigure}[c]{0.36\linewidth}  
        \includegraphics[width=\textwidth] {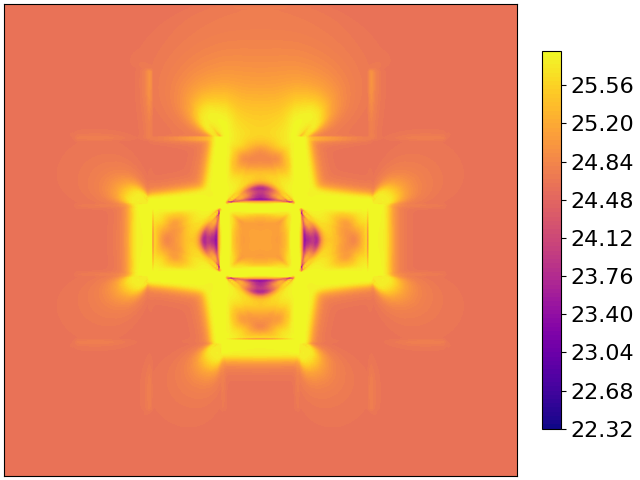}
    \end{subfigure}
    \caption{(Lattice) The figure shows the filter strength for $N=9$ at $t=1.6$ and $t=3.2$. The filter detects the geometric features of the test case domain where the solution has large gradients. }
    \label{fig:latt_sigf}
\end{figure}

\begin{table}[H]
\centering
\resizebox{0.9\textwidth}{!}{
\begin{tabular}{@{}c c ccc cc c ccc cc@{}}
\toprule
& & \multicolumn{5}{c}{\bm{$\tf = 1.6$}} & \phantom{abc} & \multicolumn{5}{c}{\bm{$\tf = 3.2$}} \\
\cmidrule{3-7} \cmidrule{9-13}
$N$ & Plot Ref 
& $e_N$ & $e_N^{\text{nn}}$ & $e_N^{\text{const}}$ & $r_N^{\text{nn}}$ & $r_N^{\text{const}}$ 
& & $e_N$ & $e_N^{\text{nn}}$ & $e_N^{\text{const}}$ & $r_N^{\text{nn}}$ & $r_N^{\text{const}}$ \\
\midrule  
3 & \subr{latt:P3} - \subr{latt:FP3} 
  & 0.0426 & $\bm{0.0126}$ & 0.0146 & $\bm{0.2959}$ &  0.3435   
  & & 0.0274 & $\bm{0.0127}$ & 0.0129 &  $\bm{0.4642}$ &  0.4695 \\
5 & - 
  & 0.0225 & $\bm{0.0063}$ & 0.0066  & $\bm{0.2792}$ &   0.2905 
  & & 0.0156 & $\bm{0.0065}$ & 0.0068  & $\bm{0.4182}$ & 0.4358\\ 
7 & - 
  &0.0140 & $\bm{0.0038}$ & 0.0042  & $\bm{0.2728}$ &   0.3021 
  & & 0.0102 & $\bm{0.0040}$ & 0.0042 & $\bm{0.3913}$ & 0.4096 \\  
9 & \subr{latt:P9} - \subr{latt:FP9} 
  & 0.0097 & $\bm{0.0029}$ & 0.0031  & $\bm{0.2953}$ & 0.3160 
  & & 0.0072 & $\bm{0.0029}$ & 0.0029 & $\bm{0.4090}$ & 0.4103 \\
\bottomrule
\end{tabular}
}
\caption{(Lattice) Errors and relative errors for the lattice problem at $t=1.6$ and $t=3.2$. Bold values indicate the smaller error between neural network and constant filters. Plot references for $N=3$ and $N=9$ are shown in the second column; solution profiles are given in Figure~\ref{fig:latt}.} 
\label{tab:latt}
\end{table}

\section{Conclusion} \label{conclusion}

In this work, we establish a strategy for adaptively filtering in \fpn based on the local state of the system.  We model the filter strength as a neural network that depends on features that are terms in the original \pn equations.  These features are pre-processed to ensure rotational invariance.

We train 1-D and 2-D models for several values of $N$ and compare the neural network based filtered solutions with those obtained using a constant filter strength that is trained with the same data.   
The resulting low-resolution \fpn solutions yield significant improvements in quality at a comparable cost to the original \pn approximation.  
The training data can be generated from simple initial conditions and cross-section configurations. 
Naturally, the quality of the training data is crucial; the corresponding \pn solutions should be representative of typical behaviors associated with irregular solutions. 
We train the model with problems evolved until final time $\tf =0.5$, which provides a reasonable balance between computational cost and sufficient training.

Across all test cases, filtering reduces the error with respect to the original \pn approximation, regardless of whether the filter strength is a constant or determined by a neural network ansatz.   In 1-D tests, this reduction can vary from 
$10-90\%$.  
The neural network filter outperforms the constant filter in several cases, but under-performs in others.  
We observe that the constant filter outperforms the neural network filter for problems with streaming regimes (when the material is vacuum).  
In 2-D tests, the error reduction can vary from $70-96\%$. 
The neural network filter reduces the error more than the constant filter in all test cases and in some cases, the error reduction is significant.

This strategy is flexible and can accommodate various forms of the filter strength ansatz, including a constant form, which offers the simplest and often most practical implementation.  
However, the constant form is limited by construction, as it does not adapt with the solution over time and space.
This limitation necessitates tuning the filter based on the known (or estimated) regularity of the solution and the final time.  
For short-time simulations, a large filter strength may improve the accuracy of the approximation significantly, but that same filter strength applied to a long-time simulation over-smooths the solution and accuracy may be lost.
In contrast, the neural network ansatz does not have this limitation, since the filter strength evolves dynamically with the state.

In future work, it would be beneficial to better understand and address the shortcomings of the neural network models in 1-D test cases, which under-perform the constant filter models in several cases.  In application spaces, we believe that improvements could be made with a dedicated set of training problems and hyperparameter optimization strategies.  In addition, we conjecture that filtering could include extended to other physically significant problems like thermal radiative transfer \cite{plumridge2025}.

\printbibliography

@misc{ioffe2015batchnormalizationacceleratingdeep,
      title={Batch Normalization: Accelerating Deep Network Training by Reducing Internal Covariate Shift}, 
      author={Sergey Ioffe and Christian Szegedy},
      year={2015},
      eprint={1502.03167},
      archivePrefix={arXiv},
      primaryClass={cs.LG},
      url={https://arxiv.org/abs/1502.03167}, 
}

@article{paszke2017automatic,
  title={Automatic differentiation in PyTorch},
  author={Paszke, Adam and Gross, Sam and Chintala, Soumith and Chanan, Gregory and Yang, Edward and DeVito, Zachary and Lin, Zeming and Desmaison, Alban and Antiga, Luca and Lerer, Adam},
  year={2017},
  url={https://pytorch.org/}
}

@article{alldredge2019regularized,
  title={A regularized entropy-based moment method for kinetic equations},
  author={Alldredge, Graham W and Frank, Martin and Hauck, Cory D},
  journal={SIAM Journal on Applied Mathematics},
  volume={79},
  number={5},
  pages={1627--1653},
  year={2019},
  publisher={SIAM}
}

@book{dai2013approximation,
  title={Approximation theory and harmonic analysis on spheres and balls},
  author={Dai, Feng},
  year={2013},
  publisher={Springer}
}

@article{green2003,
title = {Spherical harmonic lighting: The gritty details},
author = {Green, Robin},
journal = {Game Developers Conference},
year = {2003}
}

@article{ivanic1996rotation,
  title={Rotation Matrices for Real Spherical Harmonics, Direct Determination by Recursion},
  author={Ivanic, Joseph and Ruedenberg, Klaus},
  journal={The Journal of Physical Chemistry A},
  volume={100},
  year={1996},
  pages={6342--6347}
}

@article{choi1999rapid,
  title={Rapid and stable determination of rotation matrices between spherical harmonics by direct recursion},
  author={Choi, Cheol Ho and others},
  journal={The Journal of Chemical Physics},
  volume={111},
  number={19},
  year={1999},
  pages={8825--8831}
}

@article{blanco1997evaluation,
  title={Evaluation of the rotation matrices in the basis of real spherical harmonics},
  author={Blanco, Miguel A and others},
  journal={Journal of Molecular Structure: THEOCHEM},
  volume={419},
  year={1997},
  pages={19--27}
}

@article{garrett2016eigenstructure,
  title={On the eigenstructure of spherical harmonic equations for radiative transport},
  author={Garrett, C Kristopher and Hauck, Cory D},
  journal={Computers \& Mathematics with Applications},
  volume={72},
  number={2},
  pages={264--270},
  year={2016},
  publisher={Elsevier}
}

@article{laboure2016implicit,
  title={Implicit filtered PN for high-energy density thermal radiation transport using discontinuous Galerkin finite elements},
  author={Laboure, Vincent M and McClarren, Ryan G and Hauck, Cory D},
  journal={Journal of Computational Physics},
  volume={321},
  pages={624--643},
  year={2016},
  publisher={Elsevier}
}

@article{laiu2016positive,
  title={Positive filtered p \_n moment closures for linear kinetic equations},
  author={Laiu, M Paul and Hauck, Cory D and McClarren, Ryan G and O'Leary, Dianne P and Tits, Andre L},
  journal={SIAM Journal on Numerical Analysis},
  volume={54},
  number={6},
  pages={3214--3238},
  year={2016},
  publisher={Society for Industrial and Applied Mathematics}
}

@article{laiu2019positivity,
  title={Positivity limiters for filtered spectral approximations of linear kinetic transport equations},
  author={Laiu, M Paul and Hauck, Cory D},
  journal={Journal of Scientific Computing},
  volume={78},
  pages={918--950},
  year={2019},
  publisher={Springer US}
}

@article{hauck2010positive,
  title={Positive P\_N closures},
  author={Hauck, Cory and McClarren, Ryan},
  journal={SIAM Journal on Scientific Computing},
  volume={32},
  number={5},
  pages={2603--2626},
  year={2010},
  publisher={SIAM}
}

@book{case1967linear,
  title={Linear transport theory},
  author={Case, Kenneth M and Zweifel, Paul F.},
  year={1967},
  publisher={Addison-Wesley}
}

@article{brunner2005two,
  title={Two-dimensional time dependent Riemann solvers for neutron transport},
  author={Brunner, Thomas A and Holloway, James Paul},
  journal={Journal of Computational Physics},
  volume={210},
  number={1},
  pages={386--399},
  year={2005},
  publisher={Elsevier}
}

@book{howell2020thermal,
  title={Thermal radiation heat transfer},
  author={Howell, John R and Meng{\"u}{\c{c}}, M Pinar and Daun, Kyle and Siegel, Robert},
  year={2020},
  publisher={CRC press}
}

@book{pomraning2005equations,
  title={The equations of radiation hydrodynamics},
  author={Pomraning, Gerald C},
  year={2005},
  publisher={Courier Corporation}
}

@article{mihalas1984radiation,
  title={Radiation hydrodynamics},
  author={Mihalas, Dimitri and Mihalas, BW},
  journal={Computational methods for astrophysical fluid flow. Saas-Fee advanced course},
  volume={27},
  pages={161--261},
  year={1984},
  publisher={Springer}
}

@techreport{bell1970nuclear,
  title={Nuclear reactor theory},
  author={Bell, George I and Glasstone, Samuel},
  year={1970},
  institution={US Atomic Energy Commission, Washington, DC (United States)}
}

@article{mezzacappa2020physical,
  title={Physical, numerical, and computational challenges of modeling neutrino transport in core-collapse supernovae},
  author={Mezzacappa, Anthony and Endeve, Eirik and Messer, OE Bronson and Bruenn, Stephen W},
  journal={Living Reviews in Computational Astrophysics},
  volume={6},
  pages={1--174},
  year={2020},
  publisher={Springer}
}

@article{lewis1984computational,
  title={Computational methods of neutron transport},
  author={Lewis, Elmer Eugene and Miller, Warren F},
  year={1984},
  publisher={John Wiley and Sons, Inc., New York, NY}
}

@article{Radice2013,
title = {A new spherical harmonics scheme for multi-dimensional radiation transport I. Static matter configurations},
journal = {Journal of Computational Physics},
volume = {242},
pages = {648-669},
year = {2013},
issn = {0021-9991},
doi = {https://doi.org/10.1016/j.jcp.2013.01.048},
url = {https://www.sciencedirect.com/science/article/pii/S0021999113001125},
author = {David Radice and Ernazar Abdikamalov and Luciano Rezzolla and Christian D. Ott}
}

@article{Oleary1993,
	Author = {Hansen, Per and O'leary, Dianne},
	Doi = {10.1137/0914086},
	Journal = {SIAM J. Sci. Comput.},
	Month = {11},
	Pages = {1487-1503},
	Title = {The Use of the L-Curve in the Regularization of Discrete Ill-Posed Problems},
	Volume = {14},
	Year = {1993}}

@article{Frank2016,
	Author = {M. Frank and C. Hauck and Kerstin K{\"u}pper},
	Journal = {Communications in Mathematical Sciences},
	Pages = {1443-1465},
	Title = {Convergence of filtered spherical harmonic equations for radiation transport},
	Volume = {14},
	Year = {2016}}

@techreport{Brunner2002,
	Address = {United States},
	Author = {Brunner, Thomas A. },
	Lk = {https://www.osti.gov/servlets/purl/800993-gsTYRE/native/},
	M1 = {SAND--2002-1778},
	Pages = {43},
	Title = {Forms of Approximate Radiation Transport},
	Ty = {RPRT},
	U1 = {NUCLEAR PHYSICS AND RADIATION PHYSICS},
	Url = {http://inis.iaea.org/search/search.aspx?orig_q=RN:34004332},
	Year = {2002}}

@article{Hauck2010,
	Author = {McClarren, Ryan G. and Hauck, Cory D.},
	Doi = {10.1016/j.jcp.2010.03.043},
	Issn = {0021-9991},
	Journal = {Journal of Computational Physics},
	Number = {16},
	Pages = {5597--5614},
	Publisher = {Elsevier BV},
	Title = {Robust and accurate filtered spherical harmonics expansions for radiative transfer},
	Url = {http://dx.doi.org/10.1016/j.jcp.2010.03.043},
	Volume = {229},
	Year = {2010}}

@book{Dautray2000,
  author    = {R. Dautray and J.-L. Lions},
  title     = {Mathematical Analysis and Numerical Methods for Science and Technology: Volume 6 Evolution Problems 2},
  year      = {2000},
  publisher = {Springer-Verlag},
  address   = {Berlin, Heidelberg}
}

@book{Gottlieb2007, place={Cambridge}, series={Cambridge Monographs on Applied and Computational Mathematics}, title={Spectral Methods for Time-Dependent Problems}, publisher={Cambridge University Press}, author={Hesthaven, Jan S. and Gottlieb, Sigal and Gottlieb, David}, year={2007}, collection={Cambridge Monographs on Applied and Computational Mathematics}}

@book{Abramowitz1972,
  title        = {Handbook of Mathematical Functions with Formulas, Graphs, and Mathematical Tables},
  author       = {Abramowitz, Milton and Stegun, Irene A.},
  publisher    = {Dover Publications},
  year         = {1972},
  address      = {New York},
  series       = {Applied Mathematics Series},
  volume       = {55}
}

@book{Hinze2008,
  title={Optimization with PDE Constraints},
  author={Hinze, M. and Pinnau, R. and Ulbrich, M. and Ulbrich, S.},
  isbn={9781402088391},
  lccn={2008934394},
  series={Mathematical Modelling: Theory and Applications},
  url={https://books.google.com/books?id=PFbqxa2uDS8C},
  year={2008},
  publisher={Springer Netherlands}
}

@book{Atkinson2012,
  title={Spherical Harmonics and Approximations on the Unit Sphere: An Introduction},
  author={Atkinson, K. and Han, W.},
  isbn={9783642259821},
  lccn={2012931330},
  series={Lecture Notes in Mathematics},
  url={https://books.google.com/books?id=SLQxcgt2Xj0C},
  year={2012},
  publisher={Springer}
}

@book{Leveque2002,
  title     = {Finite Volume Methods for Hyperbolic Problems},
  author    = {LeVeque, Randall J.},
  year      = {2002},
  publisher = {Cambridge University Press},
  address   = {Cambridge},
  series    = {Cambridge Texts in Applied Mathematics},
  isbn      = {9780521009249}
}

@book{Suli2003,
  title     = {An Introduction to Numerical Analysis},
  author    = {Süli, Endre and Mayers, David F.},
  year      = {2003},
  publisher = {Cambridge University Press},
  isbn      = {9780521007948}
}

@article{Hauck2008,
author = {Hauck, Cory and Lowrie, Robert},
year = {2008},
month = {09},
pages = {},
title = {Temporal regularization of the PN equations},
volume = {7},
journal = {Multiscale Modeling and Simulation},
doi = {10.1137/07071024X}
}

@article{Buet2004,
title = {Asymptotic preserving scheme and numerical methods for radiative hydrodynamic models},
journal = {Comptes Rendus Mathematique},
volume = {338},
number = {12},
pages = {951-956},
year = {2004},
issn = {1631-073X},
doi = {https://doi.org/10.1016/j.crma.2004.04.006},
url = {https://www.sciencedirect.com/science/article/pii/S1631073X04001918},
author = {Christophe Buet and Stephane Cordier}
}

@misc{Kingma2017,
      title={Adam: A Method for Stochastic Optimization}, 
      author={Diederik P. Kingma and Jimmy Ba},
      year={2017},
      eprint={1412.6980},
      archivePrefix={arXiv},
      primaryClass={cs.LG},
      url={https://arxiv.org/abs/1412.6980}, 
}

@phdthesis{plumridge2025,
  author = {Plumridge, Benjamin},
  title = {Filtered Angular Discretizations in Radiative Transfer},
  school = {University of Tennessee, Knoxville},
  year = {2025},
  type = {PhD Thesis},
}

\newpage
\appendix
\section{Flux matrices in the spherical harmonic equations}
\label{sec:flux_matrices}
The complex spherical harmonics satisfy the recursion relation \cite{brunner2005two}
\begin{equation}
\label{eq:recursion_complex}
\bm{\Omega} \overline{Y}_{\ell}^m = \frac{1}{2}i
\begin{bmatrix}
-c_{\ell-1}^{m-1} \overline{Y}_{\ell-1}^{m-1} + d_{\ell+1}^{m-1} \overline{Y}_{\ell+1}^{m-1} + e_{\ell-1}^{m+1} \overline{Y}_{\ell-1}^{m+1} - f_{\ell+1}^{m+1} \overline{Y}_{\ell+1}^{m+1} \\
i \left( c_{\ell-1}^{m-1} Y_{\ell-1}^{m-1} - d_{\ell+1}^{m-1} Y_{\ell+1}^{m-1} + e_{\ell-1}^{m+1} Y_{\ell-1}^{m+1} - f_{\ell+1}^{m+1} Y_{\ell+1}^{m+1} \right) \\
2(a_{\ell-1}^m \overline{Y}_{\ell-1}^m + r_{\ell+1}^m \overline{Y}_{\ell+1}^m)
\end{bmatrix}
\end{equation}
where
\begin{align}
a_{\ell}^m &= \sqrt{\frac{(\ell-m+1)(\ell+m+1)}{(2\ell+3)(2\ell+1)}}, & b_{\ell}^m &= \sqrt{\frac{(\ell-m)(\ell+m)}{(2\ell+1)(2\ell-1)}}, & c_{\ell}^m &= \sqrt{\frac{(\ell+m+1)(\ell+m+2)}{(2\ell+3)(2\ell+1)}}, \\
d_{\ell}^m &= \sqrt{\frac{(\ell-m)(\ell-m-1)}{(2\ell+1)(2\ell-1)}}, & e_{\ell}^m &= \sqrt{\frac{(\ell-m+1)(\ell-m+2)}{(2\ell+3)(2\ell+1)}}, & f_{\ell}^m &= \sqrt{\frac{(\ell+m)(\ell+m-1)}{(2\ell+1)(2\ell-1)}}.
\end{align}
When expressed in terms of the real spherical harmonics, the recursion relation \eqref{eq:recursion_complex} becomes \cite{Frank2016}

\begin{equation}
\label{eq:recursion_real}
\bm{\Omega} R_{\ell}^m = \frac{1}{2} \text{sgn}(m)
\begin{bmatrix}
(1-\delta_{\ell,1})(\tilde{c}_{\ell-1}^{|m|-1} R_{\ell-1}^{|m|-1} - \tilde{d}_{\ell-1}^{|m|-1} R_{\ell-1}^{|m|-1}) - \tilde{e}_{\ell-1}^{|m|+1} R_{\ell-1}^{|m|+1} + \tilde{f}_{\ell-1}^{|m|+1} R_{\ell-1}^{|m|+1} \\
(1-\delta_{\ell,1})(-\tilde{c}_{\ell-1}^{|m|-1} R_{\ell-1}^{-|m|-1} + \tilde{d}_{\ell-1}^{|m|-1} R_{\ell-1}^{-|m|-1}) - \tilde{e}_{\ell-1}^{|m|+1} R_{\ell-1}^{-|m|-1} + \tilde{f}_{\ell-1}^{|m|+1} R_{\ell-1}^{-|m|+1} \\
2(\tilde{a}_{\ell-1}^m R_{\ell-1}^m + \tilde{b}_{\ell+1}^m R_{\ell+1}^m)
\end{bmatrix}
\end{equation}
where $\delta_{i,j}$ denotes the Kronecker delta, and $\text{sgn}(m)$ denotes the sign function (with
abuse of notation in zero: $\text{sgn}(0) \equiv 1$). The coefficients are given by
\begin{align}
m^{+} &= m + \text{sgn}(m), & m^{-} &= m - \text{sgn}(m) \\
\tilde{c}_{\ell}^m &= \begin{cases}
0, & m < 0 \\
\sqrt{2}c_{\ell}^m, & m = 0 \\
c_{\ell}^m, & m > 0
\end{cases}, & \tilde{d}_{\ell}^m &= \begin{cases}
0, & m < 0 \\
\sqrt{2}d_{\ell}^m, & m = 0 \\
d_{\ell}^m, & m > 0
\end{cases} \\
\tilde{e}_{\ell}^m &= \begin{cases}
\sqrt{2}e_{\ell}^m, & m=1 \\
e_{\ell}^m, & m>1
\end{cases}, & \tilde{f}_{\ell}^m &= \begin{cases}
\sqrt{2}f_{\ell}^m, & m=1 \\
f_{\ell}^m, & m>1
\end{cases}
\end{align}
The recursion in \eqref{eq:recursion_real} enables a direct evaluation of the tensor $\bfA = \vint{\bf{\Omega} \bfr\bfr}^\top$, which is simplified by the normalization condition $\vint{R_\ell^m R_{\ell'}^{m'} }= \delta_{ell,\ell'} \delta_{m,m'}$
\section{Reduction to one and two dimensions}\label{app:reduction}

\subsection{Reduction to 1-D}
\label{subsec:reduction-1D}
In a slab geometry the spatial domain is the region between two infinite parallel plates $\{z=z\textsubscript{L}\}$ and $\{z=z\textsubscript{R}\}$; that is, $X = \bbR^2 \times (z_L,z_R)$. 
With appropriate initial and boundary conditions, $\psi$ depends only on $z$ and the angular variable $\mu = \cos \theta $.

As a result, the moments $u_\ell^m$ are identically zero for $m \neq 0$.  In this setting we use moments with respect to the Legendre polynomials normalized on the interval $[-1,1]$:
\footnote{This normalization differs from the standard normalization in which $\int_{-1}^1 |P_\ell(\mu)|^2 d \mu = \frac{2 \ell+1}{2}$.}
\begin{equation}
    \tilde{P}_\ell(\mu) =  \sqrt{2 \pi}r_\ell^0(\mu). 
\end{equation}
Let $\tilde{\bfp} : [-1,1] \rightarrow \bbR^{N+1}$ be a vector-valued function of all orthonormal Legendre polynomials up to degree $N$ such that $\tilde{\bfp} (\mu) = [ \tilde{P}_0(\mu), \tilde{P}_1(\mu), \dots, \tilde{P}_N(\mu)]^\top$.
The \fpn solution in 1-D slab geometry is described by $\psifpn(z,\mu) = \tilde{\bfp} ^\top(\mu) \bfv(z,t)$ where

\begin{equation}
    \p_t \bfv + \tilde{\bfA} \p_z \bfv + \sig{a} \bfv + \sig{s} \tilde{\bfG} \bfv + \sig{f} \tilde{\bfF} \bfv = \tilde{\bfs},
\end{equation}
and $\tilde{\bfs}$ is a 1-D source.
The \fpn matrices are given by 
\begin{align}
   \tilde{\mathbf{A}}_{\ell,\ell'} &= a_\ell \, \delta_{\ell',\ell+1} + a_{\ell'} \, \delta_{\ell,\ell'+1}
    & \text{where} &&
    a_\ell &= \frac{\ell+1}{\sqrt{(2\ell+1)(2\ell+3)}},\\
   \tilde{\bfG}_{\ell,\ell'} &= (1-g_\ell)\delta_{\ell,\ell'},
    & \text{and} &&
    \tilde{\bfF}_{\ell,\ell'} ,
    &= -\log \left(f \left( \frac{\ell}{N+1}\right) \right) \delta_{\ell,\ell'}. 
\end{align}

\subsection{Reduction to 2-D}
\label{subsec:reduction-2D}
In two-dimensional planar geometry, $X = D \times \bbR$, where $D \subset \bbR^2$.  Under appropriate boundary and initial conditions, $\psi$ is independent of $z$ and an even function of $\mu$.  As a result, the fact that $r_\ell^m$ is an odd function of $\mu$ whenever $\ell + m$ is odd implies that moments the $u_\ell^m$ are identically zero whenever $\ell+m$ is odd.  In such cases, the \pn equations take the form
\begin{equation}
    \p_t \bfv 
    + \tilde{\bfA}^{(1)} \p_x \bfv 
    +  \tilde{\bfA}^{(2)} \p_y \bfv  
    + \sig{a}\bfv 
    + \sig{s} \tilde{\bfG} \bfv 
    + \sig{f} \tilde{\bfF} \bfv 
    = \tilde{\bfs}
\end{equation}
Here $\bfv$ contains the $n_2 = \frac12 (N+1)(N+2)$ nonzero components of $\upn$, $\tilde{\bfs}$ contains the corresponding components of $\bfs$ and the matrices $\tilde{\bfA}^{(1)}$, $\tilde{\bfA}^{(2)}$, $\tilde{\bfG}$, and $\tilde{\bfF}$, are formed by removing the appropriate rows and columns of $\bfA^{(1)}$, $\bfA^{(2)}$, ${\bfG}$, and ${\bfF}$, respectively.

\subsection{Space time discretization}\label{sec:space-time-disc}

Now we discuss the spatial and temporal discretizations.  We show the spatial discretization used in our 2-D tests. The 1-D discretization follows similarly.  For ease of notation, we write $\bfv:= \ufpn$.

For simplicity, we assume that $D$ is rectangle with side lengths $L_x$ and $L_y$ and is divided into $N_x \times N_y$ cells with centers $\bm{x}_{i,j} := (x_i,y_j)$ and dimension $\dx = L/N_x$ and $\dy = L/N_y$. Let  $C_{i,j} = (x_{i-1/2}, x_{i+1/2}) \times (y_{j-1/2}, y_{j+1/2})$ be the cell centered at $\bm{x}_{i,j}$ where $i \in \{0 \ddd N_x -1\}$ and $j \in \{0 \ddd N_y -1\}$ are spatial indices.  The physical cross-sections, source, and filter strength are assumed to be constant on each cell.

Since the matrices $\tilde{\bfA}^{(m)}$ are symmetric and diagonalizable, they can be decomposed as 

\begin{align*}
    \tilde{\bfA}^{(m)} =  \bfV^{(m)} \bsLambda^{(m)} 
(\bfV^{(m)})^\top
\end{align*}
where $\bsLambda^{(m)} \in  
\bbR^{n_2 \times n_2}$ is diagonal and $\bfV^{(m)} \in 
\bbR^{n_2 \times n_2}$ is orthogonal.
With this decomposition, let 

\begin{align*}
    |\tilde{\bfA}^{(m)}| = \bfV^{(m)} |\bsLambda^{(m)}| (\bfV^{(m)})^\top.
\end{align*}
Then the semi-discrete, second-order finite volume scheme \cite{Leveque2002} can be described as follows: Find 
\begin{equation}
    {\bfv}_{i,j}(t) 
    \approx \frac{1}{\dx \dy} \int_{C_{i,j}} \bfv(x,y,t) dx dy
\end{equation}
such that 
\begin{align}
  \begin{aligned}\label{eq:fv}
    \p_t \bfv_{i,j} 
    & + \frac{1}{2\dx} \tilde{\bfA}^{(1)} \left[ 
    (\bfv^-_{i+1/2,j} + \bfv^+_{i+1/2,j})
    - (\bfv^-_{i-1/2,j} + \bfv^+_{i-1/2,j})\right] 
    \\
    & -  \frac{1}{2\dx} |\tilde{\bfA}^{(1)}| \left[ {\bfv}^+_{i+1/2,j} - (\bfv^-_{i+1/2,j} + \bfv^+_{i-1/2,j} ) + \bfv^-_{i-i/2,j} \right] \\
    & + \frac{1}{2\dy} \tilde{\bfA}^{(2)} \left[ 
    (\bfv^-_{i,j+1/2} + \bfv^+_{i,j+1/2})
    - (\bfv^-_{i,j-1/2} + \bfv^+_{i,j-1/2}) \right] \\
    &-  \frac{1}{2\dy} |\tilde{\bfA}^{(2)}|\left[
    {\bfv}^+_{i,j+1/2} - (\bfv^-_{i,j+1/2} + \bfv^+_{i,j-1/2} ) + \bfv^-_{i,j-i/2})\right] \\
    &  + (\sig{a})_{i,j}\bfv_{i,j}  
    + (\sig{s})_{i,j} \tilde{\bfG} \bfv_{i,j} 
    + (\sig{f})_{i,j} \tilde{\bfF}  \bfv_{i,j} = \tilde{\bfs}_{i,j}
    \end{aligned}
\end{align}
where,
\begin{align}
    \bfv_{i + 1/2,j}^- =  \bfv_{i,j} + \frac{\dx}{2} \bfp^x_{i,j}, &   \qquad 
    \bfv_{i + 1/2,j}^+ =  \bfv_{i+1,j} - \frac{\dx}{2} \bfp^x_{i+1,j},\\
    \bfv_{i,j + 1/2}^- =  \bfv_{i,j} + \frac{\dy}{2} \bfp^y_{i,j}, &   \qquad 
    \bfv_{i,j + 1/2}^- =  \bfv_{i,j+1} - \frac{\dy}{2} \bfp^y_{i,j+1}
\end{align}
and the slope limiters are defined as 
\begin{subequations}
\begin{align}
\bfp^x_{i,j}  
= \operatorname{minmod} \left\{ \bfv_{i+1,j}  - \bfv_{i,j}, \bfv_{i,j} - \bfv_{i-1,j} \right\}, && \bfp^y_{i,j}  
= \operatorname{minmod} \left\{ \bfv_{i,j+1}  - \bfv_{i,j}, \bfv_{i,j} - \bfv_{i,j-1} \right\}
\end{align}
\end{subequations}
where
\begin{align}
\text{minmod}(a,b) := \begin{cases}
0, & ab < 0 \\
\text{sgn}(a) \min \left(|a|,|b|  \right), & ab>0.
\end{cases}
\end{align}
The approximation source is given by $ \tilde{\bfs}_{i,j}(t) = \int_{C_{i,j}} \tilde{\bfs}(x,y,t) dx dy$ and the cross-sections are given by $(\sig{a})_{i,j} := \sig{a}(\bm{x}_{i,j})$, $(\sig{s})_{i,j} := \sig{s}(\bm{x}_{i,j})$, and $(\sig{f})_{i,j}(t)  = \Sig{f}\big(\bm{\xi}(\bfv_{i,j}(t),\tilde{\bfs}_{i,j}(t);\bm{w} \big)$.

We collect the cell averages of the \fpn moments into an array $\vec{\bfv}$ and express \eqref{eq:fv} in matrix-vector form as 

\begin{align}\label{eq:fv_vec}
    \p_t \vec{\bfv} + \bfL(\Sig{f}(\bm{\xi}(\vec{\bfv}),\vec{\bfs};\bm{w} ) \vec{\bfv} = \vec{\bfs}
\end{align}
where $\bfL$ is the matrix representation of the spatially discretized operators that depends on $\vec{\bfv}, \vec{\bfs}$, and $\bm{w}$ via the filter strength ansatz $\Sig{f}$, and $\bm{\xi}$ are $\Sig{f}$ act component-wise with respect to the mesh indices $i$ and $j$.

For the temporal discretization, we employ Heun's method \cite{Suli2003}, an explicit second-order Runge-Kutta scheme.
One time step of Heun's method can be written as 

\begin{subequations}
\label{eq:Heun}
\begin{align}
\vec{\bfv}^{(1)} &= \vec{\bfv}^k - \Delta t \bfL(\Sig{f}(\bm{\xi}(\vec{\bfv}^k),\vec{\bfs}^k;\bm{w} ) \vec{\bfv}^k  + \Delta t \bfs^k, \\
\vec{\bfv}^{(2)} &= \vec{\bfv}^{(1)} - \Delta t \bfL(\Sig{f}(\bm{\xi}(\vec{\bfv}^{(1)}),\vec{\bfs}^{k+1};\bm{w} )
\vec{\bfv}^{(1)} + \Delta t \bfs^{k+1}, \\
\vec{\bfv}^{k+1} &= \frac{1}{2}\left(\vec{\bfv}^{(1)} + \vec{\bfv}^{(2)} \right), & 
\end{align}
\end{subequations}
where $k \in \{0, \dots, K -1\}$ is the temporal index.  We combine these formula into an update rule for $\vec{\bfv}^{k+1} = \vec{\bfv}^{k+1}(\bm{w})$:
\begin{equation}
    \vec{\bfv}^{k+1}(\bm{w}) = \bfH^k(\vec{\bfv}^{k}; \bm{w}),
    \quad k \in \{0, \dots, K -1\}
\end{equation}
where the index of $\bfH^k$ accounts for the sources $\vec{\bfs}^k$ and $\vec{\bfs}^{k+1}$.  This update rule is used to compute derivative of $\vec{\bfv}^{k+1}$ recursively,
\begin{equation}
    \frac{\p \vec{\bfv}^{k+1}}{\p \bm{w}} (\bm{w})
    = \frac{\p \bfH^k}{\p \bm{w}} (\vec{\bfv}^{k}; \bm{w})
    + \frac{\p \bfH^k}{\p \vec{\bfv}^k} (\vec{\bfv}^{k})
    \frac{\p \vec{\bfv}^k}{\p \bm{w}} (\vec{\bfv}^k, \bm{w}),
     \quad k \in \{0, \dots, K -1\}.
\end{equation}
The derivatives of $\bfH^k$ are straight-forward, yet tedious, calculations that apply the chain rule to \eqref{eq:Heun}.  These calculations involve state and parameter derivatives of $\Sig{f}$ which are implemented via backpropagation through the neural network ansatz.

\end{document}